\documentclass[11pt]{amsart}
\usepackage{amssymb}
\newtheorem{theorem}{Theorem}[section]
\newtheorem{lemma}[theorem]{Lemma}

\newtheorem{definition}[theorem]{Definition}

\newtheorem{corollary}[theorem]{Corollary}
\newtheorem{proposition}[theorem]{Proposition}

\newtheorem{lem-def}[theorem]{Lemma-Definition}
\def\P{\mathbb P}
\newcommand{\R}{\mathbb R}
\newcommand{\M}{\mathbb M}

\newcommand{\Z}{\mathbb Z}
\newcommand{\Q}{\mathbb Q}

\newcommand{\F}{\mathbb F}

\newcommand{\V}{\mathbb V}

\def\op{\operatorname}

\def\as#1{\renewcommand\arraystretch{#1}}

\def\cc{{\mathcal C}}

\def\diso{\lower.4ex\hbox{$\downarrow$}\raise.4ex\hbox{\mbox{\scriptsize
$\wr$}}}

\def\ep#1{\exp(\Pi i#1)}
\def\ep{\epsilon}

\def\ga{\gamma}

\def\gg{\mathcal{G}r}
\def\ggm{\mathcal{G}r(\mu)}
\def\ggmp{\mathcal{G}r(\mu')}

\def\gr{\operatorname{gr}_{\mu}K[x]}

\def\hh{{\mathcal H}}

\def\hm{H_\mu}
\def\hmp{H_{\mu'}}

\def\imp{\,\Longrightarrow\,}

\def\iso{\ \lower.3ex\hbox{\as{.08}$\begin{array}{c}\lra\\\mbox{\tiny $\sim\,$}\end{array}$}\ }

\def\kb{\overline{K}_v}
\def\kp{\op{KP}}
\def\kpm{\op{KP}(\mu)}
\def\kpp{\op{KP}(\mu)^{\op{pr}}}
\def\kps{\op{KP}(\mu)^{\op{str}}}

\def\la{\lambda}
\def\lg{l\raise.6ex\hbox to.2em{\hss.\hss}l}
\def\ll{\mathcal{L}}

\def\lra{\longrightarrow}
\def\m{{\mathfrak m}}
\def\md#1{\; \mbox{\rm(mod }{#1})}

\def\mmu{\mid_\mu}
\def\mx{\op{Max}}

\def\nph{N_{\mu,\phi}}

\def\oo{\mathcal{O}}
\def\orb{\hbox to  .3em{$\backslash$}\backslash}
\def\ord{\op{ord}}
\def\p{\mathfrak{p}}

\def\ppa{\mathcal{P}_{\alpha}}
\def\pset{\mathcal{P}}

\def\rb{\overline{R}}

\def\res{\operatorname{Res}}

\def\rr{\mathcal{R}}
\def\rrm{\mathcal{R}_\mu}

\def\smu{\sim_\mu}

\def\sii{\,\Longleftrightarrow\,}

\def\t{\theta}

\def\Vi{\V^{\op{ind}}}

\newcounter{cs}
\stepcounter{cs}
\newcommand{\casos}{\begin{itemize}}
\newcommand{\fcasos}{\end{itemize}\setcounter{cs}{1}}

\newfont{\tit}{cmr12 scaled \magstep3}

\title[Residual ideals of MacLane valuations]{Residual ideals of MacLane valuations}

\author[Fern\'andez]{Julio Fern\'andez}
\address{Departament de Matem\`atica Aplicada IV, Escola Polit\`ecnica Superior d'Enginye\-ria de Vilanova i la Geltr\'u, Av. V\'\i ctor Balaguer s/n. E-08800 Vilanova i la Geltr\'u, Catalonia, Spain}
\email{julio@ma4.upc.edu, guardia@ma4.upc.edu}

\author[Gu\`ardia]{Jordi Gu\`ardia}

\author[Montes]{Jes\'us Montes}
\address{Departament de Ci\`encies Econ\`omiques i Empresarials,
Facultat de Ci\`encies Socials,
Universitat Abat Oliba CEU,
Bellesguard 30, E-08022 Barcelona, Catalonia, Spain\\
Departament de Matem\`atica Econ\`omica, Financera i Actuarial,
Facultat d'Economia i Empresa,
Universitat de Barcelona,
Av. Diagonal 690,
E-08034 Barcelona, Catalonia, Spain}
\email{montes3@uao.es, jesus.montes@ub.edu}

\author[Nart]{Enric Nart}
\address{Departament de Matem\`{a}tiques,
         Universitat Aut\`{o}noma de Barcelona,
         Edifici C, E-08193 Bellaterra, Barcelona, Catalonia, Spain}
\email{nart@mat.uab.cat}
\thanks{Partially supported by MTM2012-34611 and MTM2009-10359 from the Spanish MEC}
\date{}

\makeatletter
\@namedef{subjclassname@2010}{%
  \textup{2010} Mathematics Subject Classification}

\begin{document}

\begin{abstract}
Let $K$ be a field equipped with a discrete valuation $v$. In a pioneering work, MacLane determined all valuations on $K(x)$ extending $v$. His work was recently reviewed and generalized by M. Vaqui\'e, by using the graded algebra of a valuation. We extend Vaqui\'e's approach by studying residual ideals of the graded algebra as an abstract counterpart of certain residual polynomials which play a key role in the computational applications of the theory. As a consequence, we determine the structure of the graded algebra of the discrete valuations on $K(x)$ and we show how these valuations may be used to parameterize irreducible polynomials over local fields up to Okutsu equivalence.    
\end{abstract}

\maketitle
\section*{Introduction}
In a pioneering work, MacLane linked in 1936 the theory of discrete valuations on a field of rational functions in one variable with the study of irreducible polynomials over local fields. Several authors have proposed since then different approaches to either of these questions. In this paper, we show that MacLane's original approach, combined with some ideas of Montes and Vaqui\'e, provides a unified insight for the main developments of these topics. Before describing the contents of the paper in more detail, let us briefly recall some milestones in these developments.\medskip

\noindent{\bf MacLane's solution to a problem raised by Ore.}
In the 1920's, Ore developed a method to construct the prime ideals of a number field, dividing a given prime number $p$, in terms of a defining polynomial $f\in\Z[x]$ satisfying a certain \emph{$p$-regularity} condition \cite{ore1, ore2}. The idea was to detect a $p$-adic factorization of $f$ according to the different irreducible factors of certain residual polynomials over finite fields, attached to the sides of a Newton polygon of $f$. 
He raised then the question of the existence of a procedure to compute the prime ideals in the $p$-irregular case, based on the consideration of similar Newton polygons and residual polynomials of \emph{higher order}. 

MacLane solved this problem in 1936 in a more general context \cite{mcla,mclb}. For any discrete valuation $v$ on an arbitrary field $K$, he described all discrete valuations extending $v$ to the rational function field $K(x)$. 
Then, given an irreducible polynomial $f\in K[x]$, he characterized all extensions of $v$ to the field $L:=K[x]/(f)$ as limits of infinite families of valuations on $K(x)$ whose value at $f$ grows to infinity. Finally, he gave a criterion to decide when a valuation on $K(x)$ is sufficiently close to a valuation on $L$ to uniquely represent it.

There is a natural extension $\mu_0$ of $v$ to $K(x)$ satisfying $\mu_0(x)=0$. Starting from $\mu_0$, MacLane constructed \emph{inductive} valuations $\mu$ on $K(x)$ extending $v$, by the concatenation of augmentation steps
$$\mu_0\ \stackrel{(\phi_1,\la_1)}\lra\  \mu_1\ \stackrel{(\phi_2,\la_2)}\lra\ \cdots\ \stackrel{(\phi_{r-1},\la_{r-1})}\lra\ \mu_{r-1} \ \stackrel{(\phi_{r},\la_{r})}\lra\ \mu_{r}=\mu,$$ 
based on the choice of certain \emph{key polynomials} $\phi_i\in K[x]$ and arbitrary positive rational numbers $\la_i$. In the case $K=\Q$, Ore's $p$-regularity condition is satisfied when all valuations on $L$ extending the $p$-adic valuation are sufficiently close to valuations on $K(x)$ that may be obtained from $\mu_0$ by a single augmentation step. 

After MacLane's work, inductive valuations were rediscovered and extensively studied as \emph{residually transcendental extensions} of $v$ to $K(x)$ \cite{APZ, PP, PV}. In this approach, the valuations are first analyzed for algebraically closed fields, where they may be obtained as a simple augmentation of $\mu_0$ with respect to a key polynomial of degree one. The general case is then deduced by descent. 

Let $K_v$ be the completion of $K$ at $v$, and let $\oo_v\subset K_v$ be the valuation ring of $K_v$. We denote by $\P$ the set of all monic irreducible polynomials in $\oo_v[x]$. An $F\in\P$ is called a \emph{prime polynomial} with respect to $v$.  \medskip

\noindent{\bf Okutsu equivalence of prime polynomials.}
For $v$ a discrete valuation on a global field $K$ and $F$ a prime polynomial, Okutsu constructed in 1982 an explicit integral basis of the local field $K_F=K_v[x]/(F)$, in terms of a finite sequence of prime polynomials $\phi_1,\dots,\phi_r$ which are a kind of optimal approximations to $F$ with respect to their degree \cite{Ok}. 
Such a family $[\phi_1,\dots,\phi_r]$ is called an \emph{Okutsu frame} of $F$. The polynomials $\phi_i$ support certain numerical data, the so-called \emph{Okutsu invariants} of $F$, containing considerable information about $F$ and the field $K_F$. 

An equivalence relation $\approx$ on the set $\P$ is defined as follows: two prime polynomials $F,G\in\P$ of the same degree are said to be \emph{Okutsu equivalent} if 
$$
v(\res(G,F))/\deg G>v(\res(\phi_r,F))/\deg \phi_r.
$$ 
In this case, $F$ and $G$ have the same Okutsu invariants, and the fields $K_F$, $K_G$ have isomorphic maximal tamely ramified subextensions \cite{okutsu}.

In 1999, Montes carried out Ore's program in its original formulation \cite{Mo,HN}. Given a finite extension $L/K$ of number fields determined by an irreducible polynomial $f\in K[x]$, and given a prime ideal $\p$ of $K$, Montes constructed the prime ideals of $L$ lying over $\p$  by finding polynomials in $K[x]$ which are Okutsu equivalent to the irreducible factors of $f$ in $K_v[x]$, where $v=v_\p$ is the $\p$-adic valuation. The method computes as well Okutsu frames and the Okutsu invariants of each prime factor.
In this setting, the use of MacLane's valuations and Newton polygon operators is complemented with the introduction of residual polynomial operators $R_i\colon K[x]\lra \F_i[y]$, 
where $i\ge 0$ is the ``order" of a valuation, $\F_i$ is a certain finite field, and $x,y$ are indeterminates. These operators make the whole theory constructive and well-suited to computational applications. 
These ideas led to the design of several fast algorithms to perform arithmetic tasks in global fields \cite{Bauch,algorithm,newapp,bases,GNP,Ndiff}.\medskip

\noindent{\bf Contents of this paper.}
In 2007, Vaqui\'e reviewed and generalized MacLane's work. For an arbitrary field $K$, he determined all valuations on $K(x)$ extending an arbitrary valuation $v$ on $K$ \cite{Vaq}. The use of the graded algebra $\ggm$ of a valuation $\mu$, restricted to the polynomial ring $K[x]$, led Vaqui\'e to a more elegant presentation of the theory.  

In this paper, restricted to the discrete case, we have a double aim. On one hand, we extend Vaqui\'e's approach by including a treatment of the residual polynomial operators atached to a discrete valuation $\mu$ over an arbitrary field $K$. The residual polynomials are interpreted as generators of \emph{residual ideals} in the degree-zero subring $\Delta(\mu)$ of the graded algebra $\ggm$. The residual ideal of a polynomial $g\in K[x]$ is defined as $\rr_\mu(g)=H_\mu(g)\ggm\cap\Delta(\mu)$, where $\hm(g)$ is the natural image of $g$ in the piece of degree $\mu(g)$ of the algebra. In sections 1-5, we review the properties of MacLane's inductive valuations, while making apparent the key role of the residual ideals in the whole theory. As an application of this point of view, we determine the structure of $\ggm$ as a graded algebra (Theorem \ref{structure}).    

Our second aim is to show that this approach leads to a natural generalization of the results of Okutsu and Montes to arbitrary discrete valued fields. A prime polynomial $F\in\P$
induces a pseudo-valuation $\mu_{\infty,F}$ on $K[x]$ via the composition
$$
\mu_{\infty,F}\colon K[x]\lra K_F\stackrel v\lra \Q\cup\{\infty\},
$$
where we denote again by $v$ the unique extension of $v$ to $K_F$. According to MacLane's insight, approximating $F$ by polynomials in $K[x]$ is equivalent to approximating $\mu_{\infty,F}$ by valuations on $K(x)$. In section \ref{secOkutsu}, we introduce a canonical inductive valuation $\mu_F$ which is a threshold valuation in this approximation process, and  
we reproduce most of the fundamental results of \cite{okutsu,HN,Mo,Ok,PZ}
with much shorter proofs. An Okutsu frame of $F$ is seen to be just a family of key polynomials of an \emph{optimal} chain of inductive valuations linking $\mu_0$ with $\mu_F$ (Theorems \ref{MLOk}, \ref{OkML}), and the Okutsu invariants of $F$ are essentially the MacLane invariants of these valuations, introduced in section \ref{secInductive}. 

Finally, in section \ref{sectionLimit} we briefly recall MacLane's results on limits of inductive valuations. We analyze in detail the interval $[\mu_0,\mu_{\infty,F})$ of all valuations $\mu$ on $K(x)$ such that $\mu(g)\le\mu_{\infty,F}(g)$ for all $g\in K[x]$. In Theorem \ref{totally} we prove that this interval is totally ordered and give an explicit description of all the valuations therein.

The main result of the paper is Theorem \ref{MLspace}, where we establish a canonical bijection between the set $\,\P/\!\approx\,$ of Okutsu equivalence classes of prime polynomials and the \emph{MacLane space} $\M$ of $(K,v)$, defined as the set of all pairs $(\mu,\ll)$, where $\mu$ is an inductive valuation on $K(x)$ and $\ll$ is a \emph{strong} maximal ideal of $\Delta(\mu)$. The bijection sends the class of $F$ to the pair 
$(\mu_F,\rr_{\mu_F}(F))$. This result reveals that MacLane's original approach is best-suited for computational applications, because the elements in the set $\M$ may be described in terms of discrete parameters which are easily manipulated by a computer. As a consequence, all algorithmic developments based on the Montes algorithm \cite{algorithm,newapp,bases,GNP,Ndiff} admit a more elegant description and a natural extension to arbitrary discrete valued fields.
However, we postpone the discussion of these computational aspects to a forthcoming paper \cite{gen}.


\section{Augmentation of valuations}\label{secMacLane}
Let $K$ be a field equipped with a discrete valuation $v\colon K^*\lra \Z$, normalized so that $v(K^*)=\Z$. Let $\oo$ be the valuation ring of $K$, $\m$ the maximal ideal, $\pi\in\m$ a generator of $\m$ and $\F=\oo/\m$ the residue class field. 

Let $K_v$ be the completion of $K$ and denote again by $v\colon \kb\to \Q\cup\{\infty\}$ the canonical extension of $v$ to a fixed algebraic closure of $K_v$. Let $\oo_v$ be the valuation ring of $K_v$, $\m_v$ its maximal ideal and $\F_v=\oo_v/\m_v$ the residue class field. The canonical inclusion $K\subset K_v$ restricts to  inclusions $\oo\subset\oo_v$, $\m\subset\m_v$, which determine a canonical isomorphism $\F\simeq\F_v$.  
We shall consider this isomorphism as an identity, $\F=\F_v$, and we indicate simply with a bar, $\raise.8ex\hbox{---}\colon \oo_v[x]\longrightarrow \F[x]$,
the canonical homomorphism of reduction of polynomials modulo $\m_v$. 

Our aim is to describe all extensions of $v$ to discrete valuations on the field $K(x)$, where $x$ is an indeterminate. 

\begin{definition}
Let $\V$ be the set of discrete valuations, $\mu\colon K(x)^*\lra \Q$, such that $\mu_{\mid K}=v$ and $\mu(x)\ge0$.
For any $\mu\in\V$, denote

\begin{itemize}
\item $\Gamma(\mu)=\mu\left(K(x)^*\right)\subset \Q$, the cyclic group of finite values of $\mu$. The \emph{ramification index} of $\mu$ is the positive integer $e(\mu)$ such that $e(\mu)\Gamma(\mu)=\Z$.   
\item $\kappa(\mu)$, the residue class field of $\mu$.
\item $\kappa(\mu)^{\op{alg}}\subset \kappa(\mu)$, the algebraic closure of $\F$ inside $\kappa(\mu)$.
\end{itemize}


From now on, the elements of $\V$ will be simply called \emph{valuations}.
\end{definition}

Since we are only interested in (rank one) discrete valuations, we may assume that our valuations are $\Q$-valued. On the other hand, the assumption $\mu(x)\ge0$ is not essential; it   
gives a more compact form to the presentation of the results. For the determination of the discrete valuations with $\mu(x)<0$ one may simply replace $x$ by $1/x$ as a generator of the field $K(x)$ over $K$.

In the set $\V$ there is a natural partial ordering:
$$
\mu\le \mu' \quad\mbox{ if }\quad\mu(g) \le \mu'(g), \ \forall\,g\in K[x]. 
$$

We denote by $\mu_0\in \V$ the valuation which acts on polynomials as
$$
\mu_0\left(\sum\nolimits_{0\le s}a_sx^s\right)=\min_{0\le s}\left\{v(a_s)\right\}.
$$
Clearly, $\mu_0\le \mu$ for all $\mu\in\V$; in other words, $\mu_0$ is the minimum element in $\V$.

\subsection{Graded algebra of a valuation}
Let $\mu\in\V$ be a valuation. For any $\alpha\in\Gamma(\mu)$ we consider the following $\oo$-submodules in $K[x]$:
$$
\ppa=\ppa(\mu)=\{g\in K[x]\mid \mu(g)\ge \alpha\}\supset
\ppa^+=\ppa^+(\mu)=\{g\in K[x]\mid \mu(g)> \alpha\}.
$$    
Clearly, $\pset_0$ is a subring of $K[x]$, and $\ppa$, $\ppa^+$
are $\pset_0$-submodules of $K[x]$ for all $\alpha$.

The \emph{graded algebra of $\mu$} is the integral domain:
$$
\ggm:=\gr:=\bigoplus\nolimits_{\alpha\in\Gamma(\mu)}\ppa/\ppa^+.
$$

Let $\;\Delta(\mu)=\pset_0/\pset_0^+$ be the subring determined by the piece of degree zero of this algebra. Clearly, $\oo\subset\pset_0$ and $\m=\pset_0^+\cap \oo$; thus, there is a canonical homomorphism $\F\lra\Delta(\mu)$, equipping  $\Delta(\mu)$ (and $\ggm$) with a canonical structure of $\F$-algebra. 

Let $A\subset K(x)$ be the valuation ring of $\mu$ and $\m_A$ its maximal ideal. Since $\pset_0=K[x]\cap A$ and $\pset_0^+=K[x]\cap \m_A$, we have an embedding
$\Delta(\mu)\hookrightarrow\kappa(\mu)$. We shall see along the paper that this embedding identifies $\kappa(\mu)$ with the field of fractions of $\Delta(\mu)$.  

There is a natural map $\hm\colon K[x]\lra \ggm$, given by $\hm(0)=0$, and
$$\hm(g)= g+\pset_{\mu(g)}^+\in\pset_{\mu(g)}/\pset_{\mu(g)}^+,$$
for $g\ne0$. Note that $\hm(g)=0$ if and only if $g=0$. For all $g,h\in K[x]$ we have:
\begin{equation}\label{Hmu}
\as{1.2}
\begin{array}{l}
 \hm(gh)=\hm(g)\hm(h), \\
 \hm(g+h)=\hm(g)+\hm(h), \mbox{ if }\mu(g)=\mu(h)=\mu(g+h).
\end{array}
\end{equation}

If $\mu\le \mu'$, a canonical homomorphism of graded algebras $\ggm\to\gg(\mu')$ is determined by $g+\ppa^+(\mu)\mapsto g+\ppa^+(\mu')$ for all $g,\alpha$. The image of $\hm(g)$ is $\hmp(g)$ if $\mu(g)=\mu'(g)$, and zero otherwise. 

\begin{definition}\label{mu}\mbox{\null}

\noindent$\bullet$ We say that $g,h\in K[x]$ are \emph{$\mu$-equivalent}, and we write $g\smu h$, if $\hm(g)=\hm(h)$. Thus, $g\smu h$ if and only if $\mu(g-h)>\mu(g)=\mu(h)$ or $g=h=0$. \medskip

\noindent$\bullet$ We say that $g$ is \emph{$\mu$-divisible} by $h$, and we write $h\mmu g$, if $\hm(g)$ is divisible by $\hm(h)$ in $\ggm$. Thus, $h\mmu g$ if and only if $g\smu hf$ for some $f\in K[x]$.
\medskip

\noindent$\bullet$ We say that $\phi\in K[x]$ is $\mu$-irreducible if $\hm(\phi)\ggm$ is a non-zero prime ideal in $\ggm$.\medskip

\noindent$\bullet$ We say that $\phi\in K[x]$ is $\mu$-minimal if $\deg \phi>0$ and $\phi\nmid_\mu g$ for any non-zero $g\in K[x]$ with $\deg g<\deg \phi$.

\end{definition}

\begin{lemma}\label{minimal0}
Let $\phi\in K[x]$ be a polynomial of positive degree. For any $g\in K[x]$, let $g=\sum_{0\le s}g_s\phi^s$, $\deg g_s<\deg \phi$, be its canonical $\phi$-expansion.
The following conditions are equivalent:
\begin{enumerate}
\item $\phi$ is $\mu$-minimal
\item For any $g\in K[x]$, $\mu(g)=\min\{\mu(g_0),\mu(g-g_0)\}$.
\item For any $g\in K[x]$, $\mu(g)=\min_{0\le s}\{\mu(g_s\phi^s)\}$. 
\item For any nonzero $g\in K[x]$, $\phi\nmid_\mu g$ if and only  if $\mu(g)=\mu(g_0)$.
\end{enumerate}
\end{lemma}

\begin{proof}  

Let $g-g_0=\phi q$. If $\mu(g)>\mu(g_0)$, or $\mu(g)>\mu(\phi q)$, then $g_0\smu-\phi q$, and $\phi\mmu g_0$. Hence, (1) implies (2). 

Clearly, (2) implies (3).
Let us show that (3) implies (4). For a non-zero polynomial $g$, (3) implies $\mu(g)\le\mu(g_0)$. If $\mu(g)<\mu(g_0)$, then $g\smu \sum_{0<s}g_s\phi^s$, and $\phi\mmu g$. Conversely, if $g\smu \phi q$ for some $q\in K[x]$, then $g_0$ is the $0$-th coefficient of the $\phi$-expansion of $g-\phi q$, and (3) implies that
$\mu(g)<\mu(g-\phi q)\le \mu(g_0)$.

Finally, (4) implies (1) because $g=g_0$ if $\deg g<\deg \phi$, and (4) implies $\phi\nmid_\mu g$. 
\end{proof}

The property of $\mu$-minimality is not stable under $\mu$-equivalence. For instance, if $g$ is $\mu$-minimal and $\mu(g)>0$, then $g+g^2\smu g$, but $g+g^2$ is not $\mu$-minimal, since $g+g^2\mmu g$ and $\deg (g+g^2)>\deg g$. Nevertheless, for $\mu$-equivalent polynomials of the same degree, $\mu$-minimality is obviously preserved.  
 
\subsection{Key polynomials and augmented valuations}
\begin{definition}
A key polynomial for the valuation $\mu$ is a monic polynomial $\phi\in K[x]$ which is $\mu$-minimal and $\mu$-irreducible. 

We denote by $\kpm$ the set of all key polynomials for $\mu$. 
\end{definition}

For instance, $\kp(\mu_0)$ is the set of all monic polynomials $g\in\oo[x]$ such that $\overline{g}$ is irreducible in $\F[x]$.

Since $\mu$-minimality is not stable under $\mu$-equivalence, the property of being a key polynomial is not stable under $\mu$-equivalence. However, for polynomials of the same degree this stability is clear.


\begin{lemma}\label{mid=sim}
Let $\phi$ be a key polynomial for $\mu$, and $g\in K[x]$ a monic polynomial such that $\phi\mmu g$ and $\deg g=\deg\phi$. Then, $\phi\smu g$ and $g$ is a key polynomial for $\mu$.
\end{lemma}

\begin{proof}
The $\phi$-expansion of $g$ is $g=a+\phi$, with $\deg a<\deg\phi$. By item 4 of Lemma \ref{minimal0}, we have $\mu(g)<\mu(a)$, so that $\phi\smu g$. Hence, $\hm(g)=\hm(\phi)$, and $g$ is $\mu$-irreducible. Since $\deg g=\deg \phi$, $g$ is $\mu$-minimal too.
\end{proof}

\begin{definition}
For $\phi\in\kpm$ and $g\in K[x]$, we denote by $\ord_{\mu,\phi}(g)$ the largest integer $s$ such that $\phi^s\mid_\mu g$. 
We convene that $\ord_{\mu,\phi}(0)=\infty$.
\end{definition}

Since $\phi$ is $\mu$-irreducible, for all $g,h\in K[x]$ we have  
\begin{equation}\label{multiplicative}
\ord_{\mu,\phi}(gh)=\ord_{\mu,\phi}(g)+\ord_{\mu,\phi}(h).
\end{equation}

The map $\ord_{\mu,\phi}$ induces a group homomorphism $K(x)^*\to \Z$, but it is not a valuation. For instance, if $n>\mu(\phi)$, then $\ord_{\mu,\phi}(\phi)=1=\ord_{\mu,\phi}(\phi+\pi^n)$, but $\ord_{\mu,\phi}(\pi^n)=0$. However, as a consequence of (\ref{Hmu}), it has the following property.

\begin{lemma}\label{omega}
 If $g,h\in K[x]$ satisfy $\mu(g)=\mu(h)=\mu(g+h)$, then $\ord_{\mu,\phi}(g+h)\ge \min\{\ord_{\mu,\phi}(g),\ord_{\mu,\phi}(h)\}$, and equality holds if $\ord_{\mu,\phi}(g)\ne\ord_{\mu,\phi}(h)$.\hfill{$\Box$}
\end{lemma}

\begin{definition}\label{muprima}
Take $\phi\in \kpm$ and $\la\in \Q_{>0}$. The augmented valuation of $\mu$ with respect to these data is the valuation $\mu'$ determined by the following action on $K[x]$:
\begin{itemize}
\item $\mu'(a)=\mu(a)$, if $\deg a<\deg \phi$.
\item $\mu'(\phi)=\mu(\phi)+\la$.
\item If $g=\sum_{0\le s}g_s\phi^s$ is the $\phi$-expansion of $g$, then $\mu'(g)=\min_{0 \le s}\{\mu'(g_s\phi^s)\}$. 
\end{itemize}

Or equivalently, $\mu'(g)=\min_{0\le s}\{\mu(g_s\phi^s)+s\la\}$. 
We denote $\mu'=[\mu;(\phi,\la)]$.
\end{definition}

\begin{proposition}\cite[Thms. 4.2, 5.1]{mcla}, \cite[Thm. 1.2, Prop. 1.3]{Vaq}\label{extension}
\begin{enumerate}
\item 
The natural extension of $\mu'$ to $K(x)$ is a valuation on this field and $\mu\le\mu'$.
\item
For a non-zero $g\in K[x]$, $\mu(g)=\mu'(g)$ if and only if $\phi\nmid_{\mu}g$. Hence, $\hm(\phi)\,\ggm=\op{Ker}(\ggm\to\gg(\mu'))$.
\item
The group $\Gamma(\mu')$ is the subgroup of $\Q$ generated by $\mu'(\phi)$ and the subset

$\Gamma_\phi(\mu):=\left\{\mu(g)\mid g\in K[x],\ g\ne0,\ \deg g<\deg\phi\right\}\subset \Gamma(\mu)$.\hfill{$\Box$}
\end{enumerate}
\end{proposition}

The group $\Gamma(\mu')$ does not necessarily contain $\Gamma(\mu)$. For instance, for the valuations
$$
\mu=[\mu_0;(x,1/2)],\qquad \mu'=[\mu;(x,1/2)]=[\mu_0;(x,1)],
$$
we have $\Gamma(\mu)=(1/2)\Z$, which is larger than $\Gamma(\mu')=\Z$.

\begin{lemma}\label{phi}
Let $\mu'=[\mu;(\phi,\la)]$ be an augmented valuation. Then, $\phi$ is a key polynomial for $\mu'$. 
\end{lemma}

\begin{proof}
By Lemma \ref{minimal0}, $\phi$ is $\mu'$-minimal; thus, $\hmp(\phi)$ is not a unit in $\gg(\mu')$. Suppose that $\phi\mid_{\mu'}gh$ for non-zero $g,h\in K[x]$. The $0$-th coefficient of the $\phi$-expansion of $gh$ is the remainder $c_0$ of the division of $g_0h_0$ by $\phi$. Since $\phi$ is $\mu$-irreducible, we have $\phi\nmid_\mu g_0h_0$ and Lemma \ref{minimal0} shows that $\mu(g_0h_0)=\mu(c_0)$. By Lemma \ref{minimal0}, from  $\phi\mid_{\mu'}gh$ we deduce $\mu'(gh)<\mu'(c_0)=\mu(c_0)=\mu(g_0h_0)=\mu'(g_0h_0)$. Hence, either $\mu'(g)<\mu'(g_0)$ or
$\mu'(h)<\mu'(h_0)$. By Lemma \ref{minimal0}, either $\phi\mid_{\mu'}g$, or $\phi\mid_{\mu'}h$. 
Thus, $\phi$ is $\mu'$-irreducible.
\end{proof}

\begin{lemma}\label{irredKv}
Every $\phi\in \kpm$ is irreducible in $K_v[x]$.
\end{lemma}

\begin{proof}
Suppose $\phi=gh$ for two monic polynomials $g,h\in K_v[x]$. Then, for any positive integer $n$, there exist polynomials $g_n,h_n\in K[x]$ such that $\phi\equiv g_nh_n
\md{\m^n}$ and $\deg g_n=\deg g$, $\deg h_n=\deg h$. By taking $n$ large enough, we get $\phi\smu g_nh_n$; by the $\mu$-irreducibility,
 $\phi$ divides one of the factors in $\ggm$, say  $\phi\mmu h_n$. By the $\mu$-minimality of $\phi$, this implies $\deg\phi\le \deg h_n$. Thus, necessarily $\phi=h$ and $g=1$.  
\end{proof}

Let $\phi$ be a key polynomial for $\mu$. Choose a root $\t \in\kb$ of $\phi$ and denote $K_\phi=K_v(\t)$ the finite extension of $K_v$ generated by $\t$. Also, let $\oo_\phi\subset K_\phi$ be the valuation ring of $K_\phi$, $\m_\phi$ the maximal ideal and $\F_\phi=\oo_\phi/\m_\phi$ the residue class field. 

We denote by $e(\phi)$ and $f(\phi)$ the ramification index and residual degree of $K_\phi/K_v$, respectively. Hence, $\deg \phi=e(\phi)f(\phi)$.

Let $\mu_{\infty,\phi}$ be the pseudo-valuation on $K[x]$ obtained as the composition:
$$
\mu_{\infty,\phi}\colon K[x]\lra K_v(\t)\stackrel{v}\lra \Q\cup\{\infty\},
$$the first mapping being determined by $x\mapsto \t$. By the uniqueness of the extension of $v$ to $\kb$, this pseudo-valuation does not depend on the choice of $\t$ as a root of $\phi$. 

We recall that a pseudo-valuation has the same properties than a valuation, except for the fact that the pre-image of $\infty$ is a prime ideal which is not necessarily zero. For $\mu_{\phi,\infty}$ this prime ideal is the ideal of $K[x]$ generated by $\phi$.

Consider now the map $\mu'\colon K[x]\to \Q\cup\{\infty\}$, where $\mu':=[\mu;(\phi,\infty)]$ is defined as in Definition \ref{muprima}, but taking $\la=\infty$. The arguments in the proof of Proposition \ref{extension} are equally valid in this case, and they show that $\mu'$ is a pseudo-valuation on $K[x]$ such that $\mu\le \mu'$, and for a non-zero $g\in K[x]$, $\mu(g)=\mu'(g)$ if and only if $\phi\nmid_{\mu}g$. 

Since $(\mu')^{-1}(\infty)=\phi\, K[x]$, the pseudo-valuations $\mu'$ and $\mu_{\infty,\phi}$ induce two valuations on the field $K[x]/(\phi)$. These valuations coincide because $\phi$ is irreducible in $K_v[x]$ and the field  $K[x]/(\phi)$ admits a unique valuation extending $v$ on $K$. This implies that $\mu'=\mu_{\infty,\phi}$. Hence, we obtain the following results by mimicking Proposition \ref{extension}.   

\begin{proposition}\label{theta}
If $\phi$ is a key polynomial for $\mu$, then 
\begin{enumerate}
\item $\mu\le \mu_{\infty,\phi}$, and for a non-zero $g\in K [x]$, $\mu(g)=\mu_{\infty,\phi}(g)$ if and only if $\phi\nmid_{\mu}g$. 
\item $v(K_\phi^*)=\Gamma_\phi(\mu):=\left\{\mu(g)\mid g\in K[x],\ g\ne0,\ \deg g<\deg\phi\right\}\subset \Gamma(\mu)$. 

In particular, $\Gamma_\phi(\mu)$ is a subgroup of $\Gamma(\mu)$.\hfill{$\Box$}
\end{enumerate}
\end{proposition}

In principle, a strict translation of Proposition \ref{extension} would only state that $v(K_\phi^*)$ is the subgroup of $\Gamma(\mu)$ genera\-ted by the subset $\Gamma_\phi(\mu)$. However, every element in $K_\phi^*$ can be expressed as $\tilde{g}(\t)$ for some non-zero $\tilde{g}\in K_v[x]$ with $\deg\tilde{g}<\deg\phi$; hence, if $g\in K[x]$ has $\deg g=\deg\tilde{g}$ and it is sufficiently close to $\tilde{g}$ in the $v$-adic topology, we have $v(\tilde{g}(\t))=v(g(\t))=\mu(g)\in\Gamma_\phi(\mu)$. This shows that $v(K_\phi^*)=\Gamma_\phi(\mu)$.

\begin{corollary}\label{inO}
$\kpm\subset \oo [x]$.
\end{corollary}

\begin{proof}
Suppose $\phi\in\kpm$ and let $\t\in\kb$ be a root of $\phi$.
Since $0\le \mu(x)\le \mu_{\infty,\phi}(x)=v(\t)$, the root $\t$ belongs to $\oo_\phi$ and its minimal polynomial $\phi$ over $K_v$ must have coefficients in $\oo_v\cap K=\oo$.  
\end{proof}

The next result is a kind of partial converse to Propositions \ref{extension} and \ref{theta}.

\begin{lemma}\label{minDegree}\cite[Thm. 1.15]{Vaq}
Let $\mu$ be a valuation and $\mu'$ a pseudo-valuation on $K[x]$ such that $\mu<\mu'$. Let $\phi\in K[x]$ be a monic polynomial with minimal degree satisfying $\mu(\phi)<\mu'(\phi)$. Then, $\phi$ is a key polynomial for $\mu$ and for any non-zero $g\in K[x]$, $\mu(g)=\mu'(g)$ is equivalent to $\phi\nmid_\mu g$. Moreover, for $\la=\mu'(\phi)-\mu(\phi)\in\Q_{>0}\cup\{\infty\}$, we have $\mu<[\mu;(\phi,\la)]\le\mu'$.\hfill{$\Box$}
\end{lemma}

\subsection{Residual ideals of polynomials}\label{subsecRideal}
Let $\mu$ be a valuation on $K(x)$. Denote $\Delta=\Delta(\mu)$, and let $I(\Delta)$ be the set of ideals in $\Delta$. Consider the \emph{residual ideal operator}:
$$
\rr=\rrm\colon K[x]\lra I(\Delta),\qquad g\mapsto \Delta\cap \hm(g)\ggm.
$$
The following basic properties of $\rr$ are immediate.
\begin{equation}\label{basic}
\begin{array}{rcl}
g\mmu h&\imp& \rr(g)\supset \rr(h),\\
g\smu h&\imp& \rr(g)=\rr(h),\\
\hm(g)\in\ggm^* &\sii& \rr(g)=\Delta,
\end{array}
\end{equation}
where $\ggm^*$ is the group of units of $\ggm$. In sections \ref{secGr} and \ref{secKP} we shall derive more properties of this operator $\rr$, which translates questions about $K[x]$ and $\mu$ into ideal-theoretic consi\-derations in the ring $\Delta$.  Let us now see that $\rr$ attaches a maximal ideal of $\Delta$ to any key polynomial for $\mu$.    

\begin{proposition}\label{sameideal}
If $\phi$ is a key polynomial for $\mu$, then 
\begin{enumerate}
\item $\rr(\phi)$ is the kernel of the onto ring homomorphism $\Delta\twoheadrightarrow \F_\phi$ determined by $g(x)+\pset^+_0\ \mapsto\ g(\t)+\m_\phi$. 
In particular, $\rr(\phi)$ is a maximal ideal of $\Delta$.
\item For any augmented valuation  $\mu'=[\mu;(\phi,\la)]$, $\rr(\phi)=\op{Ker}(\Delta\to \Delta(\mu'))$. Thus, the image of $\Delta\to\Delta(\mu')$ is a field, canonically isomorphic to $\F_\phi$. 
\end{enumerate}
\end{proposition}

\begin{proof}
If $g\in\pset_0$, we have $v(g(\t))\ge \mu(g)\ge0$, and $g(\t)\in \oo_\phi$. Thus, we get a well-defined ring homomorphism $\pset_0\to \F_\phi$. This mapping is onto, because every element in $\F_\phi$ may be represented as $g(\t)+\m_\phi$ for some $g\in K[x]$, with $\deg g<\deg\phi=[K_\phi\colon K_v]$, satisfying $v(g(\t))\ge0$. Proposition \ref{theta} shows that $\mu(g)=v(g(\t))\ge0$, so that $g$ belongs to $\pset_0$. Finally, if $g\in\pset_0^+$, then $v(g(\t))\ge \mu(g)>0$; thus, the above homomorphism vanishes on $\pset_0^+$ and it determines the onto map $\Delta\twoheadrightarrow \F_\phi$. The kernel of this homomorphism is $\rr(\phi)$ 
by Proposition \ref{theta}. 

The second item is a consequence of Proposition \ref{extension} and the first item. 
\end{proof}


\section{Newton polygons}
The choice of a key polynomial $\phi$ for a valuation $\mu$ determines a \emph{Newton polygon operator}
$$
\nph\colon K[x]\lra 2^{\R^2},
$$
where $2^{\R^2}$ is the set of subsets of the euclidean plane $\R^2$. The Newton polygon of the zero polynomial is the empty set. If $g=\sum_{0\le s}a_s\phi^s$ is the canonical $\phi$-expansion of a non-zero polynomial $g\in K[x]$, then $\nph(g)$ is the lower convex hull of the cloud of points $(s,\mu(a_s\phi^s))$ for all $0\le s$. 

\begin{definition}\label{length}
The length of a Newton polygon $N$ is the abscissa of its right end point. It will be denoted by $\ell(N)$. 
\end{definition}

Since $\mu(g)=\min_{0\le s}\{\mu(a_s\phi^s)\}$, the rational number $\mu(g)$ is the ordinate of the point of intersection of the vertical axis with the line of slope zero which first touches the polygon from below. Figure \ref{figNmodel} shows the typical shape of $\nph(g)$.    

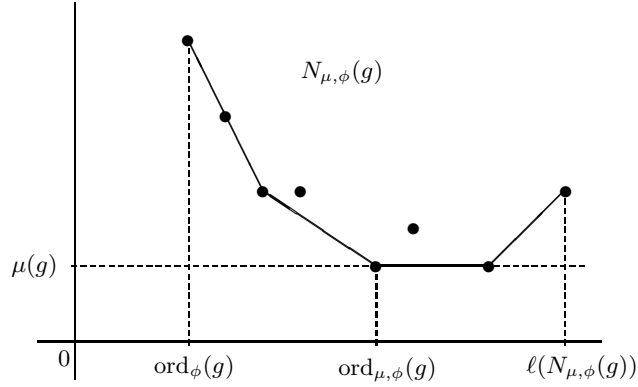
\begin{figure}
\caption{Newton polygon of $g\in K[x]$}\label{figNmodel}
\begin{center}
\setlength{\unitlength}{5mm}
\begin{picture}(14,10)
\put(10.8,2.8){$\bullet$}\put(8.8,3.8){$\bullet$}\put(7.8,2.8){$\bullet$}\put(5.8,4.8){$\bullet$}
\put(4.8,4.8){$\bullet$}\put(3.8,6.8){$\bullet$}\put(2.8,8.8){$\bullet$}
\put(-1,1){\line(1,0){15}}\put(0,0){\line(0,1){10}}
\put(8,3){\line(-3,2){3}}\put(5,5){\line(-1,2){2}}\put(8,3.03){\line(-3,2){3}}
\put(5,5.03){\line(-1,2){2}}\put(11,3){\line(-1,0){3}}\put(11,3.02){\line(-1,0){3}}
\put(12.85,4.8){$\bullet$}\put(11,3){\line(1,1){2}}\put(11,3.02){\line(1,1){2}}
\multiput(3,.9)(0,.25){32}{\vrule height2pt}
\multiput(8,.9)(0,.25){9}{\vrule height2pt}
\multiput(13,.9)(0,.25){16}{\vrule height2pt}
\put(7,.1){\begin{footnotesize}$\ord_{\mu,\phi}(g)$\end{footnotesize}}
\put(2.1,.15){\begin{footnotesize}$\ord_{\phi}(g)$\end{footnotesize}}
\put(12,.15){\begin{footnotesize}$\ell(\nph(g))$\end{footnotesize}}
\multiput(-.1,3)(.25,0){55}{\hbox to 2pt{\hrulefill }}
\put(6,8){\begin{footnotesize}$\nph(g)$\end{footnotesize}}
\put(-1.7,2.85){\begin{footnotesize}$\mu(g)$\end{footnotesize}}
\put(-.45,.35){\begin{footnotesize}$0$\end{footnotesize}}
\end{picture}
\end{center}
\end{figure}

If the Newton polygon $N=\nph(g)$ is not a single point, we formally write $N=S_1+\cdots +S_k$, where $S_i$ are the sides of $N$, ordered by their increasing slopes. The left and right end points of $N$ and the points joining two sides of different slopes are called \emph{vertices} of $N$.

Usually, we shall be interested only in the \emph{principal Newton polygon} $\nph^-(g)$ formed by the sides of negative slope. If there are no sides of negative slope, then $\nph^-(g)$ is the left end point of $\nph(g)$

\begin{lemma}\label{length2}
For every non-zero polynomial $g\in K[x]$, 
$\ell(\nph^-(g))=\ord_{\mu,\phi}(g)$. 
\end{lemma}

\begin{proof}
Let $g=\sum_{0\le s}a_s\phi^s$ be the $\phi$-expansion of $g$. If $a_s\ne0$, then $\phi\nmid_\mu a_s$, because $\deg a_s<\deg\phi$ and $\phi$ is $\mu$-minimal. Hence, $\ord_{\mu,\phi}(a_s\phi^s)=s$ for all $s$ such that  $a_s\ne0$.

Let $I=\{s\in\Z_{\ge0}\mid \mu(a_s\phi^s)=\mu(g)\}$ and consider $h=\sum_{s\in I}a_s\phi^s$. Clearly, $s_0:=\ell(\nph^-(g))=\min(I)$. Since $g\smu h$, we have 
$\ord_{\mu,\phi}(g)=\ord_{\mu,\phi}(h)=s_0$, by Lemma \ref{omega}.   
\end{proof}

From now on, we fix $\mu'=[\mu;(\phi,\la)]$, an augmented valuation of $\mu$ with respect to the key polynomial $\phi$ and a positive rational number $\la$. Let us first show how to read the value $\mu'(g)$ in the Newton polygon $N_{\mu,\phi}(g)$ (see Figure \ref{figComponent}).

\begin{lemma}\label{muprimaN}
For any non-zero $g\in K[x]$, the line of slope $-\la$ which first touches the polygon $N_{\mu,\phi}(g)$ from below cuts the vertical axis at the point $(0,\mu'(g))$. 
\end{lemma}

\begin{proof}
For any point $P=(s,\mu(a_s\phi^s))$, the value $\mu'(a_s\phi^s)=\mu(a_s\phi^s)+s\la$ is the ordinate of the point of intersection of the vertical axis with the line of slope $-\la$ passing through $P$ (Figure \ref{figNpunt}). The lemma follows from $\mu'(g)=\min_{0 \le s}\{\mu'(a_s\phi^s)\}$.
\end{proof}

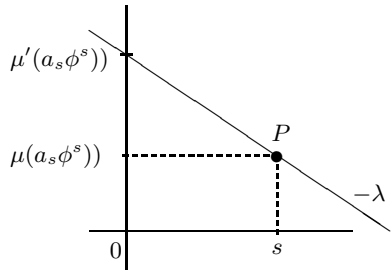
\begin{figure}
\caption{Newton polygon of a monomial $a_s\phi^s$}\label{figNpunt}
\begin{center}
\setlength{\unitlength}{5mm}
\begin{picture}(6,7.5)
\put(3.8,2.8){$\bullet$}
\put(-1,1){\line(1,0){7}}\put(0,0){\line(0,1){7}}
\put(-1,6.35){\line(3,-2){8}}
\multiput(4,.9)(0,.25){9}{\vrule height2pt}
\multiput(-.1,3)(.25,0){16}{\hbox to 2pt{\hrulefill }}
\put(-.15,5.7){\line(1,0){.3}}
\put(-3.1,2.85){\begin{footnotesize}$\mu(a_s\phi^s))$\end{footnotesize}}
\put(-3.1,5.4){\begin{footnotesize}$\mu'(a_s\phi^s))$\end{footnotesize}}
\put(-.45,.3){\begin{footnotesize}$0$\end{footnotesize}}
\put(3.85,.4){\begin{footnotesize}$s$\end{footnotesize}}
\put(3.85,3.4){\begin{footnotesize}$P$\end{footnotesize}}
\put(6,1.8){\begin{footnotesize}$-\la$\end{footnotesize}}
\end{picture}
\end{center}
\end{figure}

By Lemma \ref{phi}, $\phi$ is a key polynomial for $\mu'$ and it makes sense to consider the Newton polygon $N_{\mu',\phi}(g)$, which is related to $\nph(g)$ in an obvius way.

\begin{lemma}\label{affinity}
Let $\hh\colon \R^2\lra\R^2$ be the affine transformation $\hh(x,y)=(x,y+\la x)$. Then, $N_{\mu',\phi}(g)=\hh(\nph(g))$.\hfill{$\Box$} 
\end{lemma}

The affinity $\hh$ acts as a translation on every vertical line, and the vertical axis is pointwise invariant. If $S$ is a side of $\nph(g)$ of slope $\rho$, then $\hh(S)$ is a side of $N_{\mu',\phi}(g)$ of slope $\rho+\la$; also, the end points of $S$ and $\hh(S)$ have the same abscissas.

\begin{definition}\label{sla}
Let $g\in K[x]$ be a non-zero polynomial and denote $N=\nph(g)$. We define the \emph{$\la$-component} of $g$ as the segment 
$$S_\la(g):=S_{\mu'}(g):=\{(x,y)\in N\mid y+\la x\mbox{ is minimal}\}=N\cap L_{-\la},
$$where $L_{-\la}$ is the line of slope $-\la$ which first touches the polygon $N$ from below. 
 
We denote by $s(g)=s_{\mu'}(g)\le s'(g)=s'_{\mu'}(g)$ the abscissas of the end points of $S_\la(g)$. We denote by $u(g)=u_{\mu'}(g)$ the integer such that $(s(g),u(g)/e(\mu))$ is the left end point of $S_\la(g)$ (see Figure \ref{figComponent}). 
\end{definition}


If $N$ has a side $S$ of slope $-\la$, then $S_\la(g)=S$; otherwise, $S_\la(g)$ is a vertex of $N$ and $s(g)=s'(g)$.  

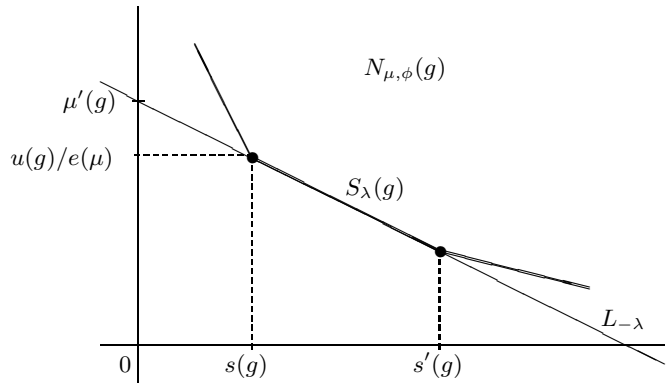
\begin{figure}
\caption{$\la$-component of a polynomial $g\in K[x]$.}\label{figComponent}
\begin{center}
\setlength{\unitlength}{5mm}
\begin{picture}(14,10)
\put(2.85,5.8){$\bullet$}\put(7.85,3.3){$\bullet$}
\put(-1,1){\line(1,0){15}}\put(0,0){\line(0,1){10}}
\put(-1,8){\line(2,-1){15}}
\put(3,6){\line(-1,2){1.5}}\put(3,6.04){\line(-1,2){1.5}}
\put(3,6){\line(2,-1){5}}\put(3,6.04){\line(2,-1){5}}
\put(8,3.5){\line(4,-1){4}}\put(8,3.54){\line(4,-1){4}}
\multiput(3,.9)(0,.25){21}{\vrule height2pt}
\multiput(8,.9)(0,.25){11}{\vrule height2pt}
\multiput(-.1,6.05)(.25,0){13}{\hbox to 2pt{\hrulefill }}
\put(7.3,.3){\begin{footnotesize}$s'(g)$\end{footnotesize}}
\put(2.3,.3){\begin{footnotesize}$s(g)$\end{footnotesize}}
\put(-3.4,5.8){\begin{footnotesize}$u(g)/e(\mu)$\end{footnotesize}}
\put(12.3,1.5){\begin{footnotesize}$L_{-\la}$\end{footnotesize}}
\put(-.5,.3){\begin{footnotesize}$0$\end{footnotesize}}
\put(6,8.2){\begin{footnotesize}$\nph(g)$\end{footnotesize}}
\put(5.5,5){\begin{footnotesize}$S_\la(g)$\end{footnotesize}}
\put(-.15,7.5){\line(1,0){.3}}
\put(-2,7.3){\begin{footnotesize}$\mu'(g)$\end{footnotesize}}
\end{picture}
\end{center}
\end{figure}

If $g=\sum_{0\le s}a_s\phi^s$ is the $\phi$-expansion of $g$, consider  $I=\{s\in\Z_{\ge0}\mid \mu'(a_s\phi^s)=\mu'(g)\}$. By Lemma \ref{muprimaN}, $I$ coincides with the set of all $s\in\Z_{\ge0}$ such that the point $(s,\mu(a_s\phi^s))$ lies on $S_\la(g)$; in particular, $s(g)=\min(I)$, $s'(g)=\max(I)$. 
According to MacLane's terminology, $s'(g)$ is the \emph{effective degree} of $g$ and $s'(g)-s(g)$ is the \emph{projection} of $g$ with respect to the augmented valuation $\mu'$ \cite[Secs. 3,4]{mclb}. 

\begin{lemma}\label{additivity}\mbox{\null}
Let $g,h\in K[x]$ be two non-zero polynomials.
\begin{enumerate}
\item If $g\sim_{\mu'}h$, then $S_\la(g)=S_\la(h)$.
\item $s(g)=\ord_{\mu',\phi}(g)$.
\item $s'(g)=\ord_{\mu'',\phi}(g)$, where $\mu''=[\mu;(\phi,\la-\epsilon)]$ for a sufficiently small $\epsilon\in\Q_{>0}$.
\end{enumerate}
\end{lemma}

\begin{proof}
The first item is a consequence of the remarks preceeding the lemma.

By Lemma \ref{affinity}, the image of $S_\la(g)$ under the affinity
$\hh$ is the side of slope zero of $N_{\mu',\phi}(g)$. Thus, $s(g)=\ord_{\mu',\phi}(g)$, by Lemma \ref{length2}. 

Finally, if $\epsilon\in\Q_{>0}$ is sufficiently small, the right end point of $S_\la(g)$ is equal to the left end point of $S_{\la-\epsilon}(g)$ and item 3 is a consequence of item 2.
\end{proof}

There is a natural addition of segments in the plane. We admit that a point in the plane is a segment whose left and right end points coincide. Given two segments $S,S'$, the addition $S+S'$ is the ordinary vector sum if at least one of the segments is a single point. Otherwise, $N=S+S'$ is the Newton polygon whose left end point is the vector sum of the two left end points of $S$ and $S'$, and whose sides are the join of $S$ and $S'$, considered with increasing slopes from left to right (see Figure \ref{figSum}).

\begin{figure}
\caption{Addition of two segments}\label{figSum}
\begin{center}
\setlength{\unitlength}{5mm}
\begin{picture}(10,9)
\put(-1,1){\line(1,0){11}}\put(0,0){\line(0,1){8.5}}
\put(0,1){\vector(1,3){1}}
\put(.8,3.8){$\bullet$}\put(2.8,2.8){$\bullet$}
\put(1,4){\line(2,-1){2}}\put(1,4.02){\line(2,-1){2}}
\put(0,1){\vector(1,1){4}}
\put(3.8,4.8){$\bullet$}\put(4.8,1.8){$\bullet$}
\put(4,5){\line(1,-3){1}}\put(4,5.02){\line(1,-3){1}}
\put(1,4){\vector(1,1){4}}\put(4,5){\vector(1,3){1}}
\put(4.8,7.8){$\bullet$}\put(5.8,4.8){$\bullet$}\put(7.8,3.8){$\bullet$}
\put(5,8){\line(1,-3){1}}\put(5,8.02){\line(1,-3){1}}
\put(6,5){\line(2,-1){2}}\put(6,5.02){\line(2,-1){2}}
\put(-.45,.3){\begin{footnotesize}$0$\end{footnotesize}}
\put(1.85,3.7){\begin{footnotesize}$S$\end{footnotesize}}
\put(4.7,3.3){\begin{footnotesize}$S'$\end{footnotesize}}
\put(6.5,5.3){\begin{footnotesize}$S+S'$\end{footnotesize}}
\end{picture}
\end{center}
\end{figure}
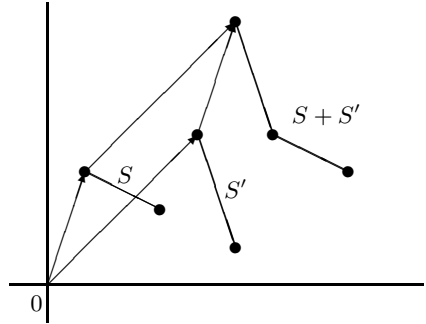

\begin{corollary}\label{sumSla}
For non-zero $g,h\in K[x]$, we have $S_\la(gh)=S_\la(g)+S_\la(h)$. 
\end{corollary}
    
\begin{proof}
For sides of the same slope (or of length zero), this additivity is equivalent to
$$
s(gh)=s(g)+s(h),\quad s'(gh)=s'(g)+s'(h),\quad \mbox{and}\quad u(gh)=u(g)+u(h).
$$
The two first equalities are a consequence of equation (\ref{multiplicative}) and Lemma \ref{additivity}. In order to prove the third, let $g=\sum_{0\le s}a_s\phi^s$,  $h=\sum_{0\le t}b_t\phi^t$ be the $\phi$-expansions of $g,h$, respectively. Let $s_0=s(g)$, $t_0=s(h)$, and consider $G=\sum_{s_0\le s}a_s\phi^s$,  $H=\sum_{t_0\le t}b_t\phi^t$. Since $G\smu g$,  $H\smu h$,  $GH\smu gh$, we may suppose $g=G$, $h=H$. Now, the left end point of $S_\la(gh)$ has abscissa $s_0+t_0$ and the $(s_0+t_0)$-th term of the $\phi$-expansion of $gh$ is the remainder $c$ of the division $a_{s_0}b_{t_0}=\phi q+c$. Since $\phi\nmid_\mu a_{s_0}b_{t_0}$, Lemma \ref{minimal0} shows that $\mu(a_{s_0}b_{t_0})=\mu(c)$. Hence, $u(g)+u(h)=e(\mu)\left(\mu(a_{s_0}\phi^{s_0})+\mu(b_{t_0}\phi^{t_0})\right)=e(\mu)\mu(c\phi^{s_0+t_0})=u(gh)$.  
\end{proof}

The addition of segments may be extended to an addition law for Newton polygons. Given two  polygons $N=S_1+\cdots+S_k$, $N'=S'_1+\cdots +S'_{k'}$, the left end point of the sum $N+N'$ is  the vector sum of the left end points of $N$ and $N'$, whereas the sides of $N+N'$ are obtained by joining all sides in the multiset $\left\{S_1,\dots,S_k,S'_1,\dots,S'_{k'}\right\}$, ordered by increasing slopes \cite[Sec. 1]{HN}.    
As an immediate consequence of Corollary \ref{sumSla}, we get the \emph{Theorem of the product} for Newton polygons.

\begin{theorem}\label{product}
Let $\mu$ be a valuation and $\phi$ a key polynomial for $\mu$. Then,
$\nph^-(gh)=\nph^-(g)+\nph^-(h)$ for any non-zero $g,h\in K[x]$.
\hfill{$\Box$}
\end{theorem}

The analogous statement for entire Newton polygons is false. For instance, consider $g,h\in K[x]$ such that $\deg g,\,\deg h<\deg\phi$ and $\deg gh\ge \deg \phi$; then, $N_{\mu,\phi}(g)$ and $N_{\mu,\phi}(h)$ are a single point, while $N_{\mu,\phi}(gh)$ has a side of length one.  

We now apply these Newton polygon techniques to obtain 
a characterization of the units in $\gg(\mu')$ and a criterion for $\mu'$-minimality in terms of $\phi$-expansions.

\begin{lemma}\label{units}
For any non-zero $g\in K[x]$, $\hmp(g)$ is a unit in $\ggmp$ if and only if $g\sim_{\mu'}a$ for some $a\in K[x]$ such that $\deg a<\deg\phi$. This condition holds if $\phi\nmid_{\mu}g$.
\end{lemma}

\begin{proof}
Suppose $g\sim_{\mu'}a$ for some $a\in K[x]$ such that $\deg a<\deg\phi$. Since $a$ is coprime to $\phi$, we have a B\'ezout identity $ah+\phi f=1$, with $h,f\in K[x]$ and $\deg h<\deg\phi$.
By Proposition \ref{extension}, $\mu(ah)=\mu'(ah)$; hence,
$$
\mu'(ah-1)=\mu'(\phi f)>\mu(\phi f)\ge \min\{\mu(ah),0\}=\min\{\mu'(ah),0\}.
$$
Therefore, $ah\sim_{\mu'}1$ and $H_{\mu'}(g)=H_{\mu'}(a)$ is a unit in $\gg(\mu')$.

Suppose $gh\sim_{\mu'}1$ for some $h\in K[x]$, and let $g=\sum_{0\le s}a_s\phi^s$ be the $\phi$-expansion of $g$. By Lemma \ref{additivity} and Corollary \ref{sumSla}, $S_\la(g)+S_\la(h)=S_\la(gh)=S_\la(1)=\{(0,0)\}$. Hence, $S_\la(g)$ is a single point of abscissa zero. This implies that $g\sim_{\mu'}a_0$.

Finally, suppose $\phi\nmid_\mu g$, and let $g=\phi q+a$ be the divison with remainder of $g$ by $\phi$. By Lemma \ref{minimal0} and Proposition \ref{extension},
$\mu'(g-a)=\mu'(\phi q)>\mu(\phi q)\ge\mu(g)=\mu'(g)$. 
Therefore, $g\sim_{\mu'}a$. 
\end{proof}

\begin{lemma}\label{minimal}
Let $g=\sum_{s=0}^\ell a_s\phi^s$ be the $\phi$-expansion of a non-zero polynomial $g\in K[x]$. 
The following conditions are equivalent:
\begin{enumerate}
\item $g$ is $\mu'$-minimal.
\item $\deg a_\ell=0$ and $\mu'(g)=\mu'\left(a_\ell\phi^\ell\right)$.
\item $\deg g=s'(g)\deg\phi$. 
\end{enumerate}
\end{lemma}

\begin{proof}
Clearly, conditions (2) and (3) are equivalent. 

Let us show that (1) implies (2). Suppose that $g$ is $\mu'$-minimal; write $g=G+H$, where $G$ is the sum of all monomials $a_s\phi^s$ with $\mu'(a_s\phi^s)=\mu'(g)$ and $H$ is the sum of all monomials with $\mu'(a_s\phi^s)>\mu'(g)$. Since $\mu'(g-G)=\mu'(H)>\mu'(g)$, we have $g\sim_{\mu'}G$. Since $g$ is $\mu'$-minimal, we get $\deg g\le \deg G$, and the leading monomial of $g$ must be one of the monomials of $G$.

On the other hand, by Lemma \ref{units}, $H_{\mu'}(a_\ell)$ is a unit in $\gg(\mu')$ and $ca_\ell\sim_{\mu'}1$ for some $c\in K[x]$ with $\deg c<\deg \phi$. For each $0\le s<\ell$, since $\phi\nmid_\mu ca_s$, Lemma \ref{units} shows that $ca_s\sim_{\mu'}c_s$ for some $c_s\in K[x]$ with $\deg c_s<\deg \phi$.
Hence, $cg\sim_{\mu'}f$, where $f:=\phi^\ell+\sum_{s=0}^{\ell-1}c_s\phi^s$. Since $g$ is $\mu'$-minimal, we get $\deg a_\ell+\ell\deg\phi=\deg g\le \deg f=\ell\deg\phi$; hence, $\deg a_\ell=0$.

Conversely, suppose that (3) holds. Consider $f\in K[x]$ such that $g\mid_{\mu'}f$; that is, $f\sim_{\mu'}gh$ for a certain $h\in K[x]$. By Lemma \ref{additivity} and Corollary \ref{sumSla}, 
$s'(f)=s'(gh)=s'(g)+s'(h)$, so that $\deg g=s'(g)\deg\phi\le s'(f)\deg\phi\le \deg f$.
\end{proof}

As a consequence of the criterion of Lemma \ref{minimal}, we may introduce an important numerical invariant of an augmented valuation.

\begin{lemma}\label{bound}
Let $g\in K[x]$ be a monic $\mu'$-minimal polynomial. Then, the positive rational number $C(\mu'):=\mu'(g)/\deg g$ does not depend on $g$.
\end{lemma}

\begin{proof}Lemma \ref{minimal} shows that the $\phi$-expansion of $g$ is of the form:
$$g=\phi^\ell+\sum\nolimits_{0\le s<\ell} a_s\phi^s,\quad \mbox{with }\ \mu'(g)=\mu'(\phi^\ell)=\ell\mu'(\phi).
$$
Since $\deg g=\ell\deg\phi$, we get $\mu'(g)/\deg g=\mu'(\phi)/\deg \phi$. This proves the lemma.
\end{proof}

This holds for the minimal valuation $\mu_0$ too. In fact, a monic polynomial $g$ is $\mu_0$-minimal if and only if it has coefficients in $\oo$; hence  
$C(\mu_0):=\mu_0(g)/\deg g=0$ is independent of $g$.

\section{MacLane's inductive valuations}\label{secInductive}
\subsection{MacLane chains of valuations}\label{subsecChains}
\begin{definition}
A valuation $\mu\in\V$ is called \emph{inductive} if $\mu$ is attained after a finite number of augmentation steps starting with $\mu_0$.
\begin{equation}\label{depth}
\mu_0\ \stackrel{(\phi_1,\la_1)}\lra\  \mu_1\ \stackrel{(\phi_2,\la_2)}\lra\ \cdots
\ \stackrel{(\phi_{r-1},\la_{r-1})}\lra\ \mu_{r-1} 
\ \stackrel{(\phi_{r},\la_{r})}\lra\ \mu_{r}=\mu.
\end{equation}
We denote by $\Vi\subset \V$ the subset of all inductive valuations. 

A chain of augmented valuations as in (\ref{depth}) is called a \emph{MacLane chain of length $r$} of $\mu$ if $\phi_{i+1}\not\sim_{\mu_i}\phi_i$ for all $1\le i<r$.  

We say that (\ref{depth}) is an \emph{optimal MacLane chain} of $\mu$ if  $\deg \phi_1<\cdots<\deg\phi_r$.  
\end{definition}

By Lemma \ref{minimal}, in every chain of augmented valuations we have
$$
\deg \phi_1\mid \deg\phi_2\mid\cdots\mid\deg\phi_{r-1}\mid\deg\phi_r.
$$
The condition $\phi_{i+1}\not\sim_{\mu_{i}}\phi_{i}$ characterizing a MacLane chain is equivalent to  $\phi_{i+1}\nmid_{\mu_{i}}\phi_{i}$. In fact, if $\phi_{i+1}\mid_{\mu_{i}}\phi_{i}$, then  $\deg\phi_{i+1}=\deg\phi_i$ because $\deg\phi_i\mid \deg\phi_{i+1}$ and $\deg\phi_i\ge\deg\phi_{i+1}$ by the $\mu_i$-minimality of $\phi_{i+1}$; hence, $\phi_{i+1}\sim_{\mu_{i}}\phi_{i}$ by Lemma \ref{mid=sim}. 

This shows that an optimal MacLane chain is in particular a MacLane chain. 

In every chain of augmented valuations, the constants $C(\mu_i)$ introduced in Lemma \ref{bound} grow strictly with $i$:
$$
0=C(\mu_0)<C(\mu_1)<\cdots<C(\mu_r)=C(\mu).
$$
In fact, by Lemma \ref{phi}, $\phi_{i+1}\in\kp(\mu_i)\cap\kp(\mu_{i+1})$ for any $0\le i<r$. Hence,
$$
C(\mu_{i+1})=\dfrac{\mu_{i+1}(\phi_{i+1})}{\deg\phi_{i+1}}=\dfrac{\mu_{i}(\phi_{i+1})+\la_{i+1}}{\deg\phi_{i+1}}=C(\mu_{i})+\dfrac{\la_{i+1}}{\deg\phi_{i+1}}>C(\mu_{i}).
$$ 

\begin{lemma}\label{groups}
In a MacLane chain, the group $\Gamma(\mu_{i})$ is the subgroup of $\Q$ generated by $\Gamma(\mu_{i-1})$ and $\la_{i}$, for all $1\le i\le r$. In particular,
$$
\Z=\Gamma(\mu_0)\subset \Gamma(\mu_1)\subset\cdots\subset\Gamma(\mu_{r-1})\subset \Gamma(\mu_r)=\Gamma(\mu).
$$
Moreover, if $e(\mu_{i-1})\la_{i}=h_{i}/e_{i}$, with $h_{i},\,e_{i}$ positive coprime integers, then $e(\mu_{i})=e_{i}e(\mu_{i-1})$ and $e(\phi_i)=e(\mu_{i-1})=e_1\cdots e_{i-1}$.
\end{lemma}

\begin{proof}
By Propositions \ref{extension} and \ref{theta}, $\Gamma(\mu_{i})$ is generated by $\mu_{i-1}(\phi_{i})+\la_{i}$ and the subgroup 
$$\Gamma_{\phi_i}:=\{\mu_{i-1}(g)\mid g\in K[x],\ g\ne0,\ \deg g<\deg\phi_{i}\}\subset \Gamma(\mu_{i-1}).$$ Thus, it suffices to show that $\Gamma_{\phi_i}=\Gamma(\mu_{i-1})$. Let $h=\sum_{0\le s}a_s(\phi_{i-1})^s$ be the $\phi_{i-1}$-expansion of an arbitrary non-zero polynomial. Since $\phi_{i-1}\in\kp(\mu_{i-1})$, we have $$\mu_{i-1}(h)=\mu_{i-1}(a_s(\phi_{i-1})^s)=\mu_{i-1}(a_s)+s\mu_{i-1}(\phi_{i-1}),$$ for a certain $s\ge0$. Since $\deg a_s<\deg \phi_{i-1}\le\deg\phi_{i}$, the value $\mu_{i-1}(a_s)$ belongs to $\Gamma_{\phi_i}$. Hence, it suffices to check that $\mu_{i-1}(\phi_{i-1})\in \Gamma_{\phi_i}$. If $\deg\phi_{i-1}<\deg\phi_{i}$, this is obvious. Suppose $\deg\phi_{i-1}=\deg\phi_{i}$, so that $\phi_{i-1}=\phi_{i}+a$, with $\deg a<\deg\phi_{i}$. In a MacLane chain, $\phi_{i}\nmid_{\mu_{i-1}}\phi_{i-1}$, so that $\mu_{i-1}(\phi_{i-1})=\mu_{i-1}(a)\in \Gamma_{\phi_i}$ by Lemma \ref{minimal0}. 

Clearly, $e_i\Gamma(\mu_i)=\Gamma(\mu_{i-1})$; thus, $e(\mu_{i})=e_{i}e(\mu_{i-1})=e_ie_{i-1}\cdots e_1$, since $e(\mu_0)=1$. Finally, by Proposition \ref{theta}, $v(K_{\phi_i}^*)=\Gamma_{\phi_i}=\Gamma(\mu_{i-1})$, so that $e(\phi_i)=e(\mu_{i-1})$.  
\end{proof}

\begin{corollary}\label{prescribed}
Let $1\le i\le r$. For any $(s,\beta)\in\Z_{\ge0}\times\Gamma(\mu_{i-1})$, there exists $a\in K[x]$  such that $\deg a<\deg\phi_i$ and $N_{\mu_{i-1},\phi_i}(a\phi_i^s)=\{(s,\beta)\}$.
\end{corollary}

\begin {proof}
We get $\beta=\mu_{i-1}(a\phi_i^s)$, by choosing $a$ with
$\mu_{i-1}(a)=\beta-s\mu_{i-1}(\phi_i)\in\Gamma(\mu_{i-1})$. This is possible because $\Gamma(\mu_{i-1})=\Gamma_{\phi_i}$, as shown in the proof of Lemma \ref{groups}.  
\end{proof}

Let us emphasize a stability property of the values of $\mu_i$ along a MacLane chain.
 
\begin{lemma}\label{stable}
For $1\le i\le r$ and $g\in K[x]$, suppose that $\phi_{i}\nmid_{\mu_{i-1}}g$. Then, $\mu_{i-1}(g)=\mu_i(g)=\cdots=\mu(g)$.
\end{lemma}

\begin{proof}
By Proposition \ref{extension}, $\mu_{i-1}(g)=\mu_{i}(g)$. If $i=r$, we are done. If $i<r$, Lemma \ref{units} shows that $g\sim_{\mu_{i}}a$ for some $a\in K[x]$ with $\deg a<\deg\phi_{i}\le\deg \phi_{i+1}$. This implies that  $\phi_{i+1}\nmid_{\mu_{i}}g$, and the argument may be iterated.
\end{proof}

\begin{lemma}\label{existence}
Consider a chain of augmented valuations:
$$
\mu\ \stackrel{(\phi,\la)}\lra\  \mu'\ \stackrel{(\phi',\la')}\lra\ \mu''.
$$
If $\deg\phi=\deg\phi'$, then $\mu''=[\mu;(\phi',\la+\la')]$.
\end{lemma}

\begin{proof}
Write $\phi'=\phi+a$, with $\deg a<\deg\phi$. By the definition of $\mu'$ and Lemma \ref{minimal0}:
$$
\mu(a)=\mu'(a)\ge \mu'(\phi)>\mu(\phi).
$$ 
Thus, $\phi\smu \phi'$, and  $\phi'$ is a key polynomial for $\mu$, by Lemma \ref{mid=sim}.  
In order to show that $\mu''=[\mu';(\phi',\la')]=[\mu;(\phi',\la+\la')]$, it suffices to check that both augmented valuations coincide on $\phi'$ and on all polynomials of degree less than $\deg\phi'$. For any $b\in K[x]$ with $\deg b<\deg \phi'=\deg\phi$, we have $\mu''(b)=\mu'(b)=\mu(b)$, by the definition of the augmented valuations. Finally, 
$$
\mu''(\phi')=\mu'(\phi')+\la'=\mu'(\phi)+\la'=\mu(\phi)+\la+\la'=\mu(\phi')+\la+\la',
$$
where the equality $\mu'(\phi')=\mu'(\phi)$ is deduced from Lemma \ref{bound} (because $\phi,\phi'$ are key polynomials for $\mu'$), and the equality $\mu(\phi')=\mu(\phi)$ is a consequence of $\phi\smu \phi'$.
\end{proof}

Lemma \ref{existence} shows that  every inductive valuation admits optimal MacLane chains.
Let us now discuss their unicity. 

\begin{lemma}\label{unique}
Let $\nu$ be an inductive valuation and let $\mu=[\nu;(\phi,\la)]$, $\mu'=[\nu;(\phi',\la')]$ be two augmented valuations of $\nu$. Then, $\mu=\mu'$ if and only if $\deg\phi=\deg\phi'$, $\mu(\phi)=\mu(\phi')$ and $\la=\la'$. 
In this case, we also have $\phi\sim_\nu\phi'$.
\end{lemma}

\begin{proof}
Suppose $\mu=\mu'$. By the definition of an augmented valuation, 
$$\deg\phi=\min\{\deg g\mid g\in K[x],\ \nu(g)<\mu(g)\},$$
so that $\deg\phi=\deg\phi'$. By Lemma \ref{phi}, $\phi,\phi'\in\kp(\nu)\cap \kpm$; hence, Lemma \ref{bound}
shows that $\nu(\phi)=\nu(\phi')$ and $\mu(\phi)=\mu(\phi')$. This implies $\la=\la'$ too. Also, since $\deg(\phi-\phi')<\deg\phi$, Proposition \ref{extension} shows that $\nu(\phi-\phi')=\mu(\phi-\phi')\ge\mu(\phi)>\nu(\phi)$, so that $\phi\sim_\nu\phi'$.

Conversely, suppose $\deg\phi=\deg\phi'$, $\mu(\phi)=\mu(\phi')$ and $\la=\la'$. We claim that:
\begin{equation}\label{all=}
\delta:=\mu(\phi)=\mu(\phi')=\mu'(\phi')=\mu'(\phi).
\end{equation}
In fact, Lemma \ref{bound} shows that $\nu(\phi)=\nu(\phi')$, leading to $\mu(\phi)=\nu(\phi)+\la=\nu(\phi')+\la'=\mu'(\phi')$. Also, if $\phi'=\phi+a$, then,
$$
\mu'(\phi)=\min\{\mu'(\phi'),\nu(a)\}=\min\{\mu(\phi),\nu(a)\}=\mu(\phi').
$$ 
This ends the proof of (\ref{all=}). Now, for any $\phi$-expansion $g=\sum_{0\le s}a_s\phi^s$, we have
$$
\mu'(g)\ge \min_{0\le s}\{\mu'(a_s\phi^s)\}=\min_{0\le s}\{\nu(a_s)+s\delta\}=\mu(g).
$$
By the symmetry of the argument, we deduce that $\mu=\mu'$.
\end{proof}

\begin{proposition}\label{unicity}
Suppose the inductive valuation $\mu$ admits an optimal MacLane chain as in (\ref{depth}). Consider another optimal MacLane chain
$$
\mu_0\ \stackrel{(\phi'_1,\la'_1)}\lra\  \mu'_1\ \stackrel{(\phi'_2,\la'_2)}\lra\ \cdots
\ \stackrel{(\phi'_{r'-1},\la_{r'-1})}\lra\ \mu'_{r'-1} 
\ \stackrel{(\phi'_{r'},\la'_{r'})}\lra\ \mu'_{r'}=\mu'.
$$
Then, $\mu=\mu'$ if and only if $r=r'$ and:
$$\deg \phi_i=\deg\phi'_i, \quad \mu_i(\phi_i)=\mu_i(\phi'_i), \quad \la_i=\la'_i,\quad\mbox{ for all }\ 1\le i \le r.
$$
In this case, we also have $\mu_i=\mu'_i$ and $\phi_i\sim_{\mu_{i-1}}\phi'_i$ for all $1\le i \le r$. 
\end{proposition}

\begin{proof}
The sufficiency of the conditions is a consequence of Lemma \ref{unique}.

Suppose $\mu=\mu'$ and (for instance) $r\le r'$. Let us show that:
$$
\mu_{i-1}=\mu'_{i-1}\imp \deg \phi_i=\deg\phi'_i, \ \mu_i(\phi_i)=\mu_i(\phi'_i), \ \la_i=\la'_i, \ \mbox{ and }\  \mu_i=\mu'_i,
$$
for all  $1\le i\le r$. In fact, by Lemma \ref{stable},
$$\deg\phi_i=\min\{\deg g\mid g\in K[x],\ \mu_{i-1}(g)<\mu(g)\}=\deg\phi'_i.$$
Write $\phi_i'=\phi_i+a$, with $\deg a<\deg\phi_i=\deg \phi'_i$. By the optimality of both Maclane chains
, $\phi_{i+1}\nmid_{\mu_i}\phi_i'$ and $\phi'_{i+1}\nmid_{\mu'_i}\phi_i$; hence, Lemma \ref{stable} shows that 
$$
\begin{array}{c}
\mu_i(\phi_i)=\mu(\phi_i)=\mu'_i(\phi_i)=\min\{\mu'_i(\phi'_i),\mu_{i-1}(a)\},\\
\mu'_i(\phi'_i)=\mu(\phi'_i)=\mu_i(\phi'_i)=\min\{\mu_i(\phi_i),\mu_{i-1}(a)\}.
\end{array}
$$
Hence, $\mu_i(\phi_i)=\mu'_i(\phi'_i)$. Also, $\mu_{i-1}(\phi_i)=\mu_{i-1}(\phi'_i)$, by Lemma \ref{bound}, so that $\la_i=\mu_i(\phi_i)-\mu_{i-1}(\phi_i)=\mu'_i(\phi'_i)-\mu_{i-1}(\phi'_i)=\la'_i$. By Lemma \ref{unique}, $\mu_i=\mu'_i$.

Since both chains start with $\mu_0$, the iteration of this argument leads to $\mu=\mu_r=\mu'_r$. The inequality $r<r'$ implies $\mu=\mu'_r<\mu'$, against our assumption. Thus, $r=r'$. 
\end{proof}

Hence, in an optimal MacLane chain of $\mu$, the intermediate valuations $\mu_1,\dots,\mu_{r-1}$, the positive rational numbers $\la_1,\dots,\la_r$ and the integers $\deg\phi_1,\dots,\deg\phi_r$ are intrinsic data of $\mu$, whereas the key polynomials $\phi_1,\dots,\phi_r$ admit different choices. More precisely, $\phi'_1,\dots,\phi'_r$ is the family of key polynomials of another optimal MacLane chain of $\mu$ if and only if
$$\phi'_i=\phi_i+a_i, \quad \deg a_i<\deg\phi_i, \quad \mu_i(a_i)\ge \mu_i(\phi_i),\quad \mbox{for all }1\le i\le r.$$
We also have $\phi_i\sim_{\mu_{i-1}}\phi'_i$ for all $i$. Nevertheless, $\phi_i\not\sim_{\mu_{i}}\phi'_i$ when $\mu_i(a_i)=\mu_i(\phi_i)$.

\begin{definition}
The \emph{MacLane depth} of an inductive valuation $\mu$ is the length $r$ of any optimal MacLane chain of $\mu$.  
\end{definition}

We end this section with several applications of the existence of MacLane chains.  

\begin{proposition}\label{KKv}
Let $\mu$ be an inductive valuation on $K_v(x)$. The restriction of $\mu$
to $K(x)$ is an inductive valuation with graded algebra isomorphic to $\ggm$. The mapping
$\Vi(K_v)\to \Vi(K)$ obtained in this way is bijective.  
\end{proposition}

\begin{proof}
Clearly, the restriction of the minimal valuation $\mu_0$ on $K_v(x)$ is the minimal valuation on $K(x)$. On the other hand, Proposition \ref{unicity} shows that every $\mu\in\Vi(K_v)$ admits an optimal MacLane chain whose key polynomials have coefficients in $K$; clearly, the inductive valuation on $K(x)$ determined by this optimal MacLane chain is the restriction $\mu_{\mid K(x)}$. Thus, the restriction of valuations induces a well-defined mapping $\Vi(K_v)\to \Vi(K)$. The statement about the graded algebras is obvious.

Conversely, an optimal MacLane chain of any $\mu\in \Vi(K)$ may be considered as an optimal MacLane chain of an inductive valuation $\tilde{\mu}$ on $K_v(x)$. By Proposition \ref{unicity} applied to both valuations $\mu$ and $\tilde{\mu}$, all optimal MacLane chains of $\mu$ determine the same valuation on $K_v(x)$. Therefore, we get a mapping  $\Vi(K)\to \Vi(K_v)$, which is the inverse of the restriction map. 
\end{proof}

\begin{proposition}\label{frf}
For any inductive valuation $\mu$, the canonical embedding $\Delta(\mu)\hookrightarrow \kappa(\mu)$ induces an isomorphism between the field of fractions of $\Delta(\mu)$ and $\kappa(\mu)$.
\end{proposition}

\begin{proof}
We must show that the natural morphism $\op{Frac}(\Delta(\mu))\to \kappa(\mu)$ is onto.
An element in $\kappa(\mu)^*$ is the class, modulo the maximal ideal of the valuation, of a fraction $g/h$ of polynomials with $\mu(g/h)=0$. Denote $\alpha=\mu(g)=\mu(h)\in\Gamma(\mu)$. If there exists a polynomial $f$  such that $\mu(f)=-\alpha$, then $\hm(fg),\hm(fh)$ belong to $\Delta(\mu)$ and the fraction $\hm(fg)/\hm(fh)$ is sent to the class of $g/h$ by the above morphism.

If $\mu=\mu_0$, then $\alpha\in\Z$ and there exists $f\in K$ with $\mu_0(f)=-\alpha$. If $\mu>\mu_0$, consider a MacLane chain of length $r>0$ of $\mu$, and let $-\alpha=m/e(\mu)$ for some $m\in\Z$. Since $\gcd(h_r,e_r)=1$, there exists an integer $s\ge0$ such that $m\equiv sh_r\md{e_r}$. Let $u=(m-sh_r)/e_r$ and take $\beta=u/e(\mu_{r-1})\in\Gamma(\mu_{r-1})$. By Corollary \ref{prescribed}, $\mu_{r-1}(a\phi_r^s)=\beta$ for some $a\in K[x]$ with $\deg a<\deg\phi_r$. Hence, $\mu(a\phi_r^s)=\beta+s\la_r=-\alpha$.   
\end{proof}

\begin{theorem}\label{preMLOk}
Let $\mu$ be an inductive valuation. For every monic $g\in K[x]$, we have $\mu(g)/\deg g\le C(\mu)$. Equality holds if and only if $g$ is $\mu$-minimal.
\end{theorem}

\begin{proof}
By induction on the length $r$ of a MacLane chain of $\mu$. For $r=0$, the statement is obvious because a monic polynomial $g$ has $\mu_0(g)\le0$, and $g$ is $\mu_0$-minimal if and only if it has coefficients in $\oo$. 

Let $r>0$ and suppose that $\mu_{r-1}(g)/\deg g\le C(\mu_{r-1})$ for all monic polynomials $g\in K[x]$. Let $g=\sum_{s=0}^\ell a_s\phi_r^s$ be the $\phi_r$-expansion of a monic polynomial $g$. If we denote $m_r=\deg\phi_r$, we have $\deg g=\deg a_\ell+\ell m_r$ and $\mu(g)\le \mu(a_\ell\phi_r^\ell)=\mu_{r-1}(a_\ell)+\ell m_rC(\mu)$.

If $\deg a_\ell=0$, we have $a_\ell=1$, because $g$ is monic. Hence, $\mu(g)\le \ell m_r C(\mu)=(\deg g) C(\mu)$. In this case, equality holds if and only if $\mu(g)=\mu(\phi_r^\ell)$, which is equivalent to $g$ being $\mu$-minimal, by Lemma \ref{minimal}. 

If $\deg a_\ell>0$, then $a_\ell$ is monic and $\mu_{r-1}(a_\ell)/\deg a_\ell\le C(\mu_{r-1})<C(\mu)$, by the induction hypothesis. Therefore,
\begin{equation}\label{rec}
\dfrac{\mu(g)}{\deg g}\le\dfrac{\mu_{r-1}(a_\ell)+\ell m_r C(\mu)}{\deg a_\ell +\ell m_r}<
\dfrac{C(\mu)\left(\deg a_\ell +\ell m_r\right)}{\deg a_\ell +\ell m_r}=C(\mu).
\end{equation}

In this case, the inequality is strict and $g$ is not $\mu$-minimal by Lemma \ref{minimal}. 
\end{proof}

\subsection{Numerical data of a MacLane chain}\label{subsecNum}
Let us fix an inductive valuation $\mu$ equipped with a Maclane chain of length $r$ as in (\ref{depth}). In this section and in sections \ref{subsecRat}, \ref{subsecR}, we attach to this chain several data and operators.  

Take $\phi_0:=x$, $\la_0:=0$ and $\mu_{-1}:=\mu_0$. We denote 
$$\Gamma_i=\Gamma(\mu_i)=e(\mu_i)^{-1}\Z, \quad \Delta_i=\Delta(\mu_i),\qquad 0\le i\le r.
$$
$$
\F_{-1}:=\F_0:=\op{Im}(\F\to\Delta_0); \quad \F_i:=\op{Im}(\Delta_{i-1}\to\Delta_i),\quad 1\le i\le r.
$$

By Proposition \ref{sameideal}, $\F_i$ is a field which may be identified with the residue class field $\F_{\phi_i}$ of the extension of $K_v$ determined by $\phi_i$; in particular, $\F_i$ is a finite extension of $\F$. We abuse of language and we identify $\F$ with $\F_0$ and each field $\F_i\subset\Delta_i$ with its image under the canonical map $ \Delta_i\to\Delta_j$ for $j\ge i$. In other words, we consider as inclusions the canonical embeddings
\begin{equation}\label{Fchain}
\F=\F_0\subset \F_1\subset\cdots\subset \F_r. 
\end{equation}

To these objects we attach several numerical data. For all $0\le i\le r$, we define:
$$
\as{1.2}
\begin{array}{l}
e_i:=e(\mu_i)/e(\mu_{i-1}), \\
f_{i-1}:=[\F_{i}\colon \F_{i-1}], \\
h_i:=e(\mu_i)\la_i,
\end{array}\qquad
\begin{array}{l}
m_i:=\deg\phi_i,\\
w_i:=\mu_{i-1}(\phi_i),\ V_i:=e(\mu_{i-1})w_i,\\ 
C_i:=C(\mu_i)=\mu_i(\phi_i)/\deg\phi_i,
\end{array}
$$
Note that $e_0=1$, $f_0=m_1$, $h_0=0$. Lemma \ref{groups} shows that $\gcd(h_i,e_i)=1$. All these data may be expressed in terms of the positive integers 
\begin{equation}\label{MLinvariants}
e_0,\dots,e_r, \ f_0,\dots,f_{r-1},\  h_1,\dots,h_r.
\end{equation}
In fact, the reader may easily check that for all $1\le i\le r$:
\begin{equation}\label{C}
\as{1.2}
\begin{array}{l}
e(\phi_i)=e(\mu_{i-1})=e_0\cdots e_{i-1},\\
f(\phi_i)=\left[\F_i\colon \F_0\right]=f_0\cdots f_{i-1},\\
\la_i=h_i/(e_0\cdots e_i),\\
m_i=e_{i-1}f_{i-1}m_{i-1}=(e_0\cdots e_{i-1})(f_0\cdots f_{i-1}),\\
w_i=e_{i-1}f_{i-1}(w_{i-1}+\la_{i-1})=\sum_{1\le j< i}(e_jf_j\cdots e_{i-1}f_{i-1})\la_j,\\
C_i=(w_i+\la_i)/m_i.
\end{array}
\end{equation}
The recurrence on $w_i$ is deduced from equation (\ref{rec}).

If the MacLane chain is optimal, Proposition \ref{unicity} shows that all these rational numbers are intrinsic data of $\mu$. In this case, we refer to them as $e_i(\mu)$, $f_i(\mu)$,
$h_i(\mu)$, $\la_i(\mu)$, $m_i(\mu)$, $w_i(\mu)$, $V_i(\mu)$, $C_i(\mu)$, respectively. The positive integers in (\ref{MLinvariants}) are then called the \emph{basic MacLane invariants} of $\mu$.
Also, the chain of $\F$-algebra homomorphisms $\Delta_0\to\cdots\to\Delta_r$ and the induced chain (\ref{Fchain}) of finite extensions of $\F$ are intrinsic objects attached to $\mu$.   

Clearly, the bijection $\Vi(K_v)\to\Vi(K)$ described in Proposition \ref{KKv} preserves all these invariants.

\subsection{Rational functions attached to a MacLane chain}\label{subsecRat}

For every $0\le i\le r$, we consider integers $\ell_i,\ell'_i$ uniquely determined by
$$
\ell_i h_i+\ell'_i e_i=1,\qquad 0\le \ell_i<e_i.
$$
We consider several rational functions in $K(x)$ defined in a recursive way. 

\begin{definition}\label{ratfracs}
We take $\pi_0=\pi_1=\pi$, $\Phi_0=\phi_0=\gamma_0=x$ and
$$
\Phi_i=\phi_i\,(\pi_{i})^{-V_i},\quad
\ga_i=(\Phi_i)^{e_i}(\pi_i)^{-h_i},\quad
\pi_{i+1}=(\Phi_{i})^{\ell_{i}}(\pi_{i})^{\ell'_{i}},\quad 1\le i\le r.
$$
\end{definition}

By construction, these rational functions may be expressed as $\pi^{n_0}(\phi_1)^{n_1}\cdots (\phi_r)^{n_r}$ for ade\-quate integers $n_j$. For $i\ge1$, it is easy to deduce from the definition that:
\begin{equation}\label{phis}
\as{1.2}
\begin{array}{ccl}
\Phi_i&=&\pi^{n_0}(\phi_1)^{n_1}\cdots (\phi_{i-1})^{n_{i-1}}\phi_i, \\
\pi_{i}&=&\pi^{n'_0}(\phi_1)^{n'_1}\cdots (\phi_{i-2})^{n'_{i-2}}(\phi_{i-1})^{\ell_{i-1}},\\
\gamma_i&=&\pi^{n''_0}(\phi_1)^{n''_1}\cdots (\phi_{i-1})^{n''_{i-1}}(\phi_i)^{e_i}.
\end{array}
\end{equation}

By Lemma \ref{stable}, $\mu_i(\phi_i)=
\mu_j(\phi_i)$ for all $1\le i\le j\le r$. Hence, (\ref{phis}) shows that 
\begin{equation}\label{stability}
\mu_i(\Phi_i)=\mu_{j}(\Phi_i),\quad \mu_i(\gamma_i)=\mu_{j}(\gamma_i),\quad \mu_{i}(\pi_{i+1})=\mu_{j}(\pi_{i+1}),\quad 1\le i\le j\le r.
\end{equation}

Let us compute these stable values.

\begin{lemma}\label{values}For every index $0\le i\le r$, we have
\begin{enumerate}
\item $\mu_{i}(\pi_{i})=1/e(\mu_{i-1})$, $\mu_{i}(\pi_{i+1})=1/e(\mu_{i})$.  
\item $\mu_{i-1}(\Phi_i)=0$, $\mu_{i}(\Phi_i)=\la_i$.
\item $\mu_{i}(\ga_i)=0$.
\end{enumerate}
Note that $\pi_{i+1}\in K(x)^*$ is a uniformizer of $\mu_{i}$.
\end{lemma}

\begin{proof}
We prove items 1, 2 by induction on $i$. For $i=0$ the statements are obvious. Suppose that $i>0$ and (1), (2) hold for a lower index. The identity $\mu_{i}(\pi_{i})=\mu_{i-1}(\pi_{i})=1/e(\mu_{i-1})$ is a consequence of (\ref{stability}).
$$
\as{1.2}
\begin{array}{l}
\mu_{i-1}(\Phi_i)=\mu_{i-1}(\phi_i)-V_i/e(\mu_{i-1})=w_i-w_i=0.\\
\mu_{i}(\Phi_i)=\mu_{i}(\phi_i)-V_i/e(\mu_{i-1})=w_i+\la_i-w_i=\la_i.\\
\mu_{i}(\pi_{i+1})=\ell_{i}\la_{i}+\ell'_{i}/e(\mu_{i-1})=1/e(\mu_{i}).
\end{array}
$$
The third item follows from the first two items. 
\end{proof}


By Lemma \ref{units}, the element $H_{\mu_{i}}(\phi_k)$ is a unit in $\gg(\mu_i)$ for all $k< i\le r$. Hence, by using (\ref{phis}), it makes sense to define, for all $0\le i\le r$:
$$\as{1.2}
\begin{array}{l}
x_i:=H_{\mu_i}(\Phi_i):=H_{\mu_i}(\pi)^{n_0}H_{\mu_i}(\phi_1)^{n_1}\cdots H_{\mu_i}(\phi_{i-1})^{n_{i-1}}H_{\mu_i}(\phi_i)\in \gg(\mu_i),\\
p_{i}:=H_{\mu_{i}}(\pi_i):=H_{\mu_i}(\pi)^{n'_0}H_{\mu_i}(\phi_1)^{n'_1}\cdots H_{\mu_i}(\phi_{i-2})^{n'_{i-2}}H_{\mu_i}(\phi_{i-1})^{\ell_{i-1}}\in \gg(\mu_i)^*,\\
y_i:=H_{\mu_i}(\gamma_i):=(x_i)^{e_i}(p_i)^{-h_i}\in\Delta_i.
\end{array}
$$

All factors of $x_i$ except for $H_{\mu}(\phi_i)$ are units in $\gg(\mu_i)$. Hence, these two elements generate the same ideal in $\gg(\mu_i)$. Let us emphasize this observation.

\begin{lemma}\label{associate}
For $0\le i \le r$, the elements $x_i$ and $H_{\mu_i}(\phi_i)$ are associate in $\gg(\mu_{i})$.\hfill{$\Box$}
\end{lemma}

Also, for $0\le i<r$ we define:
$$\as{1.2}
\begin{array}{l}
z_{i}\in\F_{i+1}, \ \mbox{ the image of $y_{i}$ under } \Delta_{i}\lra \Delta_{i+1},\\
\psi_{i}\in\F_{i}[y], \ \mbox{ minimal polynomial of $z_{i}$ over }\F_{i}.
\end{array}
$$

By Proposition \ref{sameideal}, $\op{Ker}(\Delta_i\to\Delta_{i+1})=\rr_{\mu_i}(\phi_{i+1})=H_{\mu_i}(\phi_{i+1})\gg(\mu_i)\cap \Delta_i$. For $i>0$, $\phi_{i+1}\nmid_{\mu_{i}}\phi_{i}$ implies that $H_{\mu_i}(\phi_i)$ has non-zero image in $\gg(\mu_{i+1})$. Therefore, $z_{i}\ne0$ for $i>0$, by Lemma \ref{associate}. In particular, $\psi_{i}\ne y$ for all $i>0$.
For $i=0$ we have $z_0=0$ (and $\psi_0=y$) if and only if $\phi_1\sim_{\mu_0} x$, or equivalently, $\overline{\phi}_1=\overline{x}$ in $\F[x]$.
We shall see in Corollary \ref{degpsi} that 
$$\F_{i+1}=\F_{i}[z_{i}]=\F_0[z_0,\dots,z_{i}],\qquad \deg\psi_{i}=f_{i}.
$$


\subsection{Operators attached to a MacLane chain}\label{subsecR} 
We consider Newton polygon operators 
$$N_i:=N_{\mu_{i-1},\phi_i}\colon\  K[x]\lra 2^{\R^2}, \quad 0\le i\le r,
$$ and residual polynomial operators:
$$
\as{1.3}
\begin{array}{rll}
R_{i,\alpha}\colon& \ppa(\mu_i)\lra \F_i[y],&\quad \ 0\le i\le r, \quad \alpha\in\Gamma_i,\\
R_{i}\colon& K[x]\lra \F_i[y],&\quad \ 0\le i\le r.
\end{array}
$$
The resi\-dual polynomial operators are defined by a recurrent formula involving certain constants $\ep_i(\alpha)\in\F_{i+1}^*$. Let us first define these constants.

Given $0\le i\le r$ and $\alpha\in\Gamma_i$, consider integers $s(\alpha)$, $u(\alpha)$ uniquely determined by 
$$(u(\alpha)/e(\mu_{i-1}))+s(\alpha)\la_i=\alpha, \qquad 0\le s(\alpha)<e_i,$$
or equivalently,
\begin{equation}\label{salpha}
u(\alpha)e_i+s(\alpha)h_i=e(\mu_i)\alpha, \qquad 0\le s(\alpha)<e_i.
\end{equation}
These integers $s(\alpha)$, $u(\alpha)$ depend on $i$, or more precisely on the group $\Gamma_i$. Note that $s(1/e(\mu_i))=\ell_i$, $u(1/e(\mu_i))=\ell'_i$, are the integers satisfying the B\'ezout identity $\ell_ih_i+\ell'_ie_i=1$, considered in section \ref{subsecRat}.

\begin{definition}\label{ep}
For $0\le i<r$ and $\alpha\in\Gamma_i$, we define 
$$\ep_i(\alpha)=
(z_{i})^{\ell'_is(\alpha)-\ell_iu(\alpha)}\in\F_{i+1}^*,
$$
where $s(\alpha)$, $u(\alpha)$ are the integers uniquely determined by (\ref{salpha}).

Note that $\ep_i(0)=1$ for all $0< i\le r$. For $i=0$, we get $\ep_0(\alpha)=(z_0)^0$ for all $\alpha\in\Z$. We convene that $\ep_0(\alpha)=1$, even in the case $z_0=0$.
\end{definition}

\begin{definition}\label{Rialpha}
For $0\le i\le r$, $\alpha\in\Gamma_i$, and $g=\sum_{0\le s}a_s\phi_i^s$ the $\phi_i$-expansion of $g\in\ppa(\mu_i)$, we define:
$$
R_{i,\alpha}(g)=\begin{cases}
\overline{g(y)/\pi^{\alpha}}\in\F_0[y], &\quad \mbox{if }i=0,\\
\sum\nolimits_{0\le j}\ep_{i-1}(\alpha_j)\,R_{i-1,\alpha_j}(a_{s_j})(z_{i-1})\,y^j\in\F_i[y], &\quad \mbox{if }i>0,
\end{cases}
$$
where $s_j:=s(\alpha)+je_i$ and $\alpha_j:=\alpha-s_j(w_{i}+\la_i)\in\Gamma_{i-1}$. 
\end{definition}

Let us explain the meaning of the data $s_j$, $\alpha_j$ involved in the computation of the $j$-th coefficient of $R_{i,\alpha}(g)$ (see Figure \ref{figAlphaj}).

\begin{figure}
\caption{Newton polygon $N_i(a_{s_j}\phi_i^{s_j})$}\label{figAlphaj}
\begin{center}
\setlength{\unitlength}{5mm}
\begin{picture}(12,9)
\put(0.75,6.8){$\times$}\put(2.2,6.05){$\times$}\put(3.7,5.3){$\times$}\put(5.2,4.55){$\times$}
\put(6.7,3.8){$\times$}\put(8.2,3.05){$\times$}\put(9.7,2.3){$\times$}\put(6.8,5.8){$\bullet$}
\put(-1,1){\line(1,0){13}}\put(0,0){\line(0,1){9}}
\put(-1,8){\line(2,-1){12}}
\multiput(7,.9)(0,.25){21}{\vrule height2pt}
\multiput(1.05,.9)(0,.25){25}{\vrule height2pt}
\put(7.4,4.1){\begin{footnotesize}$P_j$\end{footnotesize}}
\put(1.2,7.3){\begin{footnotesize}$P_0$\end{footnotesize}}
\put(6.9,.4){\begin{footnotesize}$s_j$\end{footnotesize}}
\put(11.2,2){\begin{footnotesize}$L_\alpha$\end{footnotesize}}
\put(6.8,6.5){\begin{footnotesize}$Q_{s_j}$\end{footnotesize}}
\put(-.5,.3){\begin{footnotesize}$0$\end{footnotesize}}
\put(-3.8,3.8){\begin{footnotesize}$u_j/e(\mu_{i-1})$\end{footnotesize}}
\put(-4.4,6.8){\begin{footnotesize}$u(\alpha)/e(\mu_{i-1})$\end{footnotesize}}
\put(.4,.3){\begin{footnotesize}$s(\alpha)$\end{footnotesize}}
\multiput(-.1,4)(.25,0){29}{\hbox to 2pt{\hrulefill }}
\multiput(-.1,7)(.25,0){5}{\hbox to 2pt{\hrulefill }}
\put(-.15,7.5){\line(1,0){.3}}
\put(-1,7.4){\begin{footnotesize}$\alpha$\end{footnotesize}}
\end{picture}
\end{center}
\end{figure}

For $1\le i\le r$, let $\cc_i=(\Z_{\ge0})\times \Gamma_{i-1}\subset\R^2$ be the set of points of the plane that may be vertexs of $N_i(g)$ for some $g\in K[x]$. 

Let $L_\alpha$ be the line of slope $-\la_i$ cutting the vertical axis at the point $(0,\alpha)$. The point $P_0=(s(\alpha),u(\alpha)/e(\mu_{i-1}))$ lies on $\cc_i\cap L_\alpha$ and it is the point with least abscissa in this set.
  Actually, the points on $\cc_i\cap L_\alpha$ may be parameterized as: 
$$P_j=(s(\alpha)+j e_i,(u(\alpha)-j h_i)/e(\mu_{i-1})), \quad j\in\Z_{\ge0}.
$$ 
We may write $P_j=(s_j,u_j/e(\mu_{i-1}))$, with $s_j=s(\alpha)+j e_i$, $u_j=u(\alpha)-j h_i\in\Z$.  

Also, since $u_j/e(\mu_{i-1})$ and $w_i$ (by definition) belong to $\Gamma_{i-1}$, we may consider $$\alpha_j=u_j/e(\mu_{i-1})-s_jw_i=\alpha-s_j(\la_i+w_i)\in\Gamma_{i-1}.
$$
 
Let $g=\sum_{0\le s}a_s\phi_i^s$ be the $\phi_i$-expansion of a polynomial $g\in\ppa(\mu_i)$. Denote $Q_s:=(s,\mu_{i-1}(a_s\phi_i^s))\in\cc_i$, so that $\{Q_s\mid 0\le s\}$ is the cloud of points whose lower convex hull is $N_i(g)$. 
  By Lemma \ref{muprimaN}, all $Q_s$ lie on or above the line $L_\alpha$, and 
$Q_s$ lies on $L_\alpha$ if and only if $\mu_i(a_s\phi_i^s)=\alpha$. Hence, 
\begin{equation}\label{onL}
\begin{cases}
s\not\in \{s_j\mid 0\le j\} &\imp\quad  \mu_i(a_s\phi_i^s)>\alpha,\\
\qquad s=s_j&\imp\quad \mu_i(a_s\phi_i^s)=\alpha \ \mbox{ if and only if }\mu_{i-1}(a_{s})=\alpha_j.
\end{cases} 
\end{equation}

The monomials of $R_{i,\alpha}(g)$ are in 1-1 correspondence with the points of $\cc_i\cap L_\alpha$. We shall see in Corollary \ref{nonzero} that the $j$-th coefficient of $R_{i,\alpha}(g)$ is non-zero if and only if $Q_{s_j}=P_j$. Let us now check that the $j$-th coefficient vanishes if $Q_{s_j}$ lies above $L_\alpha$.

\begin{lemma}\label{Rvanishes}
For all $0\le i\le r$, $\alpha\in\Gamma_i$, the operator $R_{i,\alpha}$ vanishes on $\ppa^+(\mu_i)$.
\end{lemma}

\begin{proof}
We proceed by induction on $i$. For $i=0$ the statement is clear. Assume $i>0$ and consider the $\phi_i$-expansion $g=\sum_{0\le s}a_s\phi_i^s$ of a polynomial $g$ with $\mu_i(g)>\alpha$. By (\ref{onL}), $\mu_{i-1}(a_{s_j})>\alpha_j$ for all $j\ge0$; hence, by the induction hypothesis, all $R_{i-1,\alpha_j}(a_{s_j})$ vanish (as polynomials in $\F_{i-1}[y]$) and all coefficients of $R_{i,\alpha}(g)$ vanish too.
\end{proof}

Take $g$ as above with $\mu_i(g)=\alpha$, and let $s(g)=s_{\mu_i}(g)\le s'(g)=s'_{\mu_i}(g)$ be the abscissas of the end points of the $\la_i$-component of $g$ (Definition \ref{sla}).
These end points belong to $\cc_i\cap L_\alpha$, so that $s(g)=s_{j_0}$ for $j_0=(s(g)-s(\alpha))/e_i=\lfloor s(g)/e_i\rfloor$, and $s'(g)=s_{j_0+d}$, where $d=(s'(g)-s(g))/e_i$ is called the \emph{degree} of the segment $S_{\la_i}$. 

\begin{figure}
\caption{Newton polygon of $g\in K[x]$ with $\mu_i(g)=\alpha$}\label{figAlpha}
\begin{center}
\setlength{\unitlength}{5mm}
\begin{picture}(14,10)
\put(0.7,6.8){$\times$}\put(2.2,6.05){$\times$}\put(3.8,5.3){$\bullet$}\put(5.2,4.55){$\times$}\put(6.7,3.8){$\times$}\put(8.3,3.05){$\bullet$}\put(9.8,2.25){$\times$}\put(6.85,5.8){$\bullet$}
\put(-1,1){\line(1,0){15}}\put(0,0){\line(0,1){10}}
\put(-1,8){\line(2,-1){12}}
\put(4,5.5){\line(-1,2){2}}\put(4,5.54){\line(-1,2){2}}
\put(4,5.5){\line(2,-1){4.5}}\put(4,5.54){\line(2,-1){4.5}}
\put(8.5,3.2){\line(4,-1){4}}\put(8.5,3.24){\line(4,-1){4}}
\multiput(4,.9)(0,.25){19}{\vrule height2pt}
\multiput(8.5,.9)(0,.25){10}{\vrule height2pt}
\multiput(1,.9)(0,.25){25}{\vrule height2pt}
\multiput(7,.9)(0,.25){21}{\vrule height2pt}
\put(1.2,7.3){\begin{footnotesize}$P_0$\end{footnotesize}}
\put(8.5,3.5){\begin{footnotesize}$P_{j_0+d}$\end{footnotesize}}
\put(4,5.8){\begin{footnotesize}$P_{j_0}$\end{footnotesize}}
\put(7.4,4.1){\begin{footnotesize}$P_j$\end{footnotesize}}
\put(6.8,6.5){\begin{footnotesize}$Q_{s_j}$\end{footnotesize}}
\put(6.8,.4){\begin{footnotesize}$s_j$\end{footnotesize}}
\put(8,.3){\begin{footnotesize}$s'(g)$\end{footnotesize}}
\put(3.4,.3){\begin{footnotesize}$s(g)$\end{footnotesize}}
\put(11.2,1.7){\begin{footnotesize}$L_\alpha$\end{footnotesize}}
\put(7,8.5){\begin{footnotesize}$N_i(g)$\end{footnotesize}}
\put(-.5,.3){\begin{footnotesize}$0$\end{footnotesize}}
\put(-4.3,5.3){\begin{footnotesize}$u(g)/e(\mu_{i-1})$\end{footnotesize}}
\put(-4.4,6.8){\begin{footnotesize}$u(\alpha)/e(\mu_{i-1})$\end{footnotesize}}
\put(.4,.3){\begin{footnotesize}$s(\alpha)$\end{footnotesize}}
\multiput(-.1,5.5)(.25,0){16}{\hbox to 2pt{\hrulefill }}
\multiput(-.1,7)(.25,0){5}{\hbox to 2pt{\hrulefill }}
\put(-.15,7.5){\line(1,0){.3}}
\put(-1,7.4){\begin{footnotesize}$\alpha$\end{footnotesize}}
\end{picture}
\end{center}
\end{figure}
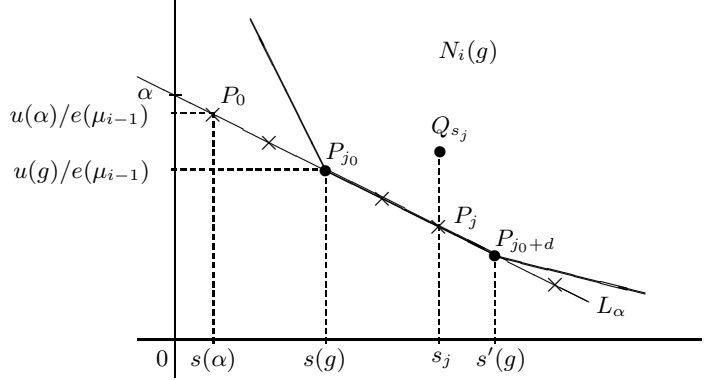

By Lemma \ref{Rvanishes}, the non-zero coefficients of $R_{i,\alpha}(g)$ correspond to abscissas $s_j$ with $j_0\le j\le j_0+d$ (see Figure \ref{figAlpha}). Hence, we may define the residual polynomial operator
$\;R_{i}\colon K[x]\lra \F_i[y]$ as follows.

\begin{definition}\label{Ri}
For $g\in K[x]$, $g\ne0$, let $\alpha=\mu_i(g)$. We define $R_0(g)=R_{0,\alpha}(g)=\overline{g(y)/\pi^{\alpha}}\in\F_0[y]$. For $1\le i\le r$, we define
$$
R_i(g):=    
R_{i,\alpha}(g)/y^{j_0}=\sum\nolimits_{0\le k\le d}\ep_{i-1}(\alpha_{j_0+k})\,R_{i-1,\alpha_{j_0+k}}(a_{s_{j_0+k}})(z_{i-1})\,y^k\in\F_i[y],
$$
where $j_0=\lfloor s(g)/e_i\rfloor$. We convene that $R_i(0)=0$ for all $0\le i\le r$.
\end{definition}

For any $1\le i\le r$ and any abscissa $s\ge0$, $N_i(\phi_i^s)$ is the point $(s,sw_i)$ and $\alpha:=\mu_i(\phi_i^s)=s(w_i+\la_i)$. Let $j=\lfloor s/e_i\rfloor=(s-s(\alpha))/e_i$. With the above notation, $s=s_j$ and $\alpha_j=0$. Since $\ep_{i-1}(0)=1=R_{i-1,0}(1)$, we have
\begin{equation}\label{LC}
R_{i,\alpha}(\phi_i^s)=y^{\lfloor s/e_i\rfloor},\qquad R_i(\phi_i^s)=1,\qquad 1\le i\le r.
\end{equation}

Corollary \ref{nonzero} below shows that $R_i(g)$ has always degree $d$ and $R_i(g)(0)\ne0$.
Also, Corollary \ref{astypes} shows that $R_i(\phi_{i+1})=\psi_i$ for all $0\le i<r$. 

\section{Structure of the graded algebra of an inductive valuation}\label{secGr}
In this section, we fix an inductive valuation $\mu$ equipped with a MacLane chain of length $r$, and we denote $\Delta=\Delta(\mu)$. We shall freely use all data and operators of the MacLane chain described in section \ref{secInductive}.

The main property of the residual polynomial operators is reflected in Theorem \ref{Hmug}.  We shall derive from this result some more properties of the residual polynomials, their link with the residual ideals, and the structure of the graded algebra of $\mu$.

\begin{lemma}\label{varphi}
For $0\le i\le r$, $\alpha\in\Gamma_i$ and a non-zero $g\in K[x]$, consider $$\varphi_i(\alpha):=x_i^{s(\alpha)}p_i^{u(\alpha)},\qquad 
\varphi_i(g):=x_i^{s(g)}p_i^{u(g)},$$ 
where $s(\alpha),\,u(\alpha)$ are defined in (\ref{salpha}), $(s(g),u(g)/e(\mu_{i-1}))$ is the left end point of $S_{\la_i}(g)$, if $i>0$ (see Figure \ref{figAlpha}), and $s(g)=0$, $u(g)=\mu_0(g)$, if $i=0$. 
These homogeneous elements in $\gg(\mu_i)$ have degree
$\deg\varphi_i(\alpha)=\alpha$, $\deg 
\varphi_i(g)= \mu_i(g)$.
Moreover, $\varphi_i(gh)=\varphi_i(g)\varphi_i(h)$
for any pair of non-zero polynomials $g,h\in K[x]$. 
\end{lemma}

\begin{proof}
The equalities $\mu_i\left(\Phi_i^{s(\alpha)}\pi_i^{u(\alpha)}\right)=\alpha$, $\mu_i\left(\Phi_i^{s(g)}\pi_i^{u(g)}\right)=\mu_i(g)$ are a consequence of Lemma \ref{values}. Corollary \ref{sumSla} shows that $\varphi_i(gh)=\varphi_i(g)\varphi_i(h)$. 
\end{proof}

\begin{theorem}\label{Hmug}
Let $g\in K[x]$ be a non-zero polynomial and let $\alpha=\mu(g)$. Then,
$$
H_{\mu}(g)=\varphi_r(\alpha)R_{r,\alpha}(g)(y_r)=\varphi_r(g)R_r(g)(y_r).
$$
In particular, $\ppa(\mu)/\ppa^+(\mu)=\varphi_r(\alpha)\Delta$ is a free $\Delta$-module of rank one.
\end{theorem}

\begin{proof}
Let $g=\sum_{0\le s}a_s\phi_r^s$ be the $\phi_r$-expansion of $g$, and consider the set of indices $I=\{s\ge0\mid \mu\left(a_s\phi_r^s\right)=\alpha\}$. Since $g\smu \sum_{s\in I}a_s\phi_r^s$, we have
 $\hm(g)=\sum_{s\in I}\hm(a_s\phi_r^s)$ by equation (\ref{Hmu}). 

Let us prove by induction on $r$ the identity 
\begin{equation} \label{aim}
H_{\mu}(g)=\varphi_r(\alpha)R_{r,\alpha}(g)(y_r). 
\end{equation}

If $r=0$, we have $\phi_0=x$, $e_0=1$, $s(\alpha)=0$ and $\varphi_0(\alpha)=H_{\mu_0}(\pi)^\alpha$. For all $s\in I$, we have $\mu_0(a_s)=\mu_0(a_sx^s)=\alpha$; thus, $b_s:=a_s\pi^{-\alpha}$ has $\mu_0(b_s)=0$, and
$$
H_{\mu_0}(a_sx^s)=\varphi_0(\alpha) H_{\mu_0}(b_s)\,y_0^s=\varphi_0(\alpha) \overline{b}_s\,y_0^s,
$$ 
the last equality by the identification of $\F$ with the subfield $\F_0\subset\Delta_0$. This proves (\ref{aim}) in this case.

Let now $i>0$ and suppose that (\ref{aim}) is true for all inductive valuations equipped with a MacLane chain of length less than $r$. By (\ref{onL}), $s\in I$ if and only if $s=s_j$ and $ \mu_{r-1}(a_{s_j})=\alpha_j$ for some $j\ge 0$. 
Thus,  (\ref{aim}) is equivalent to
$$
\hm(a_{s_j}\phi_r^{s_j})=\varphi_r(\alpha)\ep_{r-1}(\alpha_j)R_{r-1,\alpha_j}(a_{s_j})(z_{r-1})\,y_r^j,
$$
for all $j\ge 0$ such that $\mu_{r-1}(a_{s_j})=\alpha_j$.
Since 
\begin{equation}\label{varphiy}
\varphi_r(\alpha)\,y_r^j=x_r^{s(\alpha)}p_r^{u(\alpha)}\,y_r^j=x_r^{s(\alpha)+je_i}p_r^{u(\alpha)-jh_i}=
x_r^{s_j}p_r^{u_j}, 
\end{equation}
our aim is equivalent to showing that $\mu_{r-1}(a_{s_j})=\alpha_j$ implies: 
\begin{equation}\label{claim}
\hm(a_{s_j}\phi_r^{s_j})=x_r^{s_j}p_r^{u_j}\ep_{r-1}(\alpha_j)R_{r-1,\alpha_j}(a_{s_j})(z_{r-1}).
\end{equation}

By the induction hypothesis, if $\mu_{r-1}(a_{s_j})=\alpha_j$ we have
\begin{equation}\label{step}
H_{\mu_{r-1}}(a_{s_j})=\varphi_{r-1}(\alpha_j)R_{r-1,\alpha_j}(a_{s_j})(y_{r-1}). 
\end{equation}

Since $\deg a_{s_j}<\deg\phi_r$, we have $\mu(a_{s_j})=\mu_{r-1}(a_{s_j})$. Also, (\ref{stability}) implies that $\mu(\Phi_{r-1})=\mu_{r-1}(\Phi_{r-1})$, $\mu(\pi_{r-1})=\mu_{r-1}(\pi_{r-1})$. Hence, if we apply the canonical homomorphism $\gg(\mu_{r-1})\to\ggm$ to the identity  (\ref{step}), we get
\begin{equation}\label{ne0}
H_{\mu}(a_{s_j})=\hm(\Phi_{r-1})^{s(\alpha_j)}\hm(\pi_{r-1})^{u(\alpha_j)}R_{r-1,\alpha_j}(a_{s_j})(z_{r-1}). 
\end{equation}

Therefore, (\ref{claim}) is equivalent to
$$
\hm\left((\Phi_{r-1})^{s(\alpha_j)}(\pi_{r-1})^{u(\alpha_j)}\phi_r^{s_j}\right)=
\hm\left(\Phi_{r}^{s_j}\pi_{r}^{u_j}\right)\ep_{r-1}(\alpha_j),
$$
and this is a consequence of an identity between the involved rational functions which is proved in Lemma \ref{recurrence} below.

Finally, equality (\ref{varphiy}) applied to $j=j_0=(s(g)-s(\alpha))/e_r=\lfloor s(g)/e_r\rfloor$ yields $\varphi_r(\alpha)\,y_r^{j_0}=\varphi_r(g)$. Hence, $\varphi_r(\alpha)R_{r,\alpha}(g)(y_r)=\varphi_r(g)R_r(g)(y_r)$.
\end{proof}

\begin{lemma}\label{recurrence}
With the above notation, take $i\ge1$, $j\ge0$. Then, 
$$
(\Phi_{i-1})^{s(\alpha_j)}(\pi_{i-1})^{u(\alpha_j)}\,\phi_i^{s_j}=\Phi_i^{s_j}\pi_i^{u_j}(\gamma_{i-1})^{\ell'_{i-1}s(\alpha_j)-\ell_{i-1}u(\alpha_j)},
$$ 
where $\alpha_j$ is considered as an element in $\Gamma_{i-1}$.
\end{lemma}

\begin{proof}
Denote for simplicity $s=s_j$, $u=u_j$, $\bar{s}=s(\alpha_j)$, $\bar{u}=u(\alpha_j)$, $\ell=\ell_{i-1}$, $\ell'=\ell'_{i-1}$, $e=e_{i-1}$, $f=f_{i-1}$. 
The following identities are derived from the definitions of $\ga_{i-1}$, $\pi_i$, $\Phi_i$ and the B\'ezout identity $\ell h+\ell'e=1$.
\begin{align*}
(\Phi_{i-1})^{\bar{s}}(\pi_{i-1})^{\bar{u}}&\ (\gamma_{i-1})^{\ell \bar{u}-\ell'\bar{s}}= 
(\Phi_{i-1})^{\bar{s}+e(\ell \bar{u}-\ell'\bar{s})}(\pi_{i-1})^{\bar{u}-h(\ell \bar{u}-\ell'\bar{s})}\\=&\ 
(\Phi_{i-1})^{\ell(h\bar{s}+e\bar{u})}(\pi_{i-1})^{\ell'(h\bar{s}+e\bar{u})}= 
\pi_i^{h\bar{s}+e\bar{u}}=\pi_i^{u-sV_i}=(\Phi_i/\phi_i)^s\pi_i^u.
\end{align*}
In the last but one equality we used the identity $u-sV_i=\bar{u}e+\bar{s}h$, which is derived from $u/e(\mu_{i-1})-sw_i=\alpha_j=\bar{u}/e(\mu_{i-2})+\bar{s}\la_{i-1}$ by multiplying by $e(\mu_{i-1})$.
\end{proof}

\begin{corollary}\label{nonzero}Let $1\le i\le r$ and consider $g\in \ppa(\mu_i)$, $g\ne0$.
\begin{enumerate}
\item $R_{i,\alpha}(g)(y_i)=0$ if and only if $R_{i,\alpha}(g)=0$ if and only if $g\in\ppa^+(\mu_i)$. 
\item The $j$-th coefficient of $R_{i,\alpha}(g)$ is non-zero if and only if $\mu_{i-1}(a_{s_j})=\alpha_j$, or equivalently, the point $Q_{s_j}$ lies on $L_\alpha$ (see Figure \ref{figAlpha}).
\item $\deg R_{i,\alpha}(g)=\lfloor s'(g)/e_i\rfloor$ and $\ord_y(R_{i,\alpha}(g))=\lfloor s(g)/e_i\rfloor$.
\item $\deg R_i(g)=(s'(g)-s(g))/e_i$ and $R_i(g)(0)\ne0$.
\end{enumerate}
\end{corollary}

\begin{proof}
By Lemma \ref{Rvanishes}, $R_{i,\alpha}(g)=0$ if $g\in\ppa^+(\mu_i)$.
If $\mu_i(g)=\alpha$, Theorem \ref{Hmug} shows that $R_{i,\alpha}(g)(y_i)\ne0$ as an element in $\gg(\mu_i)$; thus, $R_{i,\alpha}(g)\ne0$. This proves item 1.

By Lemma \ref{Rvanishes}, the $j$-th coefficient of $R_{i,\alpha}(g)$ vanishes if $\mu_{i-1}(a_{s_j})>\alpha_j$. On the other hand,  $R_{i-1,\alpha_j}(a_{s_j})(z_{i-1})\ne0$ if $\mu_{i-1}(a_{s_j})=\alpha_j$, by equation (\ref{ne0}). This proves item 2.
Items 3 and 4 are a consequence of item 2.  
\end{proof}

\begin{corollary}\label{uniqueness}
For non-zero $g,h\in K[x]$, the following conditions are equivalent: 
\begin{enumerate}
\item $g\smu h$.
\item $\mu(g)=\mu(h)$ and $R_{r,\alpha}(g)=R_{r,\alpha}(h)$ for $\alpha=\mu(g)$.
\item $S_{\la_r}(g)=S_{\la_r}(h)$ and $R_{r}(g)=R_{r}(h)$.
\end{enumerate}
\end{corollary}

\begin{proof}
Conditions (1) and (2) are equivalent by Theorem \ref{Hmug} and  Corollary \ref{nonzero}. 
Conditions (2) and (3) are equivalent by Corollary \ref{nonzero}. 
\end{proof}

\begin{corollary}\label{RR}
For any non-zero $g\in K[x]$, let $\alpha=\mu(g)$. Then, 
$$\rr(g)=y_r^{\lceil s(\alpha)/e_r\rceil}R_{r,\alpha}(g)(y_r)\Delta=y_r^{\lceil s(g)/e_r\rceil}R_{r}(g)(y_r)\Delta.
$$
For $r=0$ we agree that $s(g)=0$.
\end{corollary}

\begin{proof}
By Theorem \ref{Hmug}, $\hm(g)=x_r^{s(\alpha)}p_r^{u(\alpha)}R_{r,\alpha}(g)(y_r)$. If $s(\alpha)=0$, then since $p_r$ is a unit, we have $\rr(g)=\hm(g)\ggm\cap\Delta=R_{r,\alpha}(g)(y_r)\Delta$. If $s(\alpha)>0$, then $\lceil s(\alpha)/e_r\rceil=1$ and equation (\ref{salpha}) shows that
$$s(-\alpha)=e_r-s(\alpha), \qquad u(-\alpha)=-h_r-u(\alpha).$$ 
A polynomial $h\in K[x]$ satisfies $\hm(gh)\in\Delta$ if and only if $\mu(h)=-\alpha$; in this case,  $\hm(h)=x_r^{e_r-s(\alpha)}p_r^{-h_r-u(\alpha)}R_{r,-\alpha}(h)(y_r)$, by Theorem \ref{Hmug}. Hence, 
$$\rr(g)= \left\{y_rR_{r,\alpha}(g)(y_r)R_{r,-\alpha}(h)(y_r)\mid \mu(h)=-\alpha\right\}\subset y_rR_{r,\alpha}(g)(y_r)\Delta.
$$
On the other hand, $\hm(g)x_r^{e_r-s(\alpha)}p_r^{-h_r}=y_rR_{r,\alpha}(g)(y_r)$, so that
$y_rR_{r,\alpha}(g)(y_r)$ belongs to $\rr(g)$, and $\rr(g)=y_rR_{r,\alpha}(g)(y_r)\Delta$.   

Finally, $y_r^{\lceil s(\alpha)/e_r\rceil}R_{r,\alpha}(g)(y_r)=y_r^{\lceil s(g)/e_r\rceil}R_{r}(g)(y_r)$, by (3) of Corollary \ref{nonzero}. 
\end{proof}

\begin{corollary}\label{Ralgebra}
Let $0\le i\le r$ and $\alpha\in\Gamma_i$. 
\begin{enumerate}
\item $R_{i,\alpha}(g+h)=R_{i,\alpha}(g)+R_{i,\alpha}(h)$ for all $g,h\in\ppa(\mu_i)$.
\item If $\beta\in \Gamma_{i-1}$, then $R_{i,\alpha+\beta}(gh)=R_{i,\alpha}(g)R_{i,\beta}(h)$ for all $g\in\ppa(\mu_i)$, $h\in\pset_\beta(\mu_i)$.
\item $R_i(gh)=R_i(g)R_i(h)$ for all $g,h\in K[x]$.
\end{enumerate}
\end{corollary}

\begin{proof}
For $i=0$ the identities are easy to check. If $i>0$,  
equation (\ref{Hmu}), Theorem \ref{Hmug}, and item 1 of Corollary \ref{nonzero} show that:
$$
\as{1.2}
\begin{array}{ll}
\varphi_i(\alpha)R_{i,\alpha}(g+h)=\varphi_i(\alpha)R_{i,\alpha}(g)+\varphi_i(\alpha)R_{i,\alpha}(h),&\ \mbox{ for }g,h\in\ppa(\mu_i).\\
\varphi_i(\alpha+\beta)R_{i,\alpha+\beta}(gh)=\varphi_i(\alpha)R_{i,\alpha}(g)\varphi_i(\beta)R_{i,\beta}(h),&\ \mbox{ for }g\in\ppa(\mu_i), \ h\in\pset_\beta(\mu_i).\\
\varphi_i(gh)R_i(gh)=\varphi_i(g)R_i(g)\varphi_i(h)R_i(h),&\ \mbox{ for }g,h\in K[x].
\end{array}
$$ 
The first equality proves item 1. The second equality proves item 2 because $s(\beta)=0$, and this leads to $s(\alpha+\beta)=s(\alpha)$, $u(\alpha+\beta)=u(\alpha)+u(\beta)$. 
By Lemma \ref{varphi}, the third equality proves item 3.
\end{proof}

By Corollary \ref{nonzero}, $R_{r,\alpha}$ induces an injective mapping $\rb_{r,\alpha}\colon \ppa(\mu)/\ppa^+(\mu)\lra \F_r[y]$.
 
\begin{theorem}\label{Delta}
The mapping $\rb_{r,0}\colon \Delta\lra \F_r[y]$ is an isomorphism of $\F_r$-algebras and   $\;(\rb_{r,0})^{-1}\colon \F_r[y]\lra\Delta$ is the $\F_r$-map determined by $y\mapsto y_r$. In particular, the element $y_r\in\Delta$ is transcendental over $\F_r$ and $\Delta=\F_r[y_r]$.
\end{theorem}

\begin{proof}
By Corollaries \ref{nonzero} and \ref{Ralgebra},  $\rb_{r,0}$ is an injective ring homomorphism. Let us check that its restriction to $\F_r\subset\Delta$ is the identity.

For $r=0$, an element $\xi\in\F_0^*$ is of the form $\xi=H_{\mu_0}(a)$ for some $a\in\oo^*$. By definition, $\rb_{0,0}(\xi)=R_{0,0}(a)=\overline{a}=H_{\mu_0}(a)=\xi$, modulo the identification $\F=\F_0$.

For $r>0$, Proposition \ref{extension} and Lemma \ref{units} show that an element $\xi\in\F_r^*$ is of the form $\xi=H_{\mu}(a)$ for some $a\in K[x]$ such that $\deg a<\deg\phi_r$ and $\mu_{r-1}(a)=\mu(a)=0$. The Newton polygon $N_r(a)$ is the single point $(0,0)$ and $R_{r,0}(a)\in\F_r^*$ is a degree zero polynomial. By Theorem \ref{Hmug}, $\xi=H_{\mu}(a)=R_{r,0}(a)=\rb_{r,0}(\xi)$.

By Corollary \ref{prescribed}, there exists $a\in K[x]$ such that $g=a\phi_r^{e_r}$ has Newton polygon $N_r(g)=\{(e_r,-e_r\la_r)\}$. Hence, $\mu(g)=0$ and $R_{r,0}(g)=\epsilon y$ for some $\epsilon\in\F_r^*$ (by Corollary \ref{nonzero}). Therefore,  $\rb_{r,0}$ is an onto map. 

The statement about $(\rb_{r,0})^{-1}$ is a consequence of Theorem \ref{Hmug} and $\varphi_r(0)=1$. 
\end{proof}

\begin{corollary}\label{degpsi}
For all $0\le i<r$, $\F_{i+1}=\F_{i}[z_{i}]=\F_0[z_0,\dots,z_{i}]$ and $\deg\psi_{i}=f_{i}$.
\hfill{$\Box$}
\end{corollary}

By Proposition \ref{frf} and Theorem \ref{Delta}, we get an isomorphism $\kappa(\mu)\simeq \F_r(y)$. In particular,  $\kappa(\mu)^{alg}\simeq \F_r$ and the next result follows.

\begin{corollary}\label{MLDelta}
For an inductive valuation $\mu$, the field $\kappa(\mu)^{alg}$ is a finite extension of $\F$ and $\kappa(\mu)\simeq \kappa(\mu)^{alg}(y)$, where $y$ is an indeterminate.\hfill{$\Box$}
\end{corollary}

\begin{corollary}\label{preconstruct}
The mapping $\rb_{r,\alpha}\colon \ppa(\mu)/\ppa^+(\mu)\lra \F_r[y]$ is bijective.
\end{corollary}

\begin{proof}
By Corollary \ref{nonzero}, $\rb_{r,\alpha}$ is injective. Let us show that it is onto. For any  non-zero polynomial $\psi\in\F_r[y]$, the element $\psi(y_r)\in\Delta$ is non-zero by Theorem \ref{Delta}; hence, $\varphi_r(\alpha)\psi(y_r)=\hm(g)$ for some $g\in K[x]$ with $\mu(g)=\alpha$. By Theorem \ref{Hmug}, $R_{r,\alpha}(g)(y_r)=\psi(y_r)$, and this implies $R_{r,\alpha}(g)=\psi$, by Theorem \ref{Delta}.  
\end{proof}

\begin{corollary}\label{construct}
Let $\psi\in\F_r[y]$ be a monic polynomial of degre $f$ such that $\psi(0)\ne0$. Then, for any $\alpha\in\Gamma(\mu)$ there exists $g\in K[x]$ monic such that $\deg g=e_rf m_r$, $\mu(g)=e_rf(w_r+\la_r)$ and $R_r(g)=\psi$.
\end{corollary}

\begin{proof}
Denote $\alpha:=e_rf(w_r+\la_r)$. By Corollary  \ref{preconstruct}, there exists $g_0\in K[x]$ with $\mu(g_0)=\alpha$ and $R_{r,\alpha}(g_0)=\psi-y^f$. By dropping all terms with abscissa $s\ge e_rf$ from the $\phi_r$-expansion of $g_0$, we may assume that $\deg g_0<e_rf m_r$. Then, $g=\phi_r^{e_rf}+g_0$ satisfies what we want. In fact, $\deg (g)$, $\mu(g)$ are the right ones, and 
$R_{r,\alpha}(g)=R_{r,\alpha}(\phi_r^{e_rf})+R_{r,\alpha}(g_0)=\psi$, by the first item of Corollary \ref{Ralgebra} and equation (\ref{LC}).  Since $R_{r,\alpha}(g)(0)=\psi(0)\ne0$, we have $R_{r,\alpha}(g)=R_{r}(g)$.
\end{proof}

Corollary \ref{construct} is crucial for the computational applications of inductive valuations. It yields a routine for the construction of key polynomials with prescribed residual ideal.

\begin{theorem}\label{structure}
We get an isomorphism of graded $\F_r$-algebras
$$
\ggm=\bigoplus\nolimits_{\alpha\in\Gamma(\mu)}\varphi_r(\alpha)\Delta\simeq \F_r[y,p,p^{-1}][x],
$$
where $y,p$ are indeterminates and $x$ is an algebraic element satisfying $x^{e_r}=p^{h_r}y$. As elements in the graded algebra, these elements are homogeneous of degree $\deg y=0$, $\deg p=1/e(\mu_{r-1})$, $\deg x=\la_r$ .
\end{theorem}

\begin{proof}
By sending $y\mapsto y_r$,  $p\mapsto p_r$,  $x\mapsto x_r$, we get an onto $\F_r$-homomorphism  of graded algebras: $\F_r[y,p,p^{-1}][x]\twoheadrightarrow \ggm$. In order to show that it is an isomorphism we need only to check that $y_r,p_r$ are algebraically independent over $\F_r$, and the algebraic equation of $x_r$ over $\F_r[y,p,p^{-1}]$ has minimal degree.

Let us prove that the family $\Sigma:=\{y_r^m p_r^n\mid m\in\Z_{\ge0},\ n\in\Z\}$ is linearly independent over $\F_r$. We may group these elements by its degree:
$$
\Sigma=\bigcup\nolimits_{\alpha\in\Gamma(\mu)}\Sigma_\alpha,\qquad
\Sigma_\alpha=\{y_r^m p_r^{e(\mu_{r-1})\alpha}\mid m\in\Z_{\ge0}\}.
$$  
The families $\Sigma_\alpha$ are all $\F_r$-linearly independent because $y_r$ is transcendental over $\F_r$. Therefore, $\Sigma$ is also $\F_r$-linearly independent because a linear combination of its elements vanishes if and only if each homogeneous component vanishes.

The minimality of the equation $x_r^{e_r}=p_r^{h_r}y_r$ is a consequence of $\gcd(h_r,e_r)=1$.  
\end{proof}

\begin{corollary}\label{primes}
Let $\psi\in \F_r[y]$ such that $\psi(0)\ne0$. Then, $\psi(y_r)\in\Delta$ is a prime element in $\ggm$ if and only if $\psi$ is irreducible in $\F_r[y]$.  
\end{corollary}

\begin{proof}
If $\psi(y_r)$ is a prime element in $\ggm$, then it is a prime element in $\Delta$
and Theorem \ref{Delta} shows that $\psi$ is irreducible.

Conversely, if $\psi$ is irreducible, consider $\F'=\F_r[y]/(\psi)$ and denote by $z\in\F'$ the class of $y$. By Theorem \ref{structure}, $\ggm/\psi(y_r)\ggm\simeq \F'[p,p^{-1}][x]$, where $p$ is an indeterminate and $x$ satisfies $x^{e_r}=p^{h_r}z$. Since $\psi(0)\ne0$, we have $z\ne0$ and $\F'[p,p^{-1}][x]$ is an integral domain. Hence,  $\psi(y_r)\ggm$ is a prime ideal.
\end{proof}


\section{Canonical decomposition of the set of key polynomials}\label{secKP} 
Let $\mu$ be an inductive valuation and denote $\Delta=\Delta(\mu)$. In this section we want to study the fibers of the mapping: 
$$\rr\colon \kpm\lra \mx(\Delta),\qquad \phi\mapsto \rr(\phi)=\op{Ker}(\Delta\to \F_\phi).
$$
That is, we want to describe the partition:
$$
\kpm=\bigcup\nolimits_{\ll\in\mx(\Delta)}\kpm_\ll,\qquad \kpm_\ll:=\left\{\phi\in\kpm\mid \rr(\phi)=\ll\right\}.
$$
It is hard to analyze these subsets from a purely abstract perspective. Thus, we suppose that $\mu$ is equipped with a fixed MacLane chain of length $r$.
We shall freely use all data and operators of the MacLane chain described in section \ref{secInductive}. 

Also, for a non-zero $g\in K[x]$ we denote by $s(g)=s_\mu(g)\le s'(g)=s'_\mu(g)$ the abscissas of the end points of the $\la_r$-component of $g$ (cf. Definition \ref{sla}).

\subsection{Further properties of key polynomials}
Let us first obtain criterions for $\mu$-irreduci\-bility and for being a key polynomial, in terms of $\phi_r$-expansions.

\begin{lemma}\label{muirred}
A polynomial $g\in K[x]$ is $\mu$-irreducible if and only if either:
\begin{itemize}
\item $H_{\mu}(g)$ and $H_{\mu}(\phi_r)$ are associate elements in $\ggm$, or 
\item $s(g)=0$ and $R_r(g)$ is irreducible in $\F_r[y]$.
\end{itemize}
The first condition is equivalent to $s(g)=s'(g)=1$.
\end{lemma}

\begin{proof}
By Lemma \ref{associate}, $x_r$ is associate in $\ggm$ to the prime element $H_{\mu}(\phi_r)$.
On the other hand,  $H_{\mu}(g)=x_r^{s(g)}p_r^{u(g)}R_r(g)(y_r)$, by Theorem \ref{Hmug}.  Since $p_r$ is a unit, $H_{\mu}(g)$ is a prime element if and only if either:
\begin{itemize}
\item $s(g)=1$ and $R_r(g)(y_r)$ is a unit, or
\item $s(g)=0$ and $R_r(g)(y_r)$ is a prime element.
\end{itemize}
By Theorem \ref{Delta}, the first condition is equivalent to $s(g)=1$ and $\deg R_r(g)=0$, which is equivalent to $s(g)=s'(g)=1$, by Corollary \ref{nonzero}. Also, this holds if and only if $H_{\mu}(g)$ and $H_{\mu}(\phi_r)$ are associate. 
By Corollary \ref{primes}, the second condition is equivalent to
$s(g)=0$ and $R_r(g)$ irreducible in $\F_r[y]$.
\end{proof}

\begin{definition}\label{one-sided}
For a non-zero $g\in K[x]$, we say that $N_{\mu,\phi}(g)$ is \emph{one-sided of slope} $-\la$ if $N_{\mu,\phi}(g)=S_\la(g)$, $s(g)=0$ and $s'(g)>0$.  
\end{definition}

\begin{lemma}\label{kp}
A monic polynomial $g\in K[x]$ belongs to $\kpm$ if and only if either:
\begin{enumerate}
\item $\deg g=m_r$ and $g\smu\phi_r$, or
\item $s(g)=0$, $\deg g=s'(g)m_r$ and $R_r(g)$ is irreducible in $\F_r[y]$.
\end{enumerate}
In the last case, $\deg g=e_r(\deg R_r(g))m_r$, $N_r(g)$ is one-sided of slope $-\la_r$, and $R_r(g)\in \F_r[y]$ is monic.
\end{lemma}

\begin{proof}
A polynomial $g$ satisfying (1) is a key polynomial by Lemma \ref{mid=sim}. 
A polynomial $g$ satisfying (2) is a key polynomial by the criteria of Lemmas \ref{minimal} and \ref{muirred}.

Conversely, suppose $g$ is a key polynomial. By  Lemma \ref{minimal}, $\deg g=s'(g)m_r$. By Lemma \ref{muirred}, either $s(g)=s'(g)=1$, or $s(g)=0$ and $R_r(g)$ is irreducible. 

In the first case, we have $\deg g=m_r$ and the component $S_{\la_r}(g)$ is a single point with abscissa $s=1$. This implies that $g=\phi_r+a$ with $\deg a<m_r$ and $\mu(a)>\mu(\phi_r)$, by Lemma \ref{muprimaN}. Thus, $g$ satisfies (1). 

In the second case, $g$ satisfies (2), which clearly implies $N_r(g)=S_{\la_r}(g)$. By Corollary \ref{nonzero}, $s'(g)=e_r\deg R_r(g)$. Hence, $N_r(g)$ is one-sided of slope $-\la_r$. The polynomial $R_r(g)$ is monic by equation (\ref{LC}) and item 1 of Corollary \ref{Ralgebra}.
\end{proof}

The next result is a consequence of Theorem \ref{Hmug}, Corollary \ref{RR} and Lemma \ref{kp}.

\begin{corollary}\label{Hmuphi}
For any $\phi\in\kpm$, we have:
$$\as{1.2}\qquad  
\begin{array}{lll}
\hm(\phi)=\hm(\phi_r)=x_rp_r^{V_r},& \rr(\phi)=y_r\Delta,& \mbox{ if }\phi\smu\phi_r,\\
\hm(\phi)=p_r^{(e_rV_r+h_r)\deg R_r(\phi)}R_r(\phi)(y_r),& \rr(\phi)=R_r(\phi)(y_r)\Delta,&\mbox{ if }\phi\not\smu\phi_r.\qquad\Box
\end{array}
$$ 
\end{corollary}

\begin{corollary}\label{identification}
For $\phi\in\kpm$, take
$\psi=R_r(\phi)$, if $\phi\not\smu\phi_r$, and $\psi=y$, if $\phi\smu\phi_r$. Then,
 under the isomorphism $\Delta\simeq \F_r[y]$ determined by $\rb_{r,0}$, the maximal ideal $\rr(\phi)$ is mapped to $\psi\F_r[y]$. Thus, $\F_\phi\simeq \F_r[y]/(\psi)$ and $f(\phi)=f_0\cdots f_{r-1}\deg\psi$.\hfill{$\Box$}
\end{corollary}

\begin{corollary}\label{astypes}For all $0\le i<r$, 
\begin{enumerate}
\item $N_i(\phi_{i+1})$ is one-sided of slope $-\la_i$. 
\item $R_i(\phi_{i+1})=\psi_i$, the minimal polynomial of $z_i$ over $\F_i$.
\end{enumerate}
\end{corollary}

\begin{proof}
The polynomial $\phi_{i+1}$ is a key polynomial for $\mu_i$ and $\phi_{i+1}\not\sim_{\mu_i}\phi_i$. Hence, it satifies (2) of Lemma \ref{kp}. This proves item 1.

By Corollary \ref{Hmuphi}, $H_{\mu_i}(\phi_{i+1})$ is associate to $R_i(\phi_{i+1})(y_i)$ in $\gg(\mu_i)$; hence, its image under the canonical homomorphism $\gg(\mu_i)\to\gg(\mu_{i+1})$ is associate to $R_i(\phi_{i+1})(z_i)$
in $\gg(\mu_{i+1})$. This implies that $R_i(\phi_{i+1})(z_i)=0$ because $H_{\mu_i}(\phi_{i+1})$ belongs to the kernel of $\gg(\mu_i)\to\gg(\mu_{i+1})$, by Proposition \ref{extension}. Since $R_i(\phi_{i+1})$ is monic and irreducible (Lemma \ref{kp}), we have $R_i(\phi_{i+1})=\psi_i$.       
\end{proof}

\subsection{Analysis of the mapping $\kpm\to\mx(\Delta)$}

\begin{proposition}\label{samefiber}
Let $\phi,\phi'\in\kpm$. The following conditions are equivalent:
\begin{enumerate}
\item $\rr(\phi)=\rr(\phi')$.
\item $R_r(\phi)=R_r(\phi')$.
\item $\phi\smu\phi'$.
\item $\hm(\phi)$ and $\hm(\phi')$ are associate in $\ggm$. 
\item $\phi\mmu\phi'$. 
\end{enumerate}
\end{proposition}

\begin{proof}
By equation (\ref{LC}) and Lemma \ref{kp}, $R_r(\phi)=1$ if $\phi\smu\phi_r$, and $R_r(\phi)$ is monic, irreducible, and different from $y$ (because $R_r(\phi)(0)\ne0)$) otherwise. Therefore, Corollary \ref{Hmuphi} and Theorem \ref{Delta} show that (1), (2) and (3) are equivalent. Clearly, (3) implies (4), and (4) implies (5). Finally, (5) implies $\rr(\phi')\subset \rr(\phi)$, and this implies (1), because $\rr(\phi')$ is a maximal ideal.    
\end{proof}

The analysis of the key polynomials provided by the use of a MacLane chain yields an intrinsic description of the mapping $\rr\colon \kpm\to\mx(\Delta)$.   

\begin{theorem}\label{Max}
Let $\mu$ be an inductive valuation. The mapping $\rr\colon \kpm\to\mx(\Delta)$ induces a bijection  
between $\kpm/\!\smu$ and $\mx(\Delta)$.
\end{theorem}

\begin{proof}
By Proposition \ref{samefiber}, for any $\ll\in\mx(\Delta)$, the fiber $\kpm_\ll$ is either empty or it is one of the classes of the equivalence relation $\smu$ on the set $\kpm$. Thus, $\rr$ induces an injective mapping $\kpm/\!\smu\,\lra\mx(\Delta)$. 

Let us show that the residual ideal mapping $\rr$ is onto. A maximal ideal $\ll$ in $\Delta$
corresponds to a monic irreducible polynomial $\psi\in\F_r[y]$, under the isomorphism $\Delta\simeq\F_r[y]$ of Theorem \ref{Delta}. If $\psi=y$, then $\ll=\rr(\phi_r)$, by Corollary \ref{Hmuphi}. If $\psi\ne y$, then there exists a monic polynomial $\phi\in K[x]$ of degree $\deg\phi=e_r(\deg\psi) m_r$ such that $R_r(\phi)=\psi$, by item 2 of Corollary \ref{construct}. As a general fact, $\deg \phi\ge s'(\phi)m_r$. By Corollary \ref{nonzero}, $s'(\phi)-s(\phi)=e_r\deg\psi$; thus:
$$
\deg \phi\ge s'(\phi)m_r\ge (s'(\phi)-s(\phi))m_r=e_r(\deg\psi) m_r=\deg\phi.
$$
Hence, $s(\phi)=0$ and $\deg\phi=s'(\phi)m_r$. Therefore, $\phi$ satisfies condition (2) of Lemma \ref{kp}, and it is a key polynomial for $\mu$. By Corollary \ref{Hmuphi}, $\rr(\phi)=\psi(y_r) \Delta=\ll$.
\end{proof}


\begin{corollary}\label{homogeneousprimes}
Let $\pset\subset\kpm$ be a set of representatives of key polynomials under $\mu$-equivalence. 
Then, the set $H\pset=\left\{\hm(\phi)\mid \phi\in \pset\right\}$ is a system of representatives of homogeneous prime elements of $\ggm$ up to associates in the algebra.  Moreover, up to units in $\ggm$, for any non-zero $g\in K[x]$, there is a unique factorization:
\begin{equation}\label{factorization}
g\smu \prod\nolimits_{\phi\in \pset}\phi^{a_\phi},\quad  a_\phi=\ord_{\mu,\phi}(g).
\end{equation}
\end{corollary}

\begin{proof}
All elements in $H\pset$ are homogeneous prime elements by the definition of $\mu$-irreducibi\-lity, and they are pairwise non-associate by Proposition \ref{samefiber}. By Lemma \ref{muirred}, every homogeneous prime element is associate either to $\hm(\phi_r)$ or to $\psi(y_r)$ for some irreducible polynomial $\psi\in\F_r[y]$. The proof of Theorem \ref{Max} and Corollary \ref{Hmuphi} show that $\psi(y_r)$ is associate to an element in $H\pset$.
Finally, every homogeneous element in $\ggm$ is a product of homogeneous prime elements, by Theorem \ref{Hmug} and Corollary \ref{primes}. This implies the unique factorization (\ref{factorization}).
\end{proof}

\subsection{Proper and strong key polynomials}
Theorems \ref{Max} and \ref{Delta} yield bijections $$\kpm/\!\!\smu\,\lra\mx(\Delta)\lra\P(\F_r),$$ where $\P(\F_r)$ denotes the set of monic irreducible polynomials with coefficients in $\F_r$. 
The first bijection is canonical, but the second one depends, in principle, on the choice of a MacLane chain of $\mu$.

The class of $\phi_r$ is mapped to $y \in \P(\F_r)$ under the composition of the above bijections, and it has special properties when the prime ideal $x_r\ggm$ is ramified over the subalgebra $\Delta[p_r,p_r^{-1}]$.
In this section we analyze to what extent the bijection $\kpm/\!\!\smu\,\lra \P(\F_r)$
depends on the chosen MacLane chain for $\mu$, and the distinguished ``bad" class is intrinsic. 

Recall that  the numerical data attached to any optimal MacLane chain of $\mu$ are intrinsic data of $\mu$, denoted  $e_i(\mu)$, $f_i(\mu)$, $h_i(\mu)$, $m_i(\mu)$, $w_i(\mu)$, $\la_i(\mu)$, $C_i(\mu)$ (cf. section \ref{subsecNum}). We may formulate two different intrinsic distinctions between key polynomials, according to their degree.

\begin{definition}\label{strong}
Let $\mu$ be an inductive valuation of depth $r$, and let $\phi\in\kpm$. 

We say that $\phi$ is a \emph{proper key polynomial} for $\mu$ if $\deg\phi$ is a multiple of $e_r(\mu)m_r(\mu)$. 

We say that $g\in K[x]$ is  \emph{$\mu$-proper} if $\phi\nmid_\mu g$ for all improper key polynomials $\phi$. 

We say that $\phi$ is a \emph{strong key polynomial} for $\mu$ if $r=0$ or $\deg\phi>m_r(\mu)$. 

We denote $\kpp$, $\kps$ the sets of proper and strong key polynomials for $\mu$, respectively. 

\end{definition}

By Lemma \ref{kp}, $\kps\subset\kpp\subset\kpm$. If $e_r(\mu)=1$, all key polynomials are proper and all polynomials are $\mu$-proper. If $e_r(\mu)>1$, there is a single improper $\mu$-equivalence class of key polynomials, distinguished by the property $\deg\phi=m_r(\mu)$; all other key polynomials are proper and strong. Note that $\kp(\mu_0)^{\op{str}}=\kp(\mu_0)$. 

In every MacLane chain, all $\phi_i$ are proper key polynomials for $\mu_{i-1}$, by Lemma \ref{kp}.

\begin{lemma}\label{choice}
Suppose that in the given MacLane chain of $\mu$, we replace $\phi_r$ by $\phi_r'=\phi_r+a$ for some $a\in K[x]$ such that $\deg a<m_r$ and $\mu(a)\ge \mu(\phi_r)$. For some $\alpha\in \Gamma(\mu)$, denote by $R'_{r,\alpha}$ the residual polynomial operator attached to the new MacLane chain of $\mu$ obtained in this way.  
\begin{enumerate}
\item If $\phi'_r\smu\phi_r$, then $R'_{r,\alpha}=R_{r,\alpha}$. 
\item We may choose $a$ such that $\phi'_r\not\smu\phi_r$ if and only if $e_r=1$. For such a choice,
$R'_{r,\alpha}(g)(y)=R_{r,\alpha}(g)(y-\eta)$ for all $g\in \ppa(\mu)$, where $\eta=R_r(a)\in\F_r^*$.
\end{enumerate}
\end{lemma}

\begin{proof}
 For any $g=\sum_{0\le s}a_s\phi_r^s\in \ppa(\mu)$, each term $a_s\phi_r^s$ belongs to $\ppa(\mu)$ and 
$$
R_{r,\alpha}(g)=\sum\nolimits_{0\le s}R_{r,\alpha}(a_s\phi_r^s),\quad
R'_{r,\alpha}(g)=\sum\nolimits_{0\le s}R'_{r,\alpha}(a_s\phi_r^s),
$$ 
by Corollary \ref{Ralgebra}. Hence, it is sufficient to compare the action of both operators on polynomials of the form $g=b\phi_r^s$ with $\deg b<\deg\phi_r$ and $\mu(g)=\alpha$. 
On the other hand, if $\beta=\mu(b)=\mu_{r-1}(b)$ and $\gamma=\alpha-\beta=\mu(\phi_r^s)$, Corollary \ref{Ralgebra} shows that 
$$
R_{r,\alpha}(b\phi_r^s)=R_{r,\beta}(b)R_{r,\gamma}(\phi_r^ s), \quad
R'_{r,\alpha}(b\phi_r^s)=R'_{r,\beta}(b)R'_{r,\gamma}(\phi_r^ s).
$$
Since $R_{r,\beta}(b)=\ep_{r-1}(\beta)R_{r-1,\beta}(b)(z_{r-1})=R'_{r,\beta}(b)\in\F_r^*$, we need only to compare $R_{r,\gamma}(\phi_r^s)$ with $R'_{r,\gamma}(\phi_r^s)$.

If $\phi'_r\smu\phi_r$, then Corollary \ref{uniqueness} and equation (\ref{LC}) show that $R'_{r,\gamma}(\phi_r^s)=R'_{r,\gamma}((\phi_r')^s)=y^{\lfloor s/e_r \rfloor}=R_{r,\gamma}(\phi_r^s)$. This proves item 1.

If $e_r>1$, then $\la_r\not\in \Gamma_{r-1}$ and $\mu(\phi_r)=\mu_{r-1}(\phi_r)+\la_r\not\in \Gamma_{r-1}$. Hence,  $\mu(\phi_r)\ne\mu(a)=\mu_{r-1}(a)$ for any $a\in K[x]$ with $\deg a<\deg\phi_r$; thus, $\phi_r\smu\phi'_r$.
If $e_r=1$, then $\Gamma(\mu)=\Gamma_{r-1}$, and the proof of Lemma \ref{groups} shows that $\mu(\phi_r)=\mu(a)$ for some $a\in K[x]$ with $\deg a<\deg\phi_r$. 

Finally, suppose that  $e_r=1$ and $\phi'_r\not\smu\phi_r$. By Corollary \ref{Ralgebra}, $R'_{r,\gamma}(\phi_r^s)=R'_{r,\delta}(\phi_r)^s$, where $\delta=\mu(\phi_r)$. Since $R_{r,\gamma}(\phi_r^s)=y^s$, we need only to show that $R'_{r,\delta}(\phi_r)=y-R_r(a)$. In fact,
$$
R'_{r,\delta}(\phi_r)=R'_{r,\delta}(-a+\phi'_r)=-R'_{r,\delta}(a)+R'_{r,\delta}(\phi'_r)=-R_{r,\delta}(a)+y.
$$
Since $\mu(a)=\mu(\phi_r)=\delta$, Corollary \ref{nonzero} shows that $R_{r,\delta}(a)=R_r(a)\in\F_r^*$. 
\end{proof}


Therefore, in the case $e_r(\mu)>1$, the improper class of key polynomials always corresponds to $y\in\P(\F_r)$. In the case $e_r(\mu)=1$, we have $\phi_r\leftrightarrow y$, but an adequate choice of the MacLane chain changes the 1-1 correspondence via $\psi(y)\leftrightarrow\psi(y-\eta)$ in $\P(\F_r)$, for a certain $\eta\in\F_r^*$.
Thus, in this case, for any given $\phi\in\kpm$, we may always find a MacLane chain for $\mu$ such that $\phi\not\smu\phi_r$.

\begin{corollary}\label{criterionproper}
A key polynomial $\phi$ for $\mu$ is proper if and only if there exists a MacLane chain of $\mu$ such that $\phi\not\smu\phi_r$, where $r$ is the length of the chain.\hfill{$\Box$}
\end{corollary}

\begin{lemma}\label{multproper}
For non-zero $g,h\in K[x]$ with $g$ $\mu$-proper, we have $\rr(gh)=\rr(g)\rr(h)$. 
\end{lemma}

\begin{proof}
Denote $e=e_r(\mu)$. By Corollary \ref{RR} and Theorem \ref{Delta}, $\rr(gh)=\rr(g)\rr(h)$ is equivalent to the following equality, up to factors in $\F_r^*$:
$$
y^{\lceil s(gh)/e\rceil}R_{r}(gh)=
y^{\lceil s(g)/e\rceil}R_{r}(g)
y^{\lceil s(h)/e\rceil}R_{r}(h).
$$
By Lemma \ref{additivity}, $s(gh)=s(g)+s(h)$, and by Corollary \ref{Ralgebra}, $R_{r}(gh)=R_{r}(g)R_{r}(h)$. Thus, we want to show that 
\begin{equation}\label{end}
\lceil (s(g)+s(h))/e\rceil=
\lceil s(g)/e\rceil+\lceil s(h)/e\rceil.
\end{equation}
If $e=1$ this equality is obvious. If $e>1$, we have $x_r\nmid \hm(g)$, because $g$ is $\mu$-proper. By Theorem \ref{Hmug}, $x_r^{s(g)}\mid \hm(g)$, so that $s(g)=0$ and (\ref{end}) is obvious too. 
\end{proof}

\begin{proposition}\label{nextlength}
Let $\phi\in\kpm$ and $\ll=\rr(\phi)$. For any non-zero $g\in K[x]$:
$$
\ord_{\ll}(\rr(g))=
\begin{cases}
\ord_{\mu,\phi}(g),&\mbox{ if $\phi$ is proper},\\
\lceil\ord_{\mu,\phi}(g)/e_r(\mu)\rceil,&\mbox{ if $\phi$ is improper}. 
\end{cases}
$$
where $\ord_\ll(\rr(g))$ is the largest non-negative integer $n$ such that $\ll^n\mid \rr(g)$.
\end{proposition}

\begin{proof}
Denote $a_\phi=\ord_{\mu,\phi}(g)$: If we apply $\rr$ to both terms of the factorization  (\ref{factorization}), Lemma \ref{multproper} shows that:
$$
\rr(g)=\rr\left(\prod\nolimits_{\phi\in \pset}\phi^{a_\phi}\right)=\prod\nolimits_{\phi\in \pset}\rr(\phi^{a_\phi}).
$$
For all proper $\phi\in\pset$ we have $\rr(\phi^{a_\phi})=\rr(\phi)^{a_\phi}$, by Lemma \ref{multproper}. For the improper $\phi\in\pset$ (if $e_r(\mu)>1$), we have $\rr(\phi^{a_\phi})=\rr(\phi)^{\lceil a_\phi/e_r(\mu)\rceil}$ by
Corollary \ref{RR}, equation (\ref{LC}) and Corollary \ref{Hmuphi}.
\end{proof}

The next result follows from Proposition \ref{nextlength} and Corollary \ref{Hmuphi}.

\begin{corollary}\label{nextlength2}
Let $\phi$ be a proper key polynomial for $\mu$ and denote $\psi=R_r(\phi)$. Then, $\ord_\psi(R_r(g))=\ord_{\mu,\phi}(g)$ for any non-zero $g\in K[x]$.\hfill{$\Box$} 
\end{corollary}

\section{MacLane-Okutsu invariants of prime polynomials}\label{secOkutsu} 
In this section, we apply inductive valuations $\mu$ on $K(x)$ to polynomials in $K_v[x]$, without any explicit mention to the natural extension of $\mu$ to $K_v(x)$ described in Proposition \ref{KKv}.   

\subsection{Prime polynomials and inductive valuations}
\begin{definition}\label{prime}
Let $\P=\P(\oo_v)\subset\oo_v[x]$ be the set of all monic irreducible polynomials in $\oo_v[x]$. We say that an element in $\P$ is a \emph{prime polynomial} with respect to $v$. 
\end{definition}

Let $F\in\P$ be a prime polynomial and fix $\t\in\kb$ a root of $F$. Let $K_F=K_v(\t)$ be the finite extension of $K_v$ generated by $\t$, $\oo_F$ the ring of integers of $K_F$, $\m_F$ the maximal ideal and $\F_F$ the residue class field. 
We have $\deg F=e(F)f(F)$, where $e(F)$, $f(F)$ are the ramification index and residual degree of $K_F/K_v$, respectively.

In coherence with section \ref{subsecRideal}, we denote by $\mu_{\infty,F}$ the pseudo-valuation on $K[x]$ defined by $\mu_{\infty,F}(g)=v(g(\t))$ for any $g\in K[x]$.
 
\begin {lemma}\label{symmetry}
Let $F,F'\in\P$ be two prime polynomials, and let $\t,\t'\in\kb$ be roots of $F,F'$, respectively. Then, $v(F(\t'))/\deg(F)=v(F'(\t))/\deg(F')$. 
\end {lemma}

\begin{proof}
The value $v(F(\t'))$ does not depend on the choice of the root $\t'$; hence,
$$
\deg(F')v(F(\t'))=v(\res(F,F'))=\deg(F)v(F'(\t)),
$$ 
because $\res(F,F')=\prod_{\t\in Z(F)}F'(\t)=\pm\prod_{\t'\in Z(F')}F(\t')$, where $Z(F)$ is the multiset of roots of $F$ in $\kb$, with due count of multiplicities if $F$ is inseparable.
\end{proof}

We are interested in finding properties of prime polynomials leading to a certain comprehension of the structure of the set $\P$. An inductive valuation $\mu$ admitting a key polynomial $\phi$ such that $\phi\mmu F$ reveals many properties of $F$. 

\begin{theorem}\label{fundamental}
Let $F\in\P$ be a prime polynomial and $\t\in\kb$ a root of $F$. Let $\mu$ be an inductive valuation and $\phi$ a key polynomial for $\mu$. Then, $\phi\mmu F$ if and only if
$v(\phi(\t))>\mu(\phi)$. Moreover, if this condition holds, then:
\begin{enumerate}
\item Either $F=\phi$, or the Newton polygon $N_{\mu,\phi}(F)$ is one-sided of slope $-\la$, where $\la=v(\phi(\t))-\mu(\phi)\in\Q_{>0}$.    
\item Let $\ell=\ell(N_{\mu,\phi}(F))$. Then, $\deg F=\ell\deg \phi$ and $F$ is $\mu$-minimal.
\item  $F\smu\phi^\ell$, so that $\rr(F)$ is a power of the maximal ideal $\rr(\phi)$.  
\end{enumerate}
\end{theorem}

\begin{proof}
If $F=\phi$, then both conditions $\phi\mmu F$ and $v(\phi(\t))>\mu(\phi)$ hold. 

If $F\ne\phi$, consider the minimal polynomial $g(x)=\sum_{j=0}^kb_jx^j\in \oo_v[x]$ of $\phi(\t)$ over $K_v$.
All roots of $g(x)$ in $\kb$ have $v$-value equal to $\delta:=v(\phi(\t))\ge0$; hence,
$$
v(b_0)=k\delta,\quad v(b_j)\ge (k-j)\delta, \ 1\le j<k,\quad v(b_k)=0.
$$
Let us denote $N:=N_{\mu,\phi}$. These conditions imply that the Newton polygon $N(G)$ of the polynomial $G(x)=g(\phi(x))=\sum_{j=0}^kb_j\phi^j$ is one-sided of slope $\mu(\phi)-\delta=-\la$. Since $G(\t)=0$, the polynomial $F$ is a factor of $G$ and Theorem \ref{product} shows that
\begin{equation}\label{sum}
N^-(G)=N^-(F)+N^-(G/F).
\end{equation}
 
Now, if $\phi\mmu F$, then Lemma \ref{length2} shows that $\ell(N^-(F))=\ord_{\mu,\phi}(F)>0$; hence, $N^-(G)$ has positive length too, and $\la$ must be a positive rational number. Conversely, if $\la>0$, then $N(G)=N^-(G)$ and we have
$$
\ell(N(F))+\ell(N(G/F))\le\ell(N(G))=
\ell(N^-(G))=\ell(N^-(F))+\ell(N^-(G/F)).
$$
This implies $N(F)=N^-(F)$ and since $N^-(G)$ is one-sided of slope $-\la$, (\ref{sum}) shows that  
$N(F)$ is one-sided of slope $-\la$ too. This proves that $\phi\mmu F$ if and only if $\la>0$, and also that item 1 holds in this case. 

\begin{figure}
\caption{Newton polygon $N_{\mu,\phi}(F)$}\label{figNF}
\begin{center}
\setlength{\unitlength}{5mm}
\begin{picture}(6,7.5)
\put(8.8,2.8){$\bullet$}\put(-.2,5.8){$\bullet$}
\put(-1,1){\line(1,0){11}}\put(0,0){\line(0,1){7}}
\put(0,6){\line(3,-1){9}}
\multiput(9,.9)(0,.25){9}{\vrule height2pt}
\multiput(-.1,3)(.25,0){37}{\hbox to 2pt{\hrulefill }}
\put(-5.2,2.85){\begin{footnotesize}$\mu(F)=\mu(a_\ell\phi^\ell)$\end{footnotesize}}
\put(-7,5.8){\begin{footnotesize}$\mu(a_0)=\mu(a_\ell\phi^\ell)+\ell\la$\end{footnotesize}}
\put(-.45,.3){\begin{footnotesize}$0$\end{footnotesize}}
\put(8.85,.3){\begin{footnotesize}$\ell$\end{footnotesize}}
\put(4.5,4.5){\begin{footnotesize}$-\la$\end{footnotesize}}
\end{picture}
\end{center}
\end{figure}
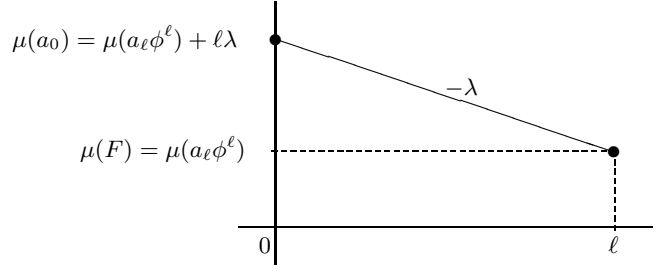

Let $F=\sum_{s=0}^\ell a_s\phi^s$ be the $\phi$-expansion of $F$. Let $\theta_\phi\in\kb$ be a root of $\phi$; by Lemma \ref{symmetry}, $v(\phi(\t))/\deg \phi=v(F(\theta_\phi))/\deg F=v(a_0(\theta_\phi))/\deg F$. On the other hand, since $\deg a_0<\deg\phi$, Proposition \ref{theta} shows that  $\mu(a_0)=v(a_0(\theta_\phi))$. Therefore, item 1 and a look at Figure \ref{figNF} show that
\begin{align*}
 \mu(a_\ell)+\ell(\mu(\phi)+\la)=\;&\mu(a_0)=v(a_0(\theta_\phi))=\dfrac{\deg F}{\deg \phi}v(\phi(\t))=\dfrac{\deg F}{\deg \phi}(\mu(\phi)+\la)\\=\;&\dfrac{\deg a_\ell+\ell\deg \phi}{\deg \phi}(\mu(\phi)+\la)=\left(\dfrac{\deg a_\ell}{\deg\phi}+\ell\right)(\mu(\phi)+\la).
\end{align*}
If $\deg a_\ell>0$, then $a_\ell$ would be a monic polynomial satisfying an inequality that contradicts Theorem \ref{preMLOk}:
$$
\dfrac{\mu(a_\ell)}{\deg a_\ell}=
\dfrac{\mu(\phi)+\la}{\deg \phi}>
\dfrac{\mu(\phi)}{\deg \phi}.
$$ 
Therefore, $a_\ell=1$ and $\deg F=\ell\deg \phi$. Also, $\mu(F)/\deg F=\mu(\phi^\ell)/\deg F=\mu(\phi)/\deg\phi$; thus, $F$ is $\mu$-minimal by Theorem \ref{preMLOk}. This proves item 2.

Item 3 follows from $\mu(F)=\mu(\phi^\ell)<\mu(a_s\phi^s)$ for all $s<\ell$. 
\end{proof}

\begin{corollary}\label{lower}
With the above notation, suppose that $\phi\mmu F$ and $\mu$ admits a MacLane chain of length $r$ as in (\ref{depth}) such that $\phi\not\smu\phi_r$. Then, for any $1\le i\le r$, the Newton polygon $N_i(F)$ is one-sided of slope $-\la_i$, we have $\mu(\phi_i)=v(\phi_i(\t))$ and
\begin{equation}\label{elli}
F\sim_{\mu_{i-1}}\phi_{i}^{\ell_{i}}, \qquad \deg F=\ell_i\deg \phi_{i}, \qquad R_{i-1}(F)=(\psi_{i-1})^{\ell_i},
\end{equation}
where $\ell_i:=\ell(N_i(F))$. 
In particular, $\ell_{i}=e_if_i\ell_{i+1}$ for all $1\le i< r$.
\end{corollary}

\begin{proof}
Since $\phi\not\smu\phi_r$, Corollary \ref{Hmuphi} shows that $\deg R_r(\phi)>0$. Since $F\smu \phi^\ell$, we have $R_r(F)=R_r(\phi)^\ell$ by Corollaries \ref{uniqueness} and \ref{Ralgebra}; hence,  $\ell(N^-_r(F))\ge e_r\deg R_r(F)>0$, and this implies that $\phi_r\mid_{\mu_{r-1}}F$ by Lemma \ref{length2}.
Therefore, $\phi_i\mid_{\mu_{i-1}} F$ for all $1\le i\le r$, and (\ref{elli}) is a consequence of Theorem \ref{fundamental} and Corollaries \ref{uniqueness}, \ref{Ralgebra} and \ref{astypes}. 

We have $F\ne\phi_i$ and the slope of $N_i(F)$ is $-\la_i$, because otherwise $R_i(F)$ would be a constant, leading by Corollary \ref{nextlength2}
to  $\phi_{i+1}\nmid_{\mu_{i}} F$ for $i<r$, or to $\phi\nmid_\mu F$ for $i=r$. Finally, $\mu_i(\phi_i)-\mu_{i-1}(\phi_i)=\la_i=v(\phi_i(\t))-\mu_{i-1}(\phi_i)$ by Theorem \ref{fundamental}. Hence, $\mu(\phi_i)=\mu_i(\phi_i)=v(\phi_i(\t))$ by Lemma \ref{stable}.
\end{proof}

If $F\ne\phi$, we may extend the given MacLane chain to a MacLane chain of length $r+1$ of the valuation $\mu'=[\mu;(\phi,\la)]$ just by taking $\phi_{r+1}=\phi$, $\la_{r+1}=\la$.
$$
\mu_0\ \stackrel{(\phi_1,\la_1)}\lra\  \mu_1\ \stackrel{(\phi_2,\la_2)}\lra\ \cdots
\ \stackrel{(\phi_{r},\la_{r})}\lra\ \mu_{r}=\mu\stackrel{(\phi_{r+1},\la_{r+1})}\lra\mu_{r+1}=\mu'.
$$
Since $s_{\mu'}(F)=0$ and $s'_{\mu'}(F)=\ell$, Corollary \ref{nonzero} shows that $\deg R_{r+1}(F)=\ell/e_{r+1}>0$. Let $\psi$ be an irreducible factor of $R_{r+1}(F)$ in $\F_{r+1}[y]$. By Theorem \ref{Max}, there exists $\phi'\in\kp(\mu')$ such that $R_{r+1}(\phi')=\psi$. Since $R_{r+1}(F)(0)\ne0$, we have $\psi\ne y$ and $\phi'\not\sim_{\mu'}\phi$. Also, $\deg\phi'=e_{r+1}\deg\psi \deg\phi$ by Lemma \ref{kp},
and $\phi'\mid_{\mu'}F$ by Corollary \ref{nextlength2}. 
By Theorem \ref{fundamental}, $F\sim_{\mu'}(\phi')^{\ell'}$ for $\ell'=\ell/(e_{r+1}\deg\psi)$. By Corollary \ref{uniqueness}, $R_{r+1}(F)=\psi^{\ell'}$. This procedure may be iterated as long as $F\ne\phi'$.

These ideas of MacLane are the germ of an algorithm to compute approximations to $F$ by prime polynomials with coefficients in $\oo$, with prescribed precision. We shall discuss the relevant computational aspects of this algorithm in \cite{gen}.  

We now deduce from Theorem \ref{fundamental} the fundamental result concerning factorization of polynomials over $K_v$. It has to be considered as a generalization of Hensel's lemma. 

\begin{definition}\label{degree}\mbox{\null}
We define the \emph{degree} of $\ll\in\mx(\Delta)$ as $\deg\ll=\dim_{\F_r}(\Delta/\ll)$.
\end{definition}

Equivalently, $\deg\ll=\deg \psi$ for the unique monic irreducible polynomial $\psi\in \F_r[y]$ such that $\ll=\psi(y_r) \Delta$.   

\begin{theorem}\label{main}
Let $\mu$ be an inductive valuation and let $\phi$ be a proper key polynomial for $\mu$. Then, every monic polynomial $g\in\oo_v[x]$ factorizes into a product of monic polynomials in $\oo_v[x]$:
$$
g=g_0\,\phi^{\ord_\phi(g)}\prod\nolimits_{(\la,\ll)} g_{\la,\ll},$$
where $-\la$ runs on the slopes of $N_{\mu,\phi}^-(g)$. For each $\la$, if $\mu_\la=[\mu;(\phi, \la)]$, then $\ll$ runs on the maximal ideals of $\Delta(\mu_\la)$ dividing $\rr_{\mu_\la}(g)$. If $e(\mu)\la=h_\la/e_\la$, with $h_\la,e_\la$ positive coprime integers, then
$$\deg g_0=\deg g-\ell(N^-_{\mu,\phi}(g))\deg\phi,\quad
\deg g_{\la,\ll}=e_\la\ord_\ll(\rr_{\mu_\la}(g))\deg\ll\deg\phi.
$$ 
Moreover, if $\ord_\ll(\rr_{\mu_\la}(g))=1$, then $g_{\la,\ll}$ is irreducible in $\oo_v[x]$. 
\end{theorem}
 
\begin{proof}
Let $g=F_1\cdots F_t$ be the factorization of $g$ into a product of monic irreducible polynomials in $\oo_v[x]$. Denote $\ell_j:=\ell\left(N^-_{\mu,\phi}(F_j)\right)=\ord_{\mu,\phi}(F_j)$ (Lemma \ref{length2}). The factor $g_0$ is the product of all $F_j$ satisfying $\phi\nmid_\mu F_j$. The factors $F_j$ with $\phi\mmu F_j$ have $\deg F_j=\ell_j \deg \phi$, by Theorem \ref{fundamental}. By Theorem \ref{product}, $N_{\mu,\phi}^-(g)=\sum_{j}N_{\mu,\phi}^-(F_j)$; hence,
$$
 \deg g-\deg g_0=\sum_{\phi\mmu F_j}\deg F_j=\sum_{\phi\mmu F_j}\ell_j\deg \phi=\sum_{j}\ell_j\deg \phi=\ell\left(N^-_{\mu,\phi}(g)\right)\deg\phi.
$$
   
The factor $\phi^{\ord_\phi(g)}$ is the product of all $F_j$ equal to $\phi$. 
By Theorem \ref{fundamental}, for the factors $F_j\ne\phi$ such that $\phi\mmu F_j$, the Newton polygon $N_{\mu,\phi}(F_j)$ is one-sided of slope $-\la$, and Theorem \ref{product} shows that $-\la$ is one of the slopes of $N^-_{\mu, \phi}(g)$.  Along the discussion previous to Theorem \ref{main}, we saw that these $F_j$ are $\mu_\la$-proper and 
$$\rr_{\mu_\la}(F_j)=\ll^{\ell'_j}, \qquad \deg F_j=e_\la\ell'_j\deg\ll\deg \phi,$$
where  $\ll$ is a certain maximal ideal in $\Delta(\mu_\la)$ and $\ell'_j=\ell_j/(e_\la\deg\ll)$. Also, since $s_{\mu_\la}(F_j)=0$, Lemma \ref{additivity} shows that $\phi\nmid_{\mu_\la} F_j$, so that $\ll\ne\rr_{\mu_\la}(\phi)$, by Proposition \ref{nextlength}. 
Now, for a given pair $(\la,\ll)$ we take $g_{\la,\ll}$ to be the product of all $F_j$ such that $N_{\mu,\phi}(F_j)$ is one-sided of slope $-\la$ and $\rr_{\mu_\la}(F_j)$ is a power of $\ll$. Let $J_{\la,\ll}$ be the set of all indices $j$ of the irreducible factors $F_j$ of  $g_{\la,\ll}$. 

We claim that $\ll\nmid \rr_{\mu_\la}(F_j)$ for all $j\not\in J_{\la,\ll}$. In fact, since  $\ll\ne\rr_{\mu_\la}(\phi)$, the statement is clear for the factors $F_j$ equal to $\phi$. If $\phi\nmid_\mu F_j$, or $N_{\mu,\phi}(F_j)$ is one-sided of a slope lower than $-\la$, then $F_j\sim_{\mu_\la} a_0$, where $a_0$ is the $0$-th term of the $\phi$-expansion of $F_j$ (see Figure \ref{figSlopes}); hence, $H_{\mu_\la}(F_j)$ is a unit and $\rr_{\mu_\la}(F_j)=1$. If $N_{\mu,\phi}(F_j)$ is one-sided of a slope larger than $-\la$, then $F_j\sim_{\mu_\la} \phi^{\ell_j}$ (see Figure \ref{figSlopes}); by Proposition \ref{nextlength}, $\rr_{\mu_\la}(F_j)$ is a power of $\rr_{\mu_\la}(\phi)$ and it is not divided by $\ll$. 

Therefore, from the equality $\rr_{\mu_\la}(g)=\prod_j \rr_{\mu_\la}(F_j)$ of Lemma \ref{multproper}, we deduce 
\begin{align*}
\ord_\ll \rr_{\mu_\la}(g)=&\ \sum\nolimits_j \ord_\ll \rr_{\mu_\la}(F_j)=\sum\nolimits_{j\in J_{\la,\ll}} \ord_\ll \rr_{\mu_\la}(F_j)\\
=&\ \sum\nolimits_{j\in J_{\la,\ll}}\deg F_j/(e_\la\deg\ll\deg \phi)=\deg g_{\la,\ll}/(e_\la\deg\ll\deg \phi).
\end{align*}
Finally, if $\ord_\ll \rr_{\mu_\la}(g)=1$, there is only one irreducible factor $F_j$ dividing $g_{\la,\ll}$.
\end{proof}

\begin{figure}\caption{$\lambda$-component of the irreducible factor $F_j$.}\label{figSlopes}
\begin{center}
\setlength{\unitlength}{5mm}
\begin{picture}(21,6.5)
\put(-.2,4.3){$\bullet$}\put(5.8,2.85){$\bullet$}
\put(-1,0.5){\line(1,0){9}}\put(0,-.5){\line(0,1){6.5}}
\put(7,1){\line(-2,1){8}}
\put(0,4.5){\line(4,-1){6}}\put(0,4.52){\line(4,-1){6}}
\put(2,2.3){\begin{footnotesize}$-\la$\end{footnotesize}}
\multiput(6,.4)(0,.25){11}{\vrule height2pt}
\put(5.8,-.2){\begin{footnotesize}$\ell_j$\end{footnotesize}}
\put(2.2,4.5){\begin{footnotesize}$N_{\mu,\phi}(F_j)$\end{footnotesize}}
\put(-.45,-.1){\begin{footnotesize}$0$\end{footnotesize}}
\put(11.8,5.25){$\bullet$}\put(17.8,1.3){$\bullet$}
\put(11,0.5){\line(1,0){9}}\put(12,-.5){\line(0,1){6.5}}
\put(19,1){\line(-2,1){8}}
\put(18,1.5){\line(-3,2){6}}\put(18,1.52){\line(-3,2){6}}
\put(14,2.5){\begin{footnotesize}$-\la$\end{footnotesize}}
\multiput(18,.4)(0,.25){5}{\vrule height2pt}
\put(17.8,-.2){\begin{footnotesize}$\ell_j$\end{footnotesize}}
\put(14.4,4.3){\begin{footnotesize}$N_{\mu,\phi}(F_j)$\end{footnotesize}}
\put(11.55,-.1){\begin{footnotesize}$0$\end{footnotesize}}
\end{picture}
\end{center}
\end{figure}
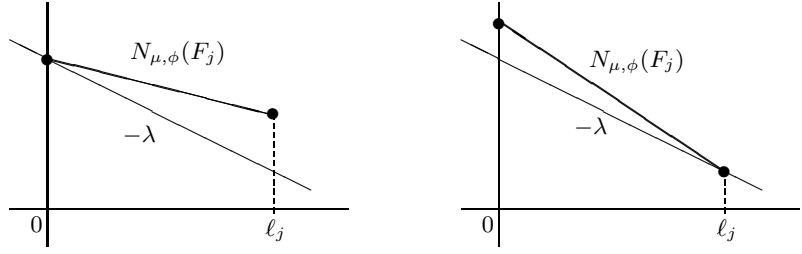

\begin{theorem}\label{computation}
 Let $F\in \P$ be a prime polynomial.
An inductive valuation $\mu$ satisfies $\mu\le\mu_{\infty,F}$ if and only if there exists $\phi\in\kpm$ such that $\phi\mmu F$. In this case, for a non-zero polynomial $g\in K[x]$, we have
\begin{equation}\label{equality}
\mu(g)=\mu_{\infty,F}(g) \quad\mbox{if and only if} \quad \phi\nmid_{\mu}g. 
\end{equation}
\end{theorem}

\begin{proof}
If $\mu\le\mu_{\infty,F}$, we may consider $\phi\in K[x]$ monic with minimal degree among all polynomials satisfying $\mu(\phi)<\mu_{\infty,F}(\phi)$. By Lemma \ref{minDegree}, $\phi$ is a key polynomial for $\mu$, and condition (\ref{equality}) is satisfied. In particular,  $\phi\mmu F$.   

Conversely, suppose that $\phi\mmu F$ for some $\phi\in\kpm$. If $F=\phi$, the statement of the theorem is proved in Proposition \ref{theta}; thus, we may assume $F\ne\phi$. If we show that $\mu\le\mu_{\infty,F}$, then there
exists $\phi'\in \kpm$ such that $\phi'\mmu F$ and (\ref{equality}) is  satisfied for $\phi'$. By Theorem \ref{fundamental}, $F\smu \phi^\ell$ for some $\ell>0$, so that $\phi'\mmu \phi$, and this implies $\phi\smu\phi'$ by Proposition \ref{samefiber}. Hence, $\phi$ satisfies (\ref{equality}) as well.

Let us prove the inequality $\mu\le\mu_{\infty,F}$ by induction on the depth $r$ of $\mu$. 
If $r=0$, then clearly $\mu=\mu_0\le\mu_{\infty,F}$.
Suppose $r>0$ and the statement true for all valuations with lower depth. Consider an optimal MacLane chain of $\mu$. 

Let $g\in K[x]$ with $\phi_r$-expansion $g=\sum_{0\le s}a_s\phi_r^s$. Since $\phi_r\nmid_{\mu_{r-1}}a_s$, we have $\mu(a_s)=\mu_{r-1}(a_s)=\mu_{\infty,F}(a_s)$ by Proposition \ref{extension} and the induction hypothesis. Thus, we need only to show that $\mu(\phi_r)\le\mu_{\infty,F}(\phi_r)$, because then 
$$
\mu_{\infty,F}(g)\ge \min_{0 \le s}\left\{\mu_{\infty,F}(a_s\phi_r^s)\right\}\ge\min_{0 \le s}\left\{\mu(a_s\phi_r^s)\right\}=\mu(g).
$$

If $\phi\smu\phi_r$, then $\phi_r\mmu F$ and Theorem \ref{fundamental} shows that $v(\phi_r(\t))>\mu(\phi_r)$. If $\phi\not\smu\phi_r$, then $v(\phi_r(\t))=\mu(\phi_r)$ by Corollary \ref{lower}. 
\end{proof}

Theorem \ref{computation} may be applied as a device for the computation of $\mu_{\infty,F}$. 
Given $g\in K[x]$, we find a pair $(\mu,\phi)$ such that $\phi\mmu F$ and $\phi\nmid_\mu g$, leading to $v(g(\t))=\mu(g)$. From a computational perspective, the condition $\phi\nmid_{\mu}g$ is checked as $R_r(\phi)\nmid R_r(F)$, with respect to a MacLane chain for $\mu$. This yields a very efficient routine for the computation of the $\p$-adic valuations $v_\p\colon K^*\to\Z$, with respect to prime ideals $\p$ in a number field $K$ \cite{newapp}.

\begin{corollary}\label{efphiF}
With the above notation, let $\t_\phi\in\kb$ be a root of $\phi$.
\begin{enumerate}
\item For any polynomial $g\in K[x]$ with $\deg g<\deg\phi$, we have $v(g(\t_\phi))=v(g(\t))$. In particular, $e(\phi)\mid e(F)$.
\item There is a canonical embedding $\F_\phi\to\F_F$, given by $g(\t_\phi)+\m_\phi\mapsto
g(\t)+\m_F$ for any polynomial $g\in K[x]$ with $\deg g<\deg\phi$ such that $v(g(\t_\phi))\ge0$.   
\end{enumerate}
\end{corollary}

\begin{proof}
If a polynomial $g\in K[x]$ has $\deg g<\deg\phi$, then $\phi\nmid_\mu g$ and $v(g(\t_\phi))=\mu(g)=v(g(\t))$, by Proposition \ref{theta} and Theorem \ref{computation}, respectively. This proves item 1.

Let $\ll_F$ be the kernel of the canonical ring homomorphism 
$$\Delta(\mu)\lra \F_F,\qquad g+\pset_0^+(\mu)\mapsto g(\t)+\m_F.
$$ 
Since $\ll_F$ is a non-zero prime ideal of the PID $\Delta(\mu)$, it is a maximal ideal. 
By Theorem \ref{fundamental} and (\ref{equality}), $\rr(\phi)^a=\rr(F)\subset \ll_F$ for a certain positive integer $a$. Since $\rr(\phi)$ and $\ll_F$ are maximal ideals, they coincide. By Proposition \ref{sameideal}, the homomorphism $\Delta(\mu)\to\F_\phi$ given by $g+\pset_0^+(\mu)\mapsto g(\t)+\m_\phi$ is onto and it has the same kernel. This proves item 2.
\end{proof}

\subsection{Okutsu invariants of prime polynomials}\label{subsecOkutsu} 
We keep dealing with a prime polynomial $F\in\P$ and a fixed root $\t\in\kb$ of $F$.

Let $F_1,\dots,F_{r}\in\oo[x]$ be monic polynomials of strictly increasing degree: $$1\le \deg F_1<\cdots <\deg F_r<\deg F.$$ Denote $F_{r+1}:=F$ and consider the following sequence of constants:
$$
C_0:=0;\quad C_i:=\dfrac{v(F_i(\t))}{\deg F_i},\ 1\le i\le r+1.
$$
Note that $C_{r+1}=\infty$.
We say that $[F_1,\dots,F_r]$ is an \emph{Okutsu frame} of $F$ if
\begin{equation}\label{frame}
\deg g<\deg F_{i+1}\ \Longrightarrow\ \dfrac{v(g(\t))}{\deg g}\le C_i<C_{i+1},
\end{equation}
for any monic polynomial $g(x)\in\oo[x]$ and any  $0\le i\le r$.

Since $v$ is discrete, every prime polynomial admits a finite Okutsu frame. The length $r$ of the frame is called the \emph{Okutsu depth} of $F$. Clearly, the depth $r$, the degrees $\deg F_i$ and the constants $C_i$ attached to any Okutsu frame are intrinsic data of $F$. We denote $C_i(F):=C_i$ for all $0\le i\le r+1$.
It is easy to deduce from (\ref{frame}) that all polynomials $F_1,\dots,F_{r}$ are prime polynomials. 


\begin{theorem}\label{MLOk}
Consider an optimal MacLane chain of an inductive valuation $\mu$. 
$$\mu_0\ \stackrel{(\phi_1,\la_1)}\lra\  \mu_1\ \stackrel{(\phi_2,\la_2)}\lra\ \cdots
\ \stackrel{(\phi_{r-1},\la_{r-1})}\lra\ \mu_{r-1} 
\ \stackrel{(\phi_{r},\la_{r})}\lra\ \mu_{r}=\mu
$$

Then, $[\phi_1,\dots,\phi_r]$ is an Okutsu frame of every strong key polynomial $F$ for $\mu$, and $C_i(F)=C_i(\mu)$ for all $1\le i\le r$. 
\end{theorem}

\begin{proof}
Let $F\in\kps$, and let $\t\in \kb$ be a root of $F$. By the optimality of the MacLane chain, $m_1<\cdots <m_r<m_{r+1}:=\deg F$. Fix an index $0\le i\le r$. For every monic polynomial $g$ with $\deg g<m_{i+1}$, Proposition \ref{theta} and Lemma \ref{stable} show that
$$
\mu_i(g)=\mu_{i+1}(g)=\cdots=\mu(g)=\mu_{\infty,F}(g).
$$
These equalities hold in particular for $\phi_i$. Hence, by Theorem \ref{preMLOk}:
$$
v(g(\t))/\deg g=\mu_i(g)/\deg g\le C(\mu_i)=\mu_i(\phi_i)/m_i=v(\phi_i(\t))/m_i.
$$
The inequality $C_i(\mu)<C_{i+1}(\mu)$  was proved at the beginning of section \ref{subsecChains}.
\end{proof}

\begin{definition}\label{muF}
The \emph{Okutsu discriminant bound} of a prime polynomial $F\in\P$ of Okutsu depth $r$ is defined as 
\begin{align*}
\delta_0(F):=&\ \deg(F) C_r(F)=v(\res(\phi_r,F))/ \deg \phi_r\\
=&\ \deg(F) \max\left\{v(g(\t)/\deg g\mid g\in \oo[x],\ g\mbox{ monic},\ \deg g<\deg F\right\}.
\end{align*}

We may attach to $F$ a valuation $\mu_F\colon K_v(x)^*\to \Q$, determined by the following action on polynomials:
\begin{itemize}
\item $\mu_F(a)=\mu_{\infty,F}(a)$, if $a\in K[x]$ has $\deg a<\deg F$.
\item $\mu_F(F)=\delta_0(F)$.
\item If $g=\sum_{0\le s}a_sF^s$ is the $F$-expansion of $g$, then $\mu_F(g)=\min_{0 \le s}\{\mu_F(a_sF^s)\}$. 
\end{itemize}

We denote by the same symbol $\mu_F$ the valuation on $K(x)$ obtained by restriction.
\end{definition}

The next theorem, which is a kind of converse of Theorem \ref{MLOk}, shows that $\mu_F$ is indeed a valuation.

\begin{theorem}\label{OkML}
Let $[F_1,\dots,F_r]$ be an Okutsu frame of a prime polynomial $F\in\P$. Then, $\mu_F$ is an inductive valuation on $K(x)$ admitting an optimal MacLane chain  
$$\mu_0=\mu_{F_1}\ \stackrel{(F_1,\la_1)}\lra\  \mu_1=\mu_{F_2}\ \stackrel{(F_2,\la_2)}\lra\ \cdots
\ \stackrel{(F_{r-1},\la_{r-1})}\lra\ \mu_{r-1}=\mu_{F_{r}} 
\ \stackrel{(F_{r},\la_{r})}\lra\ \mu_F,
$$
with $\la_i=v(F_i(\t))-\delta_0(F_i)$ for $1\le i\le r$, being $\t\in\kb$ a root of $F$.
Moreover, $F$ is a strong key polynomial for $\mu_F$ as a valuation on $K_v(x)$.
\end{theorem}

\begin{proof}
Denote $F_{r+1}:=F$. Since $F_1$ is a monic polynomial with minimal degree among all polynomials $g$ satisfying $\mu_0(g)<\mu_{\infty,F}(g)$, Lemma \ref{minDegree} shows that $F_1$ is a (strong) key polynomial for $\mu_0$ and $F_1\mid_{\mu_0} F$. As a key polynomial for $\mu_0$, $\overline{F}_1\in\F[y]$ is irreducible, and this implies that $F_1$ has Okutsu depth zero. An Okutsu frame of $F_1$ is the empty set, so that $\delta_0(F_1)=\deg(F_1)C_0(F_1)=0$. Thus, $\mu_{F_1}=\mu_0$ is a valuation and $F_1$ is a strong key polynomial for this valuation. 
This proves the theorem in the case $r=0$. If $r>0$, we have proved the following conditions for the index $i=1$:\medskip

(a) \ $\mu_{F_i}$ is a valuation admitting an optimal MacLane chain  
$$\mu_0=\mu_{F_1}\ \stackrel{(F_1,\la_1)}\lra\  \mu_1=\mu_{F_2}\ \stackrel{(F_2,\la_2)}\lra\ \cdots
\ \stackrel{(F_{i-1},\la_{i-1})}\lra\ \mu_{i-1}=\mu_{F_{i}},
$$
with $\la_j=v(F_j(\t))-\delta_0(F_j)$ for $1\le j< i$.\medskip

(b) \ $F_i$ is a strong key polynomial for $\mu_{F_i}$ and $F_i\mid_{\mu_{F_i}}F$.\medskip

We need only to show that if these conditions are satisfied for an index $1\le i\le r$, then they are satisfied for the index $i+1$.

$[F_1,\dots,F_{i-1}]$ is an Okutsu frame of $F_{i}$. Hence,

Since  $F_i\mid_{\mu_{F_i}}F$, Theorem \ref{fundamental} shows that $\mu_{F_i}(F_i)<v(F_i(\t))$. 
Therefore, by the definition of $\mu_{F_i}$, the monic polynomial $F_i$ has minimal degree among all polynomials $g$ satisfying $\mu_{F_i}(g)<\mu_{\infty,F}(g)$. By Lemma \ref{minDegree},  $[\mu_{F_i};(F_i,\la_i)]\le \mu_{\infty,F}$, where $\la_i=\mu_{\infty,F}(F_i)-\mu_{F_i}(F_i)=v(F_i(\t))-\delta_0(F_i)$. Denote $\mu:=[\mu_{F_i};(F_i,\la_i)]$.

Let $\phi$ be a monic polynomial with minimal degree among all polynomials $g$ satisfying $\mu(g)<\mu_{\infty,F}(g)$. By Lemma \ref{minDegree}, $\phi$ is a key polynomial for $\mu$ and $\mu(g)<\mu_{\infty,F}(g)$ is equivalent to $\phi\mmu g$; in particular, $\phi\mmu F$. By Lemma \ref{phi}, $F_i$ is a key polynomial for $\mu$; hence, Lemma \ref{bound} shows that 
\begin{equation}\label{contrad}
\dfrac{v(\phi(\t))}{\deg\phi}>\dfrac{\mu(\phi)}{\deg\phi}=C(\mu)=\dfrac{\mu(F_i)}{\deg F_i}=
\dfrac{\mu_{F_i}(F_i)+\la_i}{\deg F_i}=
\dfrac{v(F_i(\t))}{\deg F_i}=C_i(F).
\end{equation}
By (\ref{frame}), we have necessarily $\deg\phi\ge \deg F_{i+1}$. On the other hand, by (\ref{frame}), (\ref{contrad}) and Theorem \ref{preMLOk}, we have
$$
\dfrac{v(F_{i+1}(\t))}{\deg F_{i+1}}>C_i(F)=C(\mu)\ge
\dfrac{\mu(F_{i+1})}{\deg F_{i+1}}.
$$
Hence, $v(F_{i+1}(\t))>\mu(F_{i+1})$, which is equivalent to $\phi\mmu F_{i+1}$ by Theorem \ref{fundamental}. By the $\mu$-minimality of $\phi$, we have $\deg \phi\le \deg F_{i+1}$. Thus, $\deg\phi=\deg F_{i+1}$ and we get $\phi\smu F_{i+1}$ by Lemma \ref{mid=sim}. Therefore, $F_{i+1}$ is a key polynomial for $\mu$ and $F_{i+1}\mmu F$. Also, the inequality $\deg F_{i+1}>\deg F_i$ between two key polynomials for $\mu$ shows that $F_{i+1}$ is a strong key polynomial for $\mu$.

Let $\t_{i+1}\in \kb$ be a root of $F_{i+1}$. Since $F_{i+1}\mmu F$, Corollary \ref{efphiF} shows that $v(a(\t_{i+1}))=v(a(\t))$ for any $a\in K[x]$ with $\deg a<\deg F_{i+1}$. In particular, $C_j(F_{i+1})=C_j(F)$ for all $j\le i$, and  $[F_1,\dots,F_i]$ is an Okutsu frame of $F_{i+1}$. By (\ref{contrad}) we have:
\begin{equation}\label{delta0i+1}
\delta_0(F_{i+1})=\deg F_{i+1}\,C_i(F_{i+1})=\deg F_{i+1}\,C_i(F)=\deg F_{i+1}C(\mu).
\end{equation}

Finally, let us show that $\mu=\mu_{F_{i+1}}$. Let $g=\sum_{0\le s}a_s(F_{i+1})^s$ be the $F_{i+1}$-expansion of a polynomial $g\in K[x]$. Since $F_{i+1}\in \kpm$, we have:\medskip

$\bullet$ \ $\mu(a_s)=\mu_{\infty,F_{i+1}}(a_s)=\mu_{F_{i+1}}(a_s)$, by Proposition \ref{theta}.

$\bullet$ \ $\mu(F_{i+1})=\deg F_{i+1}C(\mu)=\delta_0(F_{i+1})=\mu_{F_{i+1}}(F_{i+1})$, by Lemma \ref{bound} and (\ref{delta0i+1}).

$\bullet$ \ $\mu(g)=\min_{0 \le s}\{\mu(a_s(F_{i+1})^s)\}=\min_{0 \le s}\{\mu_{F_{i+1}}(a_s(F_{i+1})^s)\}=\mu_{F_{i+1}}(g)$.
\end{proof}

Let us emphasize a fact that was seen along the proof of Theorem \ref{OkML}.

\begin{corollary}
Let $[F_1,\dots,F_r]$ be an Okutsu frame of $F\in\P$. For any $1\le i\le r$, let $\t_i\in\kb$ be a root of $F_i$. Then, $v(g(\t_i))=v(g(\t))$ for any polynomial $g\in K[x]$ with $\deg g<\deg F_i$. In particular, $[F_1,\dots,F_{i-1}]$ is an Okutsu frame of $F_i$.
\hfill{$\Box$}
\end{corollary}

\begin{corollary}
The MacLane depth of an inductive valuation $\mu$ is equal to the Okutsu depth of any strong key polynomial for $\mu$.
The Okutsu depth of a prime polynomial $F$ is equal to the MacLane depth of the canonical valuation $\mu_F$.\hfill{$\Box$}
\end{corollary}

\begin{corollary}\label{muFmu}
Let $\mu$ be an inductive valuation and $F$ a prime polynomial. Then,
$\mu=\mu_F$ if and only if $F$ is a strong polynomial for $\mu$.
\end{corollary}

\begin{proof}
If $\mu=\mu_F$, then $F\in\kps$ by Theorem \ref{OkML}. 
Conversely, suppose that $F\in\kps$ and consider an optimal MacLane chain of $\mu$.
$$\mu_0\ \stackrel{(\phi_1,\la_1)}\lra\  \mu_1\ \stackrel{(\phi_2,\la_2)}\lra\ \cdots
\ \stackrel{(\phi_{r-1},\la_{r-1})}\lra\ \mu_{r-1} 
\ \stackrel{(\phi_{r},\la_{r})}\lra\ \mu_{r}=\mu
$$
By Theorem \ref{fundamental} and Corollary \ref{lower}, $\phi_i\mid_{\mu_{i-1}}F$ and $\la_i=v(\phi_i(\t))-\mu_{i-1}(\phi_i)$, for all $1\le i\le r$. By Theorem \ref{OkML}, $\mu_i=\mu_{\phi_{i+1}}$ for all  $1\le i\le r$ and $\mu=\mu_F$. 
\end{proof}

\begin{definition}
Let $F$ be a prime polynomial of Okutsu depth $r$. Let $f_r:=\deg\rr_{\mu_F}(F)=\deg R_r(F)$ with respect to any optimal MacLane chain of $\mu_F$.
An \emph{Okutsu invariant} of $F$ is a rational number that depends only on $e_0,\dots,e_r,f_0,\dots,f_r,h_1,\dots,h_r$; that is, on the basic MacLane invariants of $\mu_F$ and the number $f_r$. 

Note that $f_r=\deg R_r(F)$ with respect to any optimal MacLane chain of $\mu_F$.
\end{definition}

The ramification index, residual degree, and the Okutsu discriminant bound of $F$ are Okutsu invariants:
$$e(F)=e_0\cdots e_r, \qquad f(F)=f_0\cdots f_r,\qquad \delta_0(F)=e_rf_r(w_r+\la_r),
$$
as shown in Proposition \ref{theta}, Corollary \ref{identification}, and equation (\ref{C}), respectively. The \emph{index}, the \emph{exponent} and the \emph{conductor} of a prime polynomial are also Okutsu invariants admitting explicit formulas in terms of the basic invariants $e_i,f_i,h_i$ \cite{Ndiff}. 

\begin{definition}\label{okequiv}
Let $F,G\in\P$ be two prime polynomials of the same degree, and let $\t\in\kb$ be a root of $F$. 
We say that $F$ and $G$ are \emph{Okutsu equivalent}, and we write $F\approx G$, if $v(G(\t))>\delta_0(F)$.
\end{definition}

The idea behind this concept is that $F$ and $G$ are close enough to share the same Okutsu invariants, as the next result shows.

\begin{proposition}\label{criteria}
Let $F,G\in\P$ be two prime polynomials of degree $n$. The following conditions are equivalent:
\begin{enumerate}
\item $F\approx G$.
\item $F\sim_{\mu_F}G$.
\item $v(\res(F,G))>n\delta_0(F)$.
\item $\mu_F=\mu_G$ and $\rr(F)=\rr(G)$, where $\rr:=\rr_{\mu_F}=\rr_{\mu_G}$.
\end{enumerate}
\end{proposition}

\begin{proof}
If $\mu=\mu_F$, then $F\in\kps$ by Theorem \ref{OkML}. 
Conversely, suppose that $F\in\kps$ and consider an optimal MacLane chain of $\mu$.
$$\mu_0\ \stackrel{(\phi_1,\la_1)}\lra\  \mu_1\ \stackrel{(\phi_2,\la_2)}\lra\ \cdots
\ \stackrel{(\phi_{r-1},\la_{r-1})}\lra\ \mu_{r-1} 
\ \stackrel{(\phi_{r},\la_{r})}\lra\ \mu_{r}=\mu
$$
By Corollary \ref{lower}, $\la_i=v(\phi_i(\t))-\mu_{i-1}(\phi_i)$ for all $1\le i\le r$. By Theorem \ref{MLOk}, $[\phi_1,\dots,\phi_r]$ is an Okutsu frame of $F$, and by Theorem \ref{OkML}, we get recursively $\mu_{i-1}=\mu_{\phi_{i}}$ for all  $1\le i\le r$ and $\mu=\mu_F$. 
\end{proof}

The symmetry of condition (4) shows that $\approx$ is an equivalence relation on the set $\P$ of prime polynomials. These conditions determine a parameterization of the quotient set 
$\P/\!\approx$ by an adequate space.

\begin{definition}
Let $\mu$ be an inductive valuation. We say that a maximal ideal $\ll\in\mx(\Delta(\mu))$ is \emph{strong} if $\ll=\rr(\phi)$ for a strong key polynomial $\phi$.

The \emph{MacLane space} of the valued field $(K,v)$ is defined to be the set
$$\M=\left\{(\mu,\ll)\mid \mu\in\Vi, \ \ll\in\mx(\Delta(\mu)), \ \ll \mbox{ strong}\right\}.$$
\end{definition}

The next result is a consequence of Corollary \ref{muFmu} and Proposition \ref{criteria}.

\begin{theorem}\label{MLspace}
The following mapping is bijective:
$$\M\lra \P/\!\!\approx,\qquad (\mu,\ll)\mapsto \kpm_\ll. 
$$The inverse map is determined by $F\mapsto (\mu_F,\rr_{\mu_F}(F))$.\hfill{$\Box$}
\end{theorem}

The bijection $\M\to\P/\!\!\approx$ has applications to the computational representation of irreducible polynomials over complete fields, because the elements in the MacLane space may be described by discrete parameters. This provides an efficient  manipulation of approximations to the irreducible factors in $K_v[x]$ of a polynomial with coefficients in a global field $K$.  

\section{Limit valuations}\label{sectionLimit}
\subsection{Tree structure on $\Vi$}\label{sectionTree}
\begin{definition}\label{prec}
For $\mu,\mu'\in\Vi$, we say that $\mu$ is the previous node of $\mu'$, and we write $\mu\prec\mu'$, if  $\mu'=[\mu;(\phi,\la)]$ for some strong key polynomial $\phi$ for $\mu$ and some positive rational number $\la$. 

We denote by $(\Vi,\prec)$ the oriented graph whose set of vertices is $\Vi$, and there is an edge from $\mu$ to $\mu'$ if and only if $\mu\prec\mu'$
\end{definition}

Proposition \ref{unicity} shows that  $(\Vi,\prec)$ is a connected tree with root node $\mu_0$, and any optimal MacLane chain for $\mu\in\Vi$ yields the unique path joining $\mu$ with the root node. In particular, the length of this path is the MacLane depth of $\mu$.

Also, Lemma \ref{unique} provides a description of the infinite set $E(\mu)$ of branches of any node $\mu\in\Vi$. In fact, there is a bijection:
$$
\left(\kps\times \Q_{>0}\right)/\!\sim\,\lra\,E(\mu),\qquad (\phi,\la)\mapsto [\mu;(\phi,\la)],
$$ 
where $\sim$ is the equivalence relation:
$$
(\phi,\la)\sim (\phi',\la') \quad \mbox{ if }\quad \deg\phi=\deg\phi', \ \la=\la', \ \mu(\phi-\phi')\ge\mu(\phi)+\la.
$$

Since the tree structure is determined by the optimal MacLane chains, the bijective mapping $\Vi(K)\to \Vi(K_v)$ established in Proposition \ref{KKv} is a tree isomorphism.

MacLane showed that there are two kinds of valuations that may be obtained as limits of inductive valuations: those of finite and infinite depth. In the next sections we review them. 

\subsection{Limits with infinite depth}
\begin{definition}\label{leaf}
A \emph{leaf} of $(\Vi,\prec)$ is an infinite path 
$$
\mu_0\prec\mu_1\prec\cdots\prec\mu_n\prec\cdots
$$

We say that the leaf is \emph{discrete} if the group values of the valuations are stable; that is,  $\Gamma(\mu_n)=\Gamma(\mu_{n_0})$ for all $n\ge n_0$, for a certain $n_0$. 
\end{definition}

A leaf has attached an infinite number of MacLane invariants $e_i,f_i,h_i,m_i$, which depend only on the sequence of valuations and not on the choice of the strong key polynomials $\phi_i$ used to construct $\mu_i$ from $\mu_{i-1}$. Since the degrees $m_i$ of these strong polynomials grow strictly and $m_{i+1}=e_if_im_i$, we have $e_if_i>1$ for all $i\ge1$. Also, for any $g\in K[x]$, we shall have $\deg g<m_{i+1}$ for a sufficiently advanced index $i$. Thus, Lemma \ref{stable} shows that   
$$
\mu_i(g)=\mu_j(g),\ \mbox{ for all }j\ge i.
$$
Thus, any leaf determines a limit valuation $\mu_\infty=\lim \mu_n$, defined by $\mu_\infty(g)=\mu_i(g)$ for a sufficiently advanced index $i$ such that the value $\mu_i(g)$ stabilizes. Note that $\mu_\infty(g)$ takes finite values for all non-zero $g\in K[x]$.

Since the products $e_if_i$ are always greater than one, either $\lim e(\mu_n)=\infty$, or $\lim f(\mu_n)=\infty$ (not exclusively). If $\lim e(\mu_n)=\infty$, then the group of values of $\mu_\infty$ has accumulation points at all the integers, and the valuation is not discrete. If $\lim  e(\mu_n)\ne\infty$, then there exists an index $n_0$ such that $e_n=1$ for all $n>n_0$, or equivalently, $e(\mu_n)=e(\mu_{n_0})$ for all $n\ge n_0$; thus, the leaf is discrete. In this case, $e(\mu_\infty)=e(\mu_{n_0})$ and the valuation $\mu_\infty$ is discrete. 

In this discrete case, we must have $\lim f(\mu_n)=\infty$, so that the inductive limit (union after the standard identifications) $\F_\infty=\bigcup_n\F_n$ is an infinite algebraic extension of $\F$. It is easy to check that 
$$\kp(\mu_\infty)=\emptyset,\qquad \kappa(\mu_\infty)\simeq\Delta(\mu_\infty)=\F_\infty,\qquad \gg(\mu_\infty)\simeq \F_\infty[p,p^{-1}],
$$ where $p$ is an indeterminate.  

Since the tree isomorphism $(\Vi(K),\prec)\simeq(\Vi(K_v),\prec)$ preserves the invariants $e_i,f_i$ attached to each node, it induces a 1-1 correspondence between the valuations with infinite depth on $K(x)$ and the valuations with infinite depth on $K_v(x)$.         

\subsection{Limits with finite depth}
An \emph{infinite MacLane chain} is an infinite sequence of augmented va\-luations: 
$$
\mu_0\ \stackrel{(\phi_1,\la_1)}\lra\ \cdots
\ \stackrel{(\phi_{n-1},\la_{n-1})}\lra\ \mu_{n-1} 
\ \stackrel{(\phi_{n},\la_{n})}\lra\ \mu_{n}\ 
\stackrel{(\phi_{n+1},\la_{n+1})}\lra\  \cdots
$$
such that $\phi_{n+1}\nmid_{\mu_n}\phi_n$ for all $n$. By Lemmas \ref{minimal} and \ref{groups}, $m_n\mid m_{n+1}$ and $\Gamma(\mu_n)\subset\Gamma(\mu_{n+1})$ for all $n$.

If the degrees $m_n$ of the key polynomials $\phi_n$ are not bounded, there exists a limit valuation of this sequence, which is one of the valuations with infinite depth already described in the previous section.

If the degrees $m_n$ are  bounded, there exists an index $t$ such that $m_n=m_{t}$ for all $n\ge t$. Hence, $e_n=1=f_n$ for all $n\ge t$, and this implies 
$$\Gamma(\mu_n)=\Gamma(\mu_{t-1}),\qquad \F_n=\F_{t}, \ \mbox{ for all }n\ge t.
$$ 

\begin{theorem}\cite[Thm 7.1]{mcla}
Every infinite MacLane chain with stable degrees determines a limit pseudo-valuation on $K[x]$, given by $g\mapsto\lim_n \mu_n(g)$. This pseudo-valuation coincides with $\mu_{\infty,F}$ for some prime polynomial $F\in\P$. Let $\t\in\kb$ be a root of $F$. If $\t$ is algebraic over $K$, then $\mu_{\infty,F}$ is infinite on the ideal of $K[x]$ generated by the minimal polynomial of $\t$ over $K$. If $\t$ is transcendental over $K$, then $\mu_{\infty,F}$ determines a valuation on $K(x)$ with:
$$
e(F)=e(\mu_{\infty,F})=e(\mu_{t-1}),\quad \F_F=\kappa(\mu_{\infty,F})\simeq\Delta(\mu_{\infty,F})=\F_{t},
$$
where $m_n=m_{t}$ for all $n\ge t$. Also, $\gg(\mu_{\infty,F})\simeq \F_F[p,p^{-1}]$, where $p$ is an indeterminate, and $\kp(\mu_{\infty,F})=\emptyset$.\hfill{$\Box$}
\end{theorem}

Consider an infinite MacLane chain with stable degrees and limit $\mu_{\infty,F}$ for some $F\in\P$. Let $t$ be the least index such that $m_n=m_{t}$ for all $n\ge t$. Clearly, $\deg F=e(F)f(F)=e(\phi_{t})f(\phi_{t})=m_{t}$. For all $i$, the key polynomial $\phi_i$ is $\mu_{i-1}$-proper and $\phi_i\mid_{\mu_{i-1}}F$, by Corollary \ref{lower}. By Lemma \ref{mid=sim}, $F$ is a key polynomial for $\mu_{t-1}$. By Lemma \ref{phi}, $\phi_{t-1}$ is a key polynomial for $\mu_{t-1}$ too. Hence, $\deg \phi_{t-1}<\deg F$ implies that $F$ is a strong key polynomial for $\mu_{t-1}$. Thus, $\mu_{t-1}=\mu_F$ by Corollary \ref{muFmu}.

Let us emphasize the role of $\mu_F$ as a threshold valuation in the process of constructing approximations to $\mu_{\infty,F}$.  

\begin{proposition}
 Consider an infinite MacLane chain with limit $\mu_{\infty,F}$ and let $t$ be the first index such that $\deg\phi_n=\deg\phi_t$ for all $n\ge t$. Then, $\mu_{t-1}=\mu_F$. \hfill{$\Box$}   
\end{proposition}

By Lemma \ref{existence}, all valuations $\mu_n$ with $n\ge t$ have the same depth, and by Theorem \ref{MLOk}, this depth coincides with the Okutsu depth of $F$. Thus, it makes sense to say that these pseudo-valuations are limits with \emph{finite depth}. 

\begin{theorem}\cite[Thm. 8.1]{mcla}
The set $\V$ is the union of $\Vi$, the limit valuations given by the discrete leaves of $(\Vi,\prec)$, and the valuations $\mu_{\infty,F}$ determined by all prime polynomials in $\P$ which do not divide any polynomial in $\oo[x]$. \hfill{$\Box$}
\end{theorem}

Note that limit valuations $\mu_{\infty,F}\in\V$ of finite depth do not occur if $K=K_v$. 
A posteriori, it is easy to distinguish the inductive valuations among all valuations.

\begin{corollary}\label{characterization}
For any $\mu\in\V$, the following conditions are equivalent:
\begin{enumerate}
\item $\mu$ is an inductive valuation.
\item $\mu$ is residually transcendental; that is, $\kappa(\mu)/\kappa(v)$ is a transcendental extension.
\item $\kpm\ne\emptyset$.
\item $\mu(g)/\deg g$ is bounded on all monic polynomials $g\in K[x]$.
\item there exists a pseudo-valuation $\mu'$ on $K[x]$ such that $\mu<\mu'$.\hfill{$\Box$}
\end{enumerate} 
\end{corollary}

\subsection{Intervals of valuations}\label{subsecIntervals}
For arbitrary $\mu,\mu'\in\V$, recall that the interval $[\mu,\mu']$ is defined as: 
$$[\mu,\mu']=\{\nu\in\V\mid \mu\le \nu\le \mu'\}.
$$
 
\begin{theorem}\label{totally}
For any pseudo-valuation $\mu$ on $K[x]$, the interval $[\mu_0,\mu)\subset\Vi$ is totally ordered. 
\end{theorem}

\begin{proof}
Let $\nu,\nu'$ be two valuations such that $\nu<\mu$ and $\nu'<\mu$. Consider a monic
polynomial $\phi\in K[x]$ of minimal degree satisfying $\nu(\phi)<\mu(\phi)$; by Lemma \ref{minDegree}, $\phi\in\kp(\nu)$ and for any non-zero $g\in K[x]$, $\nu(g)=\mu(g)$ is equivalent to $\phi\nmid_\nu g$. Let $\phi'\in K[x]$ be a monic polynomial with analogous properties with respect to $\nu'$. Suppose $\deg\phi\le\deg\phi'$. 

By the minimality of $\deg\phi$ and $\deg\phi'$, for all $a\in K[x]$ with $\deg a<\deg\phi$, we have $\nu(a)=\mu(a)=\nu'(a)$. 
If $\deg\phi<\deg\phi'$, then $\nu'(\phi)=\mu(\phi)>\nu(\phi)$. Hence, $\nu'\ge\nu$, because for any non-zero $g\in K[x]$ with $\phi$-expansion $g=\sum_{0\le s}a_s\phi^s$, we have:
\begin{equation}\label{unequality}
\nu'(g)\ge \min_{0 \le s}\{\nu'(a_s\phi^s)\}\ge \min_{0 \le s}\{\nu(a_s\phi^s)\}=\nu(g). 
\end{equation}

If $\deg\phi=\deg \phi'$, then $\phi'=\phi+a$ for some $a\in K[x]$ with $\deg a<\deg\phi$.
By the $\nu$-minimality of $\phi$ and the $\nu'$-minimality of $\phi'$, we have
\begin{equation}\label{twomins}
\nu(\phi')=\min\{\nu(\phi),\nu(a)\},\qquad \nu'(\phi)=\min\{\nu'(\phi'),\nu(a)\}.
\end{equation}
Suppose $\nu(\phi)\le \nu'(\phi')$. Then,  
$$
\nu(\phi')=\min\{\nu(\phi),\nu(a)\}\le \nu(\phi)\le \nu'(\phi')<\mu(\phi').
$$
Hence, $\phi\mid_{\nu}\phi'$. By Lemma \ref{mid=sim}, $\phi\sim_\nu \phi'$, so that $\nu(\phi)=\nu(\phi')\le \nu'(\phi)$, by (\ref{twomins}). Therefore, (\ref{unequality}) holds and $\nu'\ge\nu$.
\end{proof}

Our aim is to find an explicit description of the valuations in such a totally ordered interval. Let us start with the interval determined by an augmented valuation.\medskip

For any key polynomial $\phi$ for $\mu$, the pseudo-valuation $\mu_{\infty,\phi}$ can be regarded as $\mu_{\infty,\phi}=[\mu;(\phi,\infty)]$ (cf. section \ref{subsecRideal}). Also, it makes sense to regard $\mu$ as a trivial ayugmentation of itself, namely $\mu=[\mu;(\phi,0)]$.

\begin{lemma}\label{intervalAug}
Let $\phi$ be a key polynomial for an inductive valuation $\mu$, and consider the augmented valuation $\mu'=[\mu;(\phi,\la)]$ for some $\la\in\Q_{>0}\cup\{\infty\}$. Then, 
$$
[\mu,\mu')=\left\{[\mu;(\phi,\rho)]\mid \rho\in\Q,\ 0\le \rho< \la\right\}.
$$
\end{lemma}

\begin{proof}
For every $\rho\in\Q\cap[0,\la]$, denote $\mu_\rho:=[\mu;(\phi,\rho)]$. 
Consider a valuation $\nu\in\V$ such that $\mu\le\nu<\mu'$. For all $a\in K[x]$ with $\deg a<\deg\phi$, we have  $\mu(a)\le\nu(a)\le\mu'(a)=\mu(a)$, leading to $\mu(a)=\nu(a)$. Take
$\rho=\nu(\phi)-\mu(\phi)\in\Q\cap[0,\la]$. For any $g\in K[x]$, with $\phi$-expansion $g=\sum_{0\le s}a_s\phi^s$, we have
$$
\nu(g)\ge \min_{0\le s}\{\nu(a_s\phi^s)\}=\min_{0\le s}\{\mu(a_s\phi^s)+s\rho\}=\mu_\rho(g),
$$
so that $\mu_\rho\le \nu$. If $\rho=\la$, then $\mu'=\mu_\rho\le \nu$, against our assumption; thus, $\rho<\la$. 
We claim that $\mu_\rho=\nu$. In fact, let us show that $\mu_\rho<\nu<\mu'$ leads to a contradiction. Let $g\in K[x]$ be any polynomial such that $\mu_\rho(g)<\nu(g)$. By the very definition of the augmented valuations, there exists a sufficiently small rational number $\ep>0$ such that $\rho+\ep<\la$ and $\mu_{\rho+\ep}(g)<\nu(g)$. On the other hand, $\mu_{\rho+\ep}(\phi)=\rho+\ep>\rho=\nu(\phi)$. Therefore, $\nu\not\le\mu_{\rho+\ep}$ and $\nu\not\ge\mu_{\rho+\ep}$, in contradiction with Theorem \ref{totally}.
\end{proof}

Let $F\in\P$ be a prime polynomial with respect to $v$. By Theorem \ref{OkML}, $F$ is a key polynomial for the inductive valuation $\mu_F\in \Vi(K_v)$.
Consider an optimal MacLane chain of its restriction $\mu_F\in \Vi(K)$:$$\mu_0\ \stackrel{(\phi_1,\la_1)}\lra\  \mu_1\ \stackrel{(\phi_2,\la_2)}\lra\ \cdots\ \stackrel{(\phi_{r-1},\la_{r-1})}\lra\ \mu_{r-1} \ \stackrel{(\phi_{r},\la_{r})}\lra\ \mu_{r}=\mu_F.$$
By Theorem \ref{totally}, 
$$
[\mu_0,\mu_{\infty,F})=[\mu_0,\mu_1)\cup[\mu_1,\mu_2)\cup\cdots\cup[\mu_{r-1},\mu_F)\cup[\mu_F,\mu_{\infty,F}),
$$
and Lemma \ref{intervalAug} gives an explicit description of each of these subintervals. 
If we consider valuations over $K_v(x)$, the subinterval $[\mu_F,\mu_{\infty,F})$ is equal to
$$
[\mu_F,\mu_{\infty,F})=\left\{[\mu_F;(F,\la)]\mid \la\in\Q_{\ge0}\right\}\subset \Vi(K_v),
$$
By Proposition \ref{KKv}, the restriction of these valuations to $K(x)$ yields a completely analogous description of $[\mu_F,\mu_{\infty,F})\subset \Vi(K)$.
By Lemma \ref{unique}, the restriction of $[\mu_F;(F,\la)]$ to $K(x)$ coincides with $[\mu_F;(\phi,\la)]$ for any $\phi\in K[x]$ such that $\deg\phi=\deg F$ and $\mu_F(F-\phi)\ge \mu_F(F)+\la$.


\begin{thebibliography}{}


\bibitem{APZ}
V. Alexandru, N. Popescu, A. Zaharescu, \emph{All valuations on $K(x)$}, Journal of Mathematics of Kyoto University {\bf 30} (1990), 281--296.


\bibitem{Bauch}J.-D. Bauch, \emph{Genus computation of global function fields},	arXiv: 1209.0309v3 [math.NT]


\bibitem{okutsu}
J. Gu\`{a}rdia, J.  Montes, E.  Nart, \emph{Okutsu invariants and Newton polygons}, Acta Arithmetica {\bf 145} (2010), 83--108.

\bibitem{algorithm}
J. Gu\`{a}rdia, J.  Montes, E.  Nart, \emph{Higher  Newton polygons in the computation of discriminants and prime ideal decomposition in number fields},
Journal de Th\'eorie des Nombres de Bordeaux {\bf 23} (2011), no. 3, 667--696.

\bibitem{HN}
J. Gu\`{a}rdia, J. Montes, E. Nart, \emph{Newton polygons of higher order in algebraic number theory}, Transactions of the American Mathematical Society  {\bf 364} (2012), no. 1, 361--416.

\bibitem{newapp}
J. Gu\`{a}rdia, J. Montes, E.  Nart, \emph{A new computational approach to ideal theory in number fields},   Foundations of Computational Mathematics, DOI 10.1007/s10208-012-9137-5.

\bibitem{bases} J. Gu\`ardia, J. Montes, E. Nart, \emph{Higher Newton polygons and integral bases}, arXiv: 0902.3428v2 [math.NT].

\bibitem{gen}
J. Gu\`{a}rdia, E.  Nart, \emph{Genetics of polynomials over local fields},   in preparation.

\bibitem{GNP}J. Gu\`{a}rdia, E.  Nart, S. Pauli, \emph{Single-factor lifting and factorization of polynomials over local fields}, Journal of Symbolic Computation {\bf 47} (2012), 1318--1346.


\bibitem{mcla} S. MacLane, \emph{A construction for absolute values in polynomial rings}, Transactions of the American Mathematical Society, {\bf40} (1936), pp. 363--395.

\bibitem{mclb} S. MacLane, \emph{A construction for prime ideals as absolute values of an algebraic field}, Duke Mathematical Journal {\bf2} (1936), pp. 492--510.

\bibitem{Mo} J. Montes, \emph{Pol\'{\i}gonos de Newton de orden superior y aplicaciones aritm\'eticas}, PhD Thesis, Universitat de Barcelona, 1999.


\bibitem{Ndiff} E. Nart, \emph{Local computation of differents and discriminants},  arXiv: 1205.1340v1 [math.NT], to appear in Mathematics of Computation.

\bibitem{Ok}
K. Okutsu, \emph{Construction of integral basis, I, II}, Proceedings of the Japan Academy, Ser. A {\bf 58} (1982), 47--49, 87--89.

\bibitem{ore1} \O{}. Ore, \emph{Zur Theorie der algebraischen K\"orper}, Acta Mathematica {\bf44} (1923), pp. 219--314.


\bibitem{ore2} \O{}. Ore, \emph{Newtonsche Polygone in der Theorie der algebraischen K\"orper}, Mathematische Annalen {\bf99} (1928), pp. 84--117.

\bibitem{PP}
L. Popescu, N. Popescu,  \emph{On the residual transcendental extensions of a valuation. Key polynomials and augmented valuation}, Tsukuba Journal of Mathematics {\bf 15} (1991), 57--78.

\bibitem{PZ}
N. Popescu, A. Zaharescu, \emph{On the structure of the irreducible polynomials over local fields}, Journal of Number Theory {\bf 52} (1995), 98--118.

\bibitem{PV}
N. Popescu, C. Vraciu, \emph{On the extension of valuations on a field $K$ to $K(x)$, I, II}, Rendiconti del Seminario Matematico della Universit\`a di Padova  {\bf 87} (1992), 151--168, {\bf 96} (1996), 1--14.


\bibitem{Vaq}
M. Vaqui\'e, \emph{Extension d'une valuation}, Transactions of the American Mathematical Society  {\bf 359} (2007), no. 7, 3439--3481.


\end{thebibliography}
\end{document}